\documentclass[12pt,twoside]{amsart}
\usepackage{a4wide,amsmath,amsbsy,amsfonts,amssymb,mathabx,
stmaryrd,amsthm,mathrsfs,graphicx, amscd, tikz-cd, cancel}
\usepackage[normalem]{ulem}
\usepackage{soul}
\usetikzlibrary{matrix,arrows,decorations.pathmorphing}
\usepackage[active]{srcltx}
\usepackage[
	hypertexnames=false,
	hyperindex,
	pagebackref,  
	pdftex,
	breaklinks=true,
	bookmarks=false,
	colorlinks,
	linkcolor=blue,
	citecolor=red,
	urlcolor=red,
]{hyperref}

\usepackage[all]{xy}
\usepackage{subfig}
\usepackage{tikz}
\usetikzlibrary{shapes,arrows,shadows}
\usetikzlibrary{decorations.markings}
\usepackage{color}
\usepackage[normalem]{ulem}
\usepackage{hyperref}
\usepackage{wrapfig}
\usetikzlibrary{arrows}
\usepackage{pinlabel}
\long\def\symbolfootnote[#1]#2{\begingroup%
\def\thefootnote{\fnsymbol{footnote}}\footnote[#1]{#2}\endgroup}

\newtheorem{innercustomthm}{Theorem}
\newenvironment{customtheorem}[1]
  {\renewcommand\theinnercustomthm{#1}\innercustomthm}
  {\endinnercustomthm}

\newtheorem{theorem}{Theorem}[section]
\newtheorem{corollary}[theorem]{Corollary}
\newtheorem{proposition}[theorem]{Proposition}

\newtheorem{lemma}[theorem]{Lemma}
\newtheorem{conjecture}[theorem]{Conjecture}

\newtheorem{question}[theorem]{Question}

\theoremstyle{definition}
\newtheorem{definition}[theorem]{Definition}   
\newtheorem{remark}[theorem]{Remark}      
\newtheorem{remarks}[theorem]{Remarks}

\DeclareMathSymbol{\Alpha}{\mathalpha}{operators}{"41}
\DeclareMathSymbol{\Beta}{\mathalpha}{operators}{"42}
\DeclareMathSymbol{\Epsilon}{\mathalpha}{operators}{"45}
\DeclareMathSymbol{\Zeta}{\mathalpha}{operators}{"5A}
\DeclareMathSymbol{\Eta}{\mathalpha}{operators}{"48}
\DeclareMathSymbol{\Iota}{\mathalpha}{operators}{"49}
\DeclareMathSymbol{\Kappa}{\mathalpha}{operators}{"4B}
\DeclareMathSymbol{\Mu}{\mathalpha}{operators}{"4D}
\DeclareMathSymbol{\Nu}{\mathalpha}{operators}{"4E}
\DeclareMathSymbol{\Omicron}{\mathalpha}{operators}{"4F}
\DeclareMathSymbol{\Rho}{\mathalpha}{operators}{"50}
\DeclareMathSymbol{\Tau}{\mathalpha}{operators}{"54}
\DeclareMathSymbol{\Chi}{\mathalpha}{operators}{"58}
\DeclareMathSymbol{\omicron}{\mathord}{letters}{"6F}
\newcommand{\R}{\mathbb{R}}
\newcommand{\C}{\mathbb{C}}

\newcommand{\Z}{\mathbb{Z}}
\newcommand{\K}{\mathbb{K}}
\newcommand{\ZZ}{{\widehat{\mathbb Z}}}

\newcommand{\Q}{\mathbb{Q}}
\newcommand{\I}{\mathbb{I}}
\newcommand{\D}{\mathbb{D}}

\newcommand{\ovr}{\overrightarrow}

\renewcommand\Im{\operatorname{Im}}

\def\Aut{\operatorname{Aut}}

\def\Inn{\operatorname{Inn}}
\def\inn{\operatorname{inn}}

\def\Out{\operatorname{Out}}
\def\tr{\operatorname{tr}}

\def\Hom{\operatorname{Hom}}

\def\Link{\operatorname{Link}}

\def\Spec{\operatorname{Spec}}
\def\GT{\widehat{\operatorname{GT}}}
\def\P{{\mathbb P}}

\def\cP{{\mathcal P}}
\def\cS{{\mathcal S}}
\def\cL{{\mathcal L}}
\def\cI{{\mathcal I}}

\def\hL{\widehat{\mathcal L}}

\def\sr{\stackrel}

\def\fT{{\mathfrak T}}

\def\cC{{\mathscr C}}
\def\cF{{\mathcal F}}
\def\hF{{\widehat{\operatorname{F}}}}

\def\cG{{\mathcal G}}

\def\cM{{\mathcal M}}
\def\cP{{\mathcal P}}
\def\hP{{\widehat\Pi}}
\def\hp{{\widehat\pi}}
\def\hI{{\hat{I}}}

\def\ccM{{\overline{\mathcal M}}}

\def\dd{\partial}

\def\l{{\lambda}}

\def\g{{\gamma}}
\def\G{{\Gamma}}
\def\U{{\Upsilon}}
\def\PG{{\mathrm{P}\Gamma}}

\def\kG{{\check\Gamma}}

\def\kPG{{\mathrm{P}\check\Gamma}}
\def\hG{{\widehat\Gamma}}
\def\hU{{\widehat\Upsilon}}

\def\hPG{{\mathrm{P}\widehat\Gamma}}

\def\d{{\delta}}
\def\s{{\sigma}}

\def\u{{\upsilon}}

\def\tS{{\widetilde S}}

\def\dS{{\partial S}}
\def\kC{{\check  C}}
\def\hC{{\widehat{C}}}

\def\ccC{{\overline{\cC}}}

\def\ssm{\smallsetminus}
\def\ol{\overline}
\def\td{\tilde}
\def\wh{\widehat}

\def\hookra{\hookrightarrow}
\def\Ra{\Rightarrow}
\def\co{\colon\thinspace}

\begin{document}

\title{Automorphisms of profinite mapping class groups}

\author[M. Boggi]{Marco Boggi}

\address{UFF - Instituto de Matem\'atica e Estat\'{\i}stica - Niter\'oi - RJ 24210-200, Brazil.}
\email{marco.boggi@gmail.com}

\begin{abstract}    
For $S=S_{g,n}$ a closed orientable differentiable surface of genus $g$ from which $n$ points have been removed,
such that $\chi(S)=2-2g-n<0$, let $\PG(S)$ be the pure mapping class group of $S$ and $\hPG(S)$ and $\kPG(S)$ be, respectively, its profinite
and its congruence completions, the latter being identified with the image of the natural representation $\hPG(S)\to\Out(\hp_1(S))$
(where $\hp_1(S)$ is the profinite completion of the fundamental group of the surface $S$).
We determine the automorphism groups of procongruence completions under a natural rigidity condition, 
and show that the profinite Grothendieck-Teichm\"uller group embeds into the outer automorphism group of the profinite completion.

Let $\Out^{\I_0}(\hPG(S))$ and $\Out^{\I_0}(\kPG(S))$ be the groups of outer automorphisms which preserve 
the conjugacy class of a procyclic subgroup generated by a nonseparating Dehn twist (a condition trivially satisfied for $g=0$).
Our main result gives that, for $\chi(S)<g-2$ and $(g,n)\neq (1,2)$, there is a natural isomorphism:
\[\Out^{\I_0}(\kPG(S))\cong\Sigma_{n}\times\GT,\]
where $\Sigma_{n}$ is the symmetric group on $n$ letters and $\GT$ denotes the profinite Grothendieck-Teichm\"uller group.
We also prove that, for $\chi(S)<g-2$, there is a natural faithful representation $\GT\hookra\Out^{\I_0}(\hPG(S))$.

\end{abstract}


\maketitle

\section{Introduction}
Let $S=S_{g,n}$ be a closed orientable differentiable surface of genus $g(S)=g$ from which $n(S)=n$ points have been removed. 
We denote by $\chi(S)=2-2g-n$ the Euler characteristic of $S$, which we assume to be negative, and by $d(S)$ the \emph{modular dimension} 
of $S$, that is to say the number of moduli of a Riemann surface diffeomorphic to $S$. More precisely, to the topological surface $S$, we associate
the moduli stacks $\cM(S)$, which parameterizes smooth projective curves whose complex models are diffeomorphic to $S$, and  $\mathrm{P}\cM(S)$, 
which parameterizes the same curves but on which an order of the punctures has been fixed.
They are smooth irreducible DM stacks defined over $\Spec(\Z)$ of dimension $d(S)=3g-3+n$. 

We let $\G(S)$ be the mapping class group of $S$, that is to say the group of isotopy classes of orientation preserving diffeomorphisms of $S$,
and $\PG(S)$ the \emph{pure} mapping class group of $S$, that is to say the kernel of the natural representation $\G(S)\to\Sigma_{n(S)}$,
where $\Sigma_{n(S)}$ denotes the symmetric group on the punctures of $S$. For an algebraic stack $X$ over $\Spec(\Z)$ and a field $\K$, 
let $X_\K:=X\times\Spec(\K)$. The groups $\G(S)$ and $\PG(S)$ can then be identified, respectively, with the 
topological fundamental groups of the complex DM stacks $\cM(S)_\C$ and $\mathrm{P}\cM(S)_\C$.

A classical rigidity result of Ivanov states that, for $g\geq 3$, all automorphisms of $\G(S)$ are induced by the inner automorphisms of the
extended mapping class group, so that the outer automorphism group of $\G(S)$ is a cyclic group of order $2$ (cf.\ \cite{[I2]}).
McCarthy later computed this group also in genus $\leq 2$ (cf.\ \cite{[McC]}). If we let $\Aut^{\I_0}(\G(S))$ be the group of automorphisms of $\G(S)$
which preserve the conjugacy class of a cyclic subgroup generated by a nonseparating Dehn twist and $\Out^{\I_0}(\G(S))$ be its quotient by the 
group of inner automorphisms, all the previous results can be stated in a uniform way by saying that there are, for $d(S)>1$, a natural isomorphism:
\[\Out^{\I_0}(\G(S))\cong\{\pm 1\};\] 
and, for $d(S)>1$ and $S\neq S_{1,2}$, a natural isomorphism:
\[\Out^{\I_0}(\PG(S))\cong\Sigma_{n(S)}\times\{\pm 1\},\]
where the symmetric group $\Sigma_{n(S)}$ identifies with $\Inn(\G(S))/\Inn(\PG(S))$ and $\{\pm 1\}$ with its centralizer in $\Out^{\I_0}(\PG(S))$. 

Note that, for $d(S)>1$ and $Z(\G(S))=\{1\}$, namely for $(g,n)\neq (0,4), (1,1),(1,2)$ and $(2,0)$, the $\I_0$-condition is 
always satisfied. However, for $Z(\G(S))\neq\{1\}$, for instance, we have that $\Aut^{\I_0}(\G(S))\subsetneq\Aut(\G(S))$ 
(cf.\ \cite[Theorem~1, (3)]{[McC]}). 

In this paper, we determine the automorphism groups of the congruence completions of mapping class groups.

The \emph{profinite mapping class group} $\hG(S)$ and the \emph{pure profinite mapping class group} $\hPG(S)$ are, respectively, 
the profinite completions of $\G(S)$ and $\PG(S)$.
The \emph{procongruence mapping class group} $\kG(S)$ and the \emph{pure procongruence mapping class group} $\kPG(S)$ 
are, respectively, the images of $\hG(S)$ and $\hPG(S)$ in the profinite group $\Out(\hp_1(S))$. 

The \emph{congruence subgroup problem} for mapping class groups asks whether the natural epimorphism $\hG(S)\to\kG(S)$, or equivalently 
$\hPG(S)\to\kPG(S)$, is an isomorphism. This is known to be true only for $g(S)\leq 2$ (cf.\ \cite{hyp} and \cite{congtop}, for instance). 

Let $\cC(S)\to\cM(S)$ and $\mathrm{P}\cC(S)\to\mathrm{P}\cM(S)$ be the universal punctured curves. The profinite mapping class groups 
$\hG(S)$ and $\hPG(S)$ then identify, respectively, with the \'etale fundamental groups of $\cM(S)_{\ol{\Q}}$ and $\mathrm{P}\cM(S)_{\ol{\Q}}$, 
while the procongruence mapping class groups $\kG(S)$ and $\kPG(S)$ identify, respectively, with the images of the universal monodromy 
representations associated to the curves $\cC(S)_{\ol{\Q}}\to\cM(S)_{\ol{\Q}}$ and $\mathrm{P}\cC(S)_{\ol{\Q}}\to\mathrm{P}\cM(S)_{\ol{\Q}}$.

By a classical result of Grossman the natural homomorphism $\G(S)\to\kG(S)$ is injective and so, in particular, we also have an 
embedding $\G(S)\hookra\hG(S)$. We then identify $\G(S)$ with its image in $\kG(S)$ (resp.\ $\hG(S)$). Hence, we associate to
a simple closed curve $\g$ on $S$ the Dehn twist $\tau_\g\in\G(S)\subset\kG(S)$ (resp.\ $\subset\hG(S)$). The set of
\emph{profinite Dehn twists} in $\kG(S)$ (resp.\ $\hG(S)$) is the closure of the set of Dehn twists. 

A basic advantage of dealing with procongruence mapping class groups is that they have a well developed combinatorial theory 
(cf.\ \cite{[B3]}, \cite{[BZ2]} and \cite{congtop}). In fact, for the procongruence mapping 
class group $\kG(S)$, there is a more intrinsic description of profinite Dehn twists in terms of \emph{profinite simple closed curves} on $S$
(cf.\ Section~\ref{multicurves} and \cite[Section~4]{[B3]}).

A fundamental feature of mapping class groups is that, when the center is trivial, automorphisms preserve the conjugacy classes of Dehn twists.
In the profinite setting, essentially thanks to a result by Hoshi, Minamide and Mochizuki (cf.\ \cite[Corollary~C]{HMM}), this is known to be true in genus $0$ 
(cf.\ \cite[Theorem~6.6]{BF2}) but remains a major open problem in positive genus.

It is therefore natural to restrict to automorphisms preserving the conjugacy classes of profinite Dehn twists (the so called \emph{inertia condition}).
A key result of this paper is that it is in fact sufficient to impose the inertia condition only on nonseparating Dehn twists (cf.\ Theorem~\ref{IvsIns}). 

Under this natural rigidity condition, we obtain a complete classification of the automorphism groups of congruence completions of mapping class groups
in terms of the profinite \emph{Grothendieck-Teichm\"uller group} $\GT$ introduced by Drinfeld in \cite{Drinfeld}.

To be more precise, let $\Aut^{\I_0}(\kG(S))$ (resp.\ $\Aut^{\I_0}(\kPG(S))$)  be the group of automorphisms of $\kG(S)$ (resp.\ $\kPG(S)$) which preserve 
the conjugacy class of a procyclic subgroup generated by a nonseparating Dehn twist and let $\Out^{\I_0}(\kG(S))$, $\Out^{\I_0}(\kPG(S))$
be the corresponding outer automorphism groups. Our main result provides a complete classification of these automorphism groups 
(Theorem~\ref{GT=Out} is in fact somewhat more general):

\begin{customtheorem}{A}\begin{enumerate}
\item For $d(S)>1$, there is a natural isomorphism: 
\[\Out^{\I_0}(\kG(S))\cong\GT.\]
\item For $d(S)>1$ and $S\neq S_{1,2}$, there is a natural isomorphism: 
\[\Out^{\I_0}(\kPG(S))\cong\Sigma_{n(S)}\times\GT,\]
where the symmetric group $\Sigma_{n(S)}$ identifies with $\Inn(\kG(S))/\Inn(\kPG(S))$ and $\GT$
with the centralizer of this subgroup.
\end{enumerate}
\end{customtheorem}

\begin{remark}\label{credits}
The genus $0$ case of Theorem~A was essentially proved by Harbater and Schneps in \cite{HS} and Hoshi, Minamide and Mochizuki in \cite{HMM} 
(cf.\ Proposition~\ref{purevsfullgenus0}). The case $(g,n)=(1,2)$ by Minamide and Nakamura in \cite{MN} (cf.\ Proposition~\ref{g=1n=2}).
\end{remark}

The proof of Theorem~A consists of two steps. In the first one, for $S$ and $S'$ hyperbolic surfaces such that 
$g(S)\geq g(S')$ and $\chi(S)\leq\chi(S')<0$, we construct a natural homomorphism
\[\mu_{S,S'}\co\Out^{\I_0}(\kG(S))\to\Out^{\I_0}(\kG(S'))\] 
and we show that, essentially as a consequence of the theory developed in \cite{BF} and \cite{BF2}, $\mu_{S,S'}$ is injective. 
In particular, by the genus $0$ case of the theorem, we get a natural monomorphism:
\[\mu_S\co\Out^{\I_0}(\kG(S))\hookra\GT.\]

The second step then consists in showing that there is a natural (i.e.\ compatible with the homomorphism
$\mu_S$) homomorphism $\Psi_S\co\GT\to\Out^{\I_0}(\kG(S))$. In order to construct this homomorphism, we first prove what is, in its own right, 
one of the main results of the paper (cf.\ Theorem~\ref{GTrepr} for a stronger statement):

\begin{customtheorem}{B}For a hyperbolic surface $S$ such that $d(S)> 1$, there is a natural faithful representation:
\[\hat{\rho}_{\GT}\co\GT\hookra\Out^{\I_0}(\hPG(S)),\]
where $\Out^{\I_0}(\hPG(S))$ is the group of outer automorphisms of $\hPG(S)$ which preserve the conjugacy class of
the procyclic subgroup generated by a nonseparating Dehn twist.
\end{customtheorem}

The proof of Theorem~B can also be split in two steps. In the first one, we use the theory of profinite hyperelliptic mapping class groups to
reduce the cases of genus $\leq 2$ to those of genus $0$ (cf.\ Section~\ref{hypsection} to Section~\ref{proofcomplete}). In the second step, 
thanks to a profinite version (cf.\ Corollary~\ref{proamalgamatedproduct}) of a presentation given by Gervais in \cite{Gervais} for pure mapping 
class groups, we are able to extend the $\GT$-action on profinite pure mapping class groups from genus $\leq 2$ to higher genus.
The representation $\Psi_S$ is obtained observing that the representation $\hat{\rho}_{\GT}$ preserves the congruence kernel
and then composing with the natural homomorphism $\Out^{\I_0}(\kPG(S))\to\Out^{\I_0}(\kG(S))$ obtained in Theorem~\ref{compurevsfull}.

A preliminary announcement of related results appeared in \cite{LNS} (cf.\ (ii) Th\'eor\`eme Principal), with a more complete treatment given recently 
in \cite{LNS2}. Previous versions of the present work predate these new developments.


\subsection{Applications to Grothendieck-Teichm\"uller theory}
A primary source of interest in the outer automorphism groups of profinite mapping class groups comes from the observation, made by Grothendieck 
in \cite[\emph{Esquisse d'un Programme}]{Esquisse}, that, by Bely\v{\i} theorem, they contain a copy of the absolute Galois group of the rationals $G_\Q$. 

Grothendieck suggested that the \'etale fundamental groups of the moduli stacks $\mathrm{P}\cM(S)$ should 
be assembled together in what he called the \emph{Teichm\"uller tower} and that the automorphism group of this tower should be already determined 
by its truncation at modular dimension $2$. He did not give many details on how to proceed in what was and remained the sketch of a research plan. 
It was not even clear what should be done with this automorphism group but a reasonable guess is that he hoped that it would provide a combinatorial 
description of the absolute Galois group $G_\Q$ (cf.\ Conjecture~\ref{Grothendieck}).

In the paper \cite{Drinfeld}, Drinfeld initiated the study of the genus $0$ stage of the Grothendieck-Teichm\"uller tower and, to this end, he introduced the 
Grothendieck-Teichm\"uller group $\GT$. From the point of view which is relevant here, this can be described as the group of automorphisms 
of the first two levels (modular dimensions $1$ and $2$) of the genus $0$ stage of the Teichm\"uller tower which respect a suitable inertia condition
(cf.\ \cite{HS}).

Drinfeld then speculated, following Grothendieck, that $\GT$ should be the automorphism group of the Teichm\"uller tower sketched by Grothendieck but 
he did not give many details on how the tower should, in a precise way, be defined and such statement proved. 
In this paper, we will adopt the definition of the Teichm\"uller tower (cf.\ Section~\ref{deftower}) proposed by Hatcher, Lochak and Schneps in \cite{HLS}.

The study of the genus $0$ case was essentially completed by Harbater and Schneps (cf.\ \cite[Main Theorem]{HS}). 
As already mentioned above, more recently, Hoshi, Minamide and Mochizuki have improved their result by showing that the inertia condition
is in fact automatically satisfied (cf.\ \cite[Corollary~C]{HMM}). 

Theorem~A and B apply, in particular, to the study of the Grothendieck-Teichm\"uller tower. For the precise definition of the profinite and the 
procongruence Grothendieck-Teichm\"uller towers $\hat\fT^\mathrm{out}$ and $\check\fT^\mathrm{out}$, we refer the reader to 
Section~\ref{deftower}. Here it suffices to say that, in order to get enough morphisms between objects, it is necessary to 
consider also relative mapping class groups of surfaces with boundary. An almost immediate consequence of the 
general version of Theorem~A (cf.\ Theorem~\ref{GT=Out}) is then the procongruence version of the Drinfeld-Grothendieck conjecture:

\begin{customtheorem}{C}There is a natural isomorphism $\GT\cong\Aut(\check\fT^\mathrm{out})$.
\end{customtheorem}

From the stronger version of Theorem~B which we prove (cf.\ Theorem~\ref{GTrepr}), it also follows that there is a natural faithful representation
(cf.\ Corollary~\ref{GTaction}):
\[\hat\rho_{\GT}\co\GT\hookra\Aut(\hat\fT^\mathrm{out}).\]
This last result shows, in particular, that, even if the congruence subgroup problem had a negative answer, there are no extra
restrictions on Galois actions coming from the profinite rather than the procongruence Grothendieck-Teichm\"uller tower.


\section{Notations, definitions and preliminary results}
\subsection{Surfaces}\label{surfaces}
In this paper, a hyperbolic surface is a connected orientable differentiable surface of negative Euler characteristic. We denote by 
$S=S_{g,n}^k$ a closed oriented surface of genus $g(S)=g$ from which $n(S)=n$ points and $k(S)=k$ open discs have been removed.
The \emph{modular dimension} of $S$ is defined to be $d(S):=3g-3+n+k$.

For $k=0$ (resp.\ $n=0$), we let $S_{g,n}:=S_{g,n}^0$ (resp.\ $S_g^k:=S_{g,0}^k$). Also, we let $S_g:=S_{g,0}$.
The boundary of $S$ is denoted by $\dS$ and we let $\ring{S}:=S\ssm\dS$, so that $\ring{S}\cong S_{g,n+k}$.

Let then $\cS$ be the category with objects hyperbolic surfaces and maps embeddings of surfaces $S\hookra S'$ 
such that a $1$-punctured disc of $S$ is mapped either to another $1$-punctured disc or to a closed disc.

\subsection{Profinite and procongruence mapping class groups}
For $S=S_{g,n}^k$ a hyperbolic surface, we let $\G(S)$ and $\PG(S)$ be respectively the mapping class group and 
the pure mapping class group associated to $S$. Note that the inclusion $\ring{S}\hookra S$ induces, by restriction, 
a natural homomorphism $\G(S)\hookra\G(\ring{S})$ whose image is the stabilizer of the partition of the punctures of $\ring{S}$ 
into those which have a boundary in $S$ and those which do not.

For $S=S_{g,n}$, sometimes, to simplify notation, 
we will denote the corresponding mapping and pure mapping class groups by $\G_{g,[n]}$ and $\G_{g,n}$, respectively. 
These groups are then related by the short exact sequence $1\to\PG(S)\to\G(S)\to\Sigma_n\to 1$, where $\Sigma_n$ is the symmetric group on $n$ letters. 

The \emph{profinite and the pure profinite mapping class groups} $\hG(S)$ and $\hPG(S)$ are defined to be the profinite completions of the groups 
$\G(S)$ and $\PG(S)$, respectively. 

Let $\Pi$ be the fundamental group of the
surface $S$ with respect to some base point and let $\hP$ be its profinite completion. Since $\Pi$ is conjugacy separable and class automorphisms of 
$\Pi$ are inner, the group $\Out(\Pi)$ naturally embeds in $\Out(\hP)$. Moreover, since (topologically) finitely generated profinite groups are
strongly complete, $\Out(\hP)$ is a profinite group.
 
Therefore, both groups $\G(S)$ and $\PG(S)$ naturally embed in the profinite group $\Out(\hP)$ and inherit a profinite topology which is called
the \emph{congruence} topology. The completions of $\G(S)$ and $\PG(S)$ with respect to these topologies are denoted by $\kG(S)$ and $\kPG(S)$ and 
called the \emph{procongruence and the pure procongruence mapping class group}, respectively. For surfaces with empty boundary, these groups 
are studied in detail in \cite{[B3]} and \cite{congtop}. In what follows, we will recall the basic properties we need and prove a few more for surfaces 
with boundary.

\subsection{The congruence subgroup problem}\label{congsubprop}
The \emph{congruence subgroup problem} asks if there is an isomorphism $\hG(S)\cong\kG(S)$ (equiv.\ $\hPG(S)\cong\kPG(S)$).
This is known to be true only for $g(S)\leq 2$ (cf.\ \cite{Asada}, \cite{hyp} and \cite{congtop}). In the sequel, we will identify the profinite 
and the procongruence mapping class groups for $g(S)\leq 2$.

\subsection{Profinite Dehn twists and profinite braid twists}  
Let $\cL(S)$ be the set of isotopy classes of simple closed curves on $S$. We then denote by $\cL(S)_0$ the subset
of $\cL(S)$ consisting of nonperipheral curves. The set $\cL(S)_0$ parameterizes the set of \emph{Dehn twists} in $\G(S)$ 
(cf.\ \cite[Section~3.1.1]{FM} for the definition). For $\g\in\cL(S)_0$, we denote by $\tau_\g\in\PG(S)\subseteq\G(S)$ the associated Dehn twist.

Let us denote by $\cL^b(S)$ the subset of $\cL(S)$ consisting of simple closed curves bounding a $2$-punctured disc on $S$.
For $\g\in\cL^b(S)$, let $D$ be the disc bounded by $\g$. Then, the \emph{braid twist} $b_\g$ about $\g$ is the 
isotopy class of a self-homeomorphism of $S$ which is the identity outside of $D$, swaps the punctures of $B$ and satisfies 
the identity $b_\g^2=\tau_\g$.

The mapping class groups $\PG(S)$ and $\G(S)$ are generated, respectively, by Dehn twists and by Dehn twists and braid twists 
(cf.\ \cite[Corollary~4.15]{FM}).

A \emph{profinite Dehn twist} (resp.\ \emph{profinite braid twist}) of $\kG(S)$ is an element which lies in the closure of the set of Dehn twists
(resp.\ braid twists) inside the profinite group $\kG(S)$ (where we identify $\G(S)$ with a subgroup of the latter group). 

It is a remarkable and nontrivial fact that the profinite Dehn twists of $\kG(S)$ are parameterized by the profinite set of 
\emph{nonperipheral profinite simple closed curves} $\hL(S)_0$ (cf.\ \cite[Theorem~5.1]{[B3]}).
This set is realized, roughly speaking, as the closure of $\cL(S)_0$ inside the profinite set of "unoriented" conjugacy
classes in $\hP$ (cf.\ \cite[Section~4]{[B3]} and \cite[Section~3.2]{congtop} for the precise definition). 
As above, we then denote by $\tau_\g\in\kPG(S)$, for $\g\in\hL(S)_0$, the associated profinite Dehn twist.

Let $\hL^b(S)$ be the subset of $\hL(S)$ consisting of profinite simple closed curves whose topological type (cf.\ \cite[Definition~4.7]{BF}) 
is that of a simple closed curve bounding a $2$-punctured disc. Then, the the profinite braid twists of the procongruence mapping class group 
$\kG(S)$ are parameterized by this profinite set and are denoted by $b_\g$, for $\g\in\hL^b(S)$.

\subsection{Multicurves and their stabilizers}\label{stabmulticurves}
A \emph{multicurve} $\s$ on $S$ is a set of isotopy classes of disjoint simple closed curves on $S$ such that $S\ssm\s$ does not contain 
a disc or a one-punctured disc. The \emph{topological type} of $\s$ is the topological type of the surface $S\ssm\s$.
Multicurves parameterize sets of commuting Dehn twists or braid twists in $\G(S)$. The complex of curves $C(S)$ is the 
abstract simplicial complex whose simplices are multicurves on $S$. Its combinatorial dimension is $d(S)-1$. 

\begin{definition}\label{discreteinertia}For $\s\in C(S)$, we denote by $I_\s$ the abelian subgroup of $\G(S)$ generated by 
the Dehn twists $\tau_\g$, for $\g\in\s$. For $H$ a subgroup of $\G(S)$, we then let $I_\s(H):=I_\s\cap H$. 
The \emph{topological type} of $I_\s$ is the topological type of $\s$.
\end{definition}

There is a natural simplicial action of $\G(S)$ on $C(S)$. Let us describe the stabilizer of a simplex $\s\in C(S)$. 
Let $S\ssm\s=S_1\coprod\ldots\coprod S_k$, let $B_i$ be the set of punctures of $S_i$ bounded by a curve in $\s$, for $i=1,\ldots,k$, 
and let $\Sigma_{\s^\pm}$ be the symmetric group on the set of oriented circles $\s^\pm:=\s^+\cup\s^-$. Let us also denote by 
$\G(S_i)_{B_i}$ the pointwise stabilizer of the subset of punctures $B_i$ for the action of the mapping class group $\G(S_i)$, for $i=1,\ldots,k$.
Then, the stabilizer $\G(S)_\s$ is described by the two exact sequences:
\[\begin{array}{c}
1\to\G(S)_{\vec\s}\to\G(S)_\s\to\Sigma_{\s^\pm},\\
\\
1\to I_\s\to\G(S)_{\vec\s}\to\G(S_1)_{B_1}\times\ldots\times\G(S_k)_{B_k}\to 1.
\end{array}\]
 
A similar description but simpler applies to the stabilizer $\PG(S)_\s$ for the action of the pure mapping class group $\PG(S)$ on $C(S)$:
\[\begin{array}{c}
1\to\PG(S)_{\vec\s}\to\PG(S)_\s\to\Sigma_{\s^\pm},\\
\\
1\to I_\s\to\PG(S)_{\vec\s}\to\PG(S_1)\times\ldots\times\PG(S_k)\to 1.
\end{array}\]

\subsection{The procongruence curve complex}\label{multicurves}
A central object for the study of procongruence mapping class groups is a profinite version of the complex of curves. For every $k\geq 0$, 
there is a natural embedding $C(S)_k\hookra\cP_{k+1}(\hL(S))$, where $\cP_{k+1}(\hL(S))$ is
the profinite set of unordered subsets of $k+1$ distinct elements in $\hL(S)$. We then define:

\begin{definition}\label{cpc}The \emph{procongruence curve complex} (or \emph{complex of profinite curves}) $\kC(S)$ (cf.\ \cite[Section~4]{[B3]}, 
for more details) is the abstract simplicial profinite complex whose set $\kC(S)_k$ of $k$-simplices is the closure of $C(S)_k$ inside $\cP_{k+1}(\hL(S))$. 
In particular, $\kC(S)_0=\hL(S)_0$. A simplex $\s$ of $\kC(S)$ is also called a \emph{profinite multicurve}. Its \emph{topological type} 
is the topological type of a simplex in the intersection of the orbit $\kG(S)\cdot\s$ with $C(S)\subset\kC(S)$.
\end{definition}

There are continuous natural actions of the procongruence mapping class groups $\kG(S)$ and $\kPG(S)$ on $\kC(S)$.
 \cite[Theorem~4.5]{[B3]} describes the stabilizers for the action of $\kPG(S)$.
This result implies a similar description for the action of $\kG(S)$ (cf.\ \cite[Theorem~4.10]{BF}):

\begin{theorem}\label{stabilizers}For $\s\in C(S)\subset\kC(S)$, let $S\ssm\s=S_1\coprod\ldots\coprod S_k$
and $B_i$ be the set of punctures of $S_i$ bounded by a curve in $\s$, for $i=1,\ldots,k$. Then, we have:
\begin{enumerate}
\item The stabilizer $\kG(S)_\s$ is described by the two exact sequences:
\[\begin{array}{c}
1\to\kG(S)_{\vec\s}\to\kG(S)_\s\to\Sigma_{\s^\pm},\\
\\
1\to \hI_\s\to\kG(S)_{\vec\s}\to\kG(S_1)_{B_1}\times\ldots\times\kG(S_k)_{B_k}\to 1,
\end{array}\]
where $ \hI_\s$ is the profinite completion of the free abelian group $I_\s$ and $\kG(S_i)_{B_i}$ is the pointwise stabilizer 
of the subset of punctures $B_i$ in $\kG(S_i)$, for $i=1,\ldots,k$.

\item The stabilizer $\kPG(S)_\s$ is described by the two exact sequences:
\[\begin{array}{c}
1\to\kPG(S)_{\vec\s}\to\kPG(S)_\s\to\Sigma_{\s^\pm},\\
\\
1\to \hI_\s\to\kPG(S)_{\vec\s}\to\kPG(S_1)\times\ldots\times\kPG(S_k)\to 1.
\end{array}\]
\end{enumerate}
\end{theorem}

\subsection{Profinite and procongruence relative mapping class groups}\label{MCGboundary}
For a hyperbolic surface $S=S_{g,n}^k$, we let $\G(S,\dS)$ be the \emph{relative mapping class group of the surface with boundary $(S,\dS)$}, 
that is to say the group of relative, with respect to the boundary $\dS$, isotopy classes of orientation preserving diffeomorphisms of $S$.
Let $\delta_1,\ldots,\delta_k$ be the connected components of the boundary $\dS$. Then, 
the group $\G(S,\dS)$ is described by the natural short exact sequence:
\[1\to\prod_{i=1}^k\tau_{\d_i}^\Z\to\G(S,\dS)\to\G(S)\to 1.\]

The relative \emph{pure} mapping class group $\PG(S,\dS)$ is the subgroup of $\G(S,\dS)$ consisting of those elements which admit representatives 
fixing pointwise the punctures and the boundary of $S$. It is described by the short exact sequence:
\[1\to\PG(S,\dS)\to\G(S,\dS)\to\Sigma_n\times\Sigma_k\to 1.\]

\begin{remark}\label{identify}
Let $\tS=S_{g,n+2k}$, with $d(\tS)>1$ and $\s=\{\delta_1,\ldots,\delta_k\}\in C(\tS)$, where $\delta_i$ is a simple closed curve bounding 
an open disc $D_i$ containing the punctures indexed by $n+2i-1$ and $n+2i$, for $i=1,\ldots,k$. 
Let $S:=\tS\ssm\cup_{i=1}^k D_i\cong S_{g,n}^k$. The relative mapping class group $\G(S,\dS)$ then identifies with a subgroup
of the stabilizer $\G(\tS)_\s$, which is described by the short exact sequence:
\[1\to \prod_{i=1}^k\tau_{\d_i}^\Z\to\G(\tS)_\s\to(\Sigma_2)^k\rtimes\G(S)\to 1,\]
where $\G(S)$ acts on the product $(\Sigma_2)^k$ by permuting its factors through the natural representation $\G(S)\to\Sigma_k$.
There is also a natural isomorphism $\PG(S,\dS)\cong\PG(\tS)_{\vec\s}$. 
\end{remark}

The \emph{profinite (resp.\ profinite pure) relative mapping class group of the surface with boundary $(S,\dS)$} is defined to be the profinite completion 
$\hG(S,\dS)$ (resp.\ $\hPG(S,\dS)$) of $\G(S,\dS)$ (resp.\ of $\PG(S,\dS)$). From (ii) of Theorem~\ref{stabilizers} and Remark~\ref{identify},
it follows that these profinite groups are described by the short exact sequences:
\begin{equation}\label{centralext1}
\begin{array}{ll}
&1\to\prod_{i=1}^k\tau_{\d_i}^\ZZ\to\hG(S,\dS)\to\hG(S)\to 1\\
\mbox{and }&\\
&1\to\hPG(S,\dS)\to\hG(S,\dS)\to\Sigma_n\times\Sigma_k\to 1.
\end{array}
\end{equation}

\begin{definition}\label{tildedS}For a surface $S=S_{g,n}^k$ with boundary $\dS=\cup_{i=1}^k\delta_i$, we let $\tS\cong S_{g,n+2k}$ 
be the surface obtained from $S$ glueing a $2$-punctured disc over each boundary component $\delta_i$, for $i=1,\ldots,k$. 
\end{definition}

By Remark~\ref{identify}, the embedding $S\hookra \tS$ induces a monomorphism of mapping class groups 
$\G(S,\dS)\hookra\G(\tS)$, which identifies $\G(S,\dS)$ with a subgroup of the stabilizer $\G(\tS)_\s$. 

The congruence topology on $\G(\tS)$ then induces a profinite topology on $\G(S,\dS)$ which we call the
\emph{congruence topology on }$\G(S,\dS)$. This clearly also defines a profinite topology on $\PG(S,\dS)\subseteq\G(S,\dS)$, 
which we call the congruence topology on $\PG(S,\dS)$. 

The completions of $\G(S,\dS)$ and $\PG(S,\dS)$ with respect to these congruence topologies are then denoted by $\kG(S,\dS)$ 
and $\kPG(S,\dS)$, respectively, and are called the \emph{procongruence (resp.\ pure) relative mapping class group 
of the surface with boundary $(S,\dS)$}. 

By Theorem~\ref{stabilizers}, they fit in the short exact sequences:
\begin{equation}\label{centralext}
\begin{array}{ll}
&1\to\prod_{i=1}^k\tau_{\d_i}^\ZZ\to\kG(S,\dS)\to\kG(S)\to 1\\
\mbox{and }&\\
&1\to\kPG(S,\dS)\to\kG(S,\dS)\to\Sigma_n\times\Sigma_k\to 1.
\end{array}
\end{equation}
 
An immediate consequence is then:

\begin{proposition}\label{congcent}\leavevmode
\begin{enumerate}
\item For $g(S)\leq 2$, we have $\kG(S,\dS)=\hG(S,\dS)$ and $\kPG(S,\dS)=\hPG(S,\dS)$.
\item For $S\neq S_2, S_{1,2}$ and $S_1^2$, we have $Z(\kG(S,\dS))=(\prod_{i=1}^k\tau_{\d_i})^\ZZ$. Otherwise, the center $Z(\kG(S,\dS))$ 
is generated by $(\prod_{i=1}^k\tau_{\d_i})^\ZZ$ and the hyperelliptic involution.
\item  For $S\neq S_2, S_{1,2},S_1^2$ and $U$ an open subgroup of $\kPG(S,\dS)$, we have $Z_{\kG(S,\dS)}(U)=\prod_{i=1}^k\tau_{\d_i}^\ZZ$. 
Otherwise, $Z_{\kG(S,\dS)}(U)$ is generated by $\prod_{i=1}^k\tau_{\d_i}^\ZZ$ and the hyperelliptic involution.
\end{enumerate}
\end{proposition}

\begin{proof}The first item follows from the congruence subgroup property in genus $\leq 2$ for surfaces without boundary and the short
exact sequences~(\ref{centralext1}) and~(\ref{centralext}). The second and the third item from the same exact sequences, 
 \cite[Corollary~6.2]{[B3]} and \cite[Theorem~4.14]{BF}.
\end{proof}

The procongruence (resp.\ pure) relative mapping class group $\kG(S,\dS)$ (resp.\ $\kPG(S,\dS)$) acts on the procongruence curve complex
$\kC(S)$ through the natural homomorphism to $\kG(S)$ (resp.\ $\kPG(S)$). For $\s\in C(S)$, let $I_\s$ be the subgroup of $\G(S,\dS)$ generated by
the Dehn twists $\tau_\g$, for $\g\in\s$ (cf.\ Definition~\ref{discreteinertia}). Then, from Theorem~\ref{stabilizers}, it easily follows:

\begin{theorem}\label{relativestabilizers}For $\s\in C(S)\subset\kC(S)$, let $S\ssm\s=S_1\coprod\ldots\coprod S_k$
and $B_i$ be the set of punctures of $S_i$ bounded by a curve in $\s$, for $i=1,\ldots,k$. Then, we have:
\begin{enumerate}
\item The stabilizer $\kG(S,\dS)_\s$ is described by the two exact sequences:
\[\begin{array}{c}
1\to\kG(S,\dS)_{\vec\s}\to\kG(S,\dS)_\s\to\Sigma_{\s^\pm},\\
\\
1\to \hI_\s\to\kG(S,\dS)_{\vec\s}\to\kG(S_1,\dS_1)_{B_1}\times\ldots\times\kG(S_k,,\dS_k)_{B_k}\to 1,
\end{array}\]
where $ \hI_\s$ is the profinite completion of the free abelian group $I_\s$ and $\kG(S_i,,\dS_i)_{B_i}$ is the pointwise stabilizer 
of the subset of punctures $B_i$ in $\kG(S_i,,\dS_i)$, for $i=1,\ldots,k$.

\item The stabilizer $\kPG(S,\dS)_\s$ is described by the two exact sequences:
\[\begin{array}{c}
1\to\kPG(S,\dS)_{\vec\s}\to\kPG(S,\dS)_\s\to\Sigma_{\s^\pm},\\
\\
1\to \hI_\s\to\kPG(S,\dS)_{\vec\s}\to\kPG(S_1,\dS_1)\times\ldots\times\kPG(S_k,\dS_k)\to 1.
\end{array}\]
\end{enumerate}
\end{theorem}

From the definition of the congruence topology and \cite[Corollary~4.12]{BF}, it then follows
a description of the centralizers of profinite multitwists in procongruence relative mapping class groups:

\begin{theorem}\label{normalizers multitwists}For $\s=\{\g_0,\ldots,\g_k\}$ a profinite multicurve on $S$ and $m\in\Z\ssm\{0\}$, we have:
\[Z_{\kG(S,\dS)}(\tau_{\g_0}^m\cdots\tau_{\g_k}^m)=N_{\kG(S,\dS)}(( \tau_{\g_0}^m\cdots\tau_{\g_k}^m)^\ZZ)
=N_{\kG(S,\dS)}(\langle \tau_{\g_0}^m,\ldots,\tau_{\g_k}^m\rangle)=\kG(S,\dS)_\s,\]
where $\kG(S,\dS)_\s$ is the stabilizer of $\s$ described in (i) of Theorem~\ref{relativestabilizers}. 
A similar result holds for the procongruence pure relative mapping class group $\kPG(S,\dS)$.
\end{theorem}

\subsection{The Grothendieck-Teichm\"uller tower}\label{deftower}
Let $\eta\co S\hookra S'$ be a map in the category of surfaces $\cS$ defined in Section~\ref{surfaces}. 
Then, $\eta$ induces a homomorphism of pure relative mapping class groups $\eta_\ast\co\PG(S,\dS)\to\PG(S',\dS')$
and hence of profinite pure relative mapping class groups $\hat\eta_\ast\co\hPG(S,\dS)\to\hPG(S',\dS')$. 
From Remark~\ref{identify} and (ii) of Theorem~\ref{stabilizers}, it follows that $\hat\eta_\ast$ also induces a homomorphism of
procongruence pure relative mapping class groups $\check\eta_\ast\co\kPG(S,\dS)\to\kPG(S',\dS')$. 

Let $\cG$ be the category of profinite groups and continuous homomorphisms. We have then defined two functors $\hat\fT$ and $\check\fT$ 
from $\cS$ to $\cG$. It is not difficult to see that the functor $\check\fT$ factors through the functor $\hat\fT$.

There is a natural way to realize the above profinite mapping class groups as geometric \'etale fundamental groups of suitable geometric objects.
We will associate to a hyperbolic surface $S$ a logarithmic formal DM stack $\mathrm{P}\cM(S)$ (cf.\ \cite{Vaintrob} for a similar construction).

For a surface $S$ with empty boundary, let $\mathrm{P}\cM(S)$ be the DM stack of smooth curves whose complex model is diffeomorphic to $S$ 
with an order on the set of punctures and let $\mathrm{P}\ccM(S)$ be the DM compactification of $\mathrm{P}\cM(S)$. 
Let then $\mathrm{P}\ccM(S)^{\log}$ be the log DM stack with the logarithmic structure associated to the DM boundary.
There is a natural fully faithful functor from the category of DM stacks to the category of formal DM stacks. We then denote by 
$\mathrm{P}\ccM(S,\emptyset)^{\log}$ the formal log DM stack associated to the log DM stack $\mathrm{P}\ccM(S)^{\log}$ 
and assign this to the surface $S$. There is a series of isomorphisms:
\[\pi_1^\mathrm{et}(\mathrm{P}\ccM(S,\emptyset)^{\log}_{\ol\Q})\cong\pi_1^\mathrm{et}(\mathrm{P}\ccM(S)^{\log}_{\ol\Q})\cong
\pi_1^\mathrm{et}(\mathrm{P}\cM(S)_{\ol\Q})\cong\hPG(S).\]

For a surface $S$ with nonempty boundary, let $\tS\cong S_{g,n+2k}$ be as in Definition~\ref{tildedS}. 
The simplex $\s=\{\delta_1,\ldots,\delta_k\}\in C(\tS)$ determines a closed stratum $\delta_\s$ in the boundary of 
$\mathrm{P}\ccM(\tS)$ and we let $\mathrm{P}\cM(S,\dS)$ be the normalization of a formal neighborhood of 
$\delta_\s$ in $\mathrm{P}\ccM(\tS)$. 

Let us then assign to $S$ the formal logarithmic DM stack $\mathrm{P}\cM(S,\dS)^{\log}$ obtained by pulling back to the formal DM stack 
$\mathrm{P}\cM(S,\dS)$ the logarithmic structure of $\mathrm{P}\ccM(\tS,\emptyset)^{\log}$, via the natural morphism of formal DM stacks 
$\mathrm{P}\cM(S,\dS)\to \mathrm{P}\cM(\tS,\emptyset)$. There is an isomorphism: 
\[\pi_1^\mathrm{et}(\mathrm{P}\ccM(S,\dS)^{\log}_{\ol\Q})\cong\hPG(S,\dS).\]

From the canonical short exact sequence 
\[1\to\pi_1^\mathrm{et}(\mathrm{P}\ccM(S,\dS)^{\log}_{\ol\Q})\to\pi_1^\mathrm{et}(\mathrm{P}\ccM(S,\dS)^{\log}_{\Q})\to G_\Q\to 1,\]
we get a natural outer representation $\hat\rho_{S}\co G_\Q\to\Out(\hPG(S,\dS))$. A (nontrivial) consequence of Bely\v{\i} theorem is that
this representation is faithful.

The Grothendieck-Teichm\"uller functor $\hat\fT$ factors as the composition of the functor $\mathrm{P}\ccM^{\log}_\Q$, which assigns to the surface
$S$ the  log formal DM stack $\mathrm{P}\ccM(S,\dS)^{\log}_\Q$, and the functor which assigns to a log formal DM stack $X^{\log}$ over the rationals
its geometric \'etale fundamental group $\pi_1^\mathrm{et}(X_{\ol\Q}^{\log})$. 

Note that the identification of topological fundamental groups with mapping class groups takes care of base points. 
However, in order to make keeping track of base points totally irrelevant, it is convenient to modify our definition of the functors $\hat\fT$ and $\check\fT$.
Let $\cG^\mathrm{out}$ be the category with objects profinite groups and maps outer continuous homomorphisms. There is a natural functor 
$\cG\to\cG^\mathrm{out}$ and we define $\hat\fT^\mathrm{out}$ and $\check\fT^\mathrm{out}$ to be, respectively, the composition of 
$\hat\fT$ and $\check\fT$ with this functor.

\begin{definition}\label{GTtowers}
The full images in $\cG^\mathrm{out}$ of the functors $\hat\fT^\mathrm{out}$ and $\check\fT^\mathrm{out}$ are  called the 
\emph{(profinite) Grothendieck-Teichm\"uller tower} and the \emph{procongruence Grothendieck-Teichm\"uller tower}, respectively.
\end{definition} 

From the geometric description of the functor $\hat\fT^\mathrm{out}$, it follows that there is a natural faithful representation:
\[\rho_{\hat\fT}\co G_\Q\hookra\Aut(\hat\fT^\mathrm{out}),\]
defined by the assignment $\alpha\mapsto\{\hat\rho_{S}(\alpha)\}_{S\in\cS}$. At the heart of Grothendieck-Teichm\"uller theory,
is the following apocryphal conjecture by Grothendieck:

\begin{conjecture}[Grothendieck]\label{Grothendieck}The representation $\rho_{\hat\fT}$ induces a natural isomorphism: 
\[G_\Q\cong\Aut(\hat\fT^\mathrm{out}).\]
\end{conjecture}

Somewhat in this spirit is the I/OM conjecture, proved by Pop (cf.\ \cite[Theorem]{Pop2}).

\subsection{The Galois action on the procongruence Grothendieck-Teichm\"uller tower}
For $S$ a hyperbolic surface with empty boundary, let
$\mathrm{P}\ccC(S)^{\log}\to \mathrm{P}\ccM(S)^{\log}$ be the universal log stable curve (this is the universal stable curve endowed 
with the logarithmic structure associated to its DM boundary). For $\bar\xi\in \mathrm{P}\ccM(S)^{\log}$ a $\ol\Q$-point, there is an 
associated universal monodromy representation:
\[\rho_{\bar\xi}\co\pi_1^\mathrm{et}(\mathrm{P}\ccM(S)^{\log}_{\ol\Q})\to\Out(\mathrm{P}\ccC(S)^{\log}_{\bar\xi}),\]
where $\mathrm{P}\ccC(S)^{\log}_{\bar\xi}$ is the fiber of the universal curve over $\bar\xi$ endowed with the 
logarithmic structure induced by $\mathrm{P}\ccC(S)^{\log}$.
The isomorphism $\pi_1^\mathrm{et}(\mathrm{P}\ccM(S)^{\log}_{\ol\Q})\cong\hPG(S)$ then induces an isomorphism $\Im\rho_{\bar\xi}\cong\kPG(S)$.

For $\psi\co S\hookra \tS$ the embedding considered in Remark~\ref{identify}, there is an induced homomorphism:
\[\hat\fT(\psi)\co\pi_1^\mathrm{et}(\mathrm{P}\ccM(S,\dS)^{\log}_{\ol\Q})\cong\hPG(S,\dS)\to
\pi_1^\mathrm{et}(\mathrm{P}\ccM(\tS,\emptyset)^{\log}_{\ol\Q})\cong\pi_1^\mathrm{et}(\mathrm{P}\ccM(\tS)^{\log}_{\ol\Q}).\]
We then have $\rho_{\bar\xi}\circ\hat\fT(\psi)(\pi_1^\mathrm{et}(\mathrm{P}\ccM(S,\dS)^{\log}_{\ol\Q}))\cong\kPG(S,\dS)$.

The above procedure defines natural transformations of functors: 
\[\rho\co\hat\fT\Ra\check\fT\hspace{1cm}\mbox{and}\hspace{1cm}\rho^\mathrm{out}\co\hat\fT^\mathrm{out}\Ra\check\fT^\mathrm{out}.\] 

Since the universal monodromy representation $\rho_{\bar\xi}$ is compatible with the Galois actions, the natural transformation $\rho^\mathrm{out}$
is $G_\Q$-equivariant and there is a natural representation (again, Bely\v{\i} theorem implies that it is faithful): 
\[\rho_{\check\fT}\co G_\Q\hookra\Aut(\check\fT^\mathrm{out}).\]

\subsection{Levels and level structures}\label{moduli}
Let $S$ be a hyperbolic surface with empty boundary.
A finite index subgroup $\G^\l$ of $\G(S)$ is called a \emph{level}. A \emph{congruence} level is a level which is open for the congruence
topology on $\G(S)$. The \'etale covering $\cM^\l(S)$ of $\cM(S)$ associated to the level $\G^\l$ is called a \emph{level structure}
over $\cM(S)$. The level is \emph{representable} if the associated DM stack $\cM^\l(S)$ is representable.

In particular, to the pure mapping class group $\PG(S)$ is associated the level structure $\mathrm{P}\cM(S)$. This is the \'etale Galois covering of $\cM(S)$,
with Galois group the symmetric group $\Sigma_n$, which parameterizes smooth projective curves with an order on the set of punctures. 
For $S=S_{g,n}$, in order to simplify notations, the stacks $\cM(S)$ and $\mathrm{P}\cM(S)$ are sometimes simply denoted by $\cM_{g,[n]}$ and 
$\cM_{g,n}$, respectively.

Let $\ccM(S)$ be the DM compactification of the moduli stack $\cM(S)$.
This is a DM stack defined over $\Z$ which parameterizes stable $n$-punctured, genus $g$ curves. 
A \emph{level structure over $\ccM(S)$} is the compactification $\ccM^\l(S)$ of a level structure $\cM^\l(S)$ over $\cM(S)$ obtained taking the 
normalization of the DM stack $\ccM(S)$ in the function field of $\cM^\l(S)$. This is also a DM stack and, if
the level $\G^\l$ is contained in an abelian level of order $\geq 3$, then $\ccM^\l(S)$ is representable. 

\subsection{The procongruence pants graph}\label{pantsgraphdef}
The \emph{Fulton curve} $\cF(S)$ is the $1$-dimensional 
closed substack of the moduli stack $\ccM(S)$ which parameterizes curves with at least $d-1$ singular points. More generally, for a level structure
$\ccM^\l(S)$ over $\ccM(S)$, the \emph{Fulton curve} $\cF^\l(S)$ is the inverse image of $\cF(S)$ via the natural morphism $\ccM^\l(S)\to\ccM(S)$.

Let $\G^\l$ be a representable level of $\G(S)$ contained in an abelian level of order $m\geq 2$. Then, each irreducible component of the Fulton curve 
$\cF^\l(S)_\Q$ is a covering of $\P^1_\Q$ ramified at most over the set $\{0,1,\infty\}$ and the ramification points are contained in the
boundary of $\cF^\l(S)_\Q$ (that is to say the points parameterizing maximally degenerated stable curves). It follows that the analytic complex variety
$\cF(S)_\C^{an}$ admits a natural triangulation with vertices the points parameterizing maximally degenerated stable curves and edges geodesic
segments of length $1$ joining these boundary points. We denote by $C^\l_P(S)$ the $1$-skeleton of this triangulation which we regard as a 
$1$-dimensional abstract simplicial complex.

\begin{definition}\label{pantsgraph}The \emph{procongruence pants graph} $\kC_P(S)$ is the $1$-dimensional abstract simplicial profinite complex obtained 
as the inverse limit of the finite abstract simplicial complexes $C^\l_P(S)$, for $\G^\l$ varying over the set of all congruence levels of $\G(S)$.
\end{definition}

For a more comprehensive treatment of $\kC_P(S)$, we refer to \cite{BF}. Here we only recall the basic properties which we will need later.
The vertices of $\kC_P(S)$ are in natural bijective correspondence with the profinite set $\kC(S)_{d(S)-1}$ of facets of the procongruence curve complex 
and its edges with the profinite set $\kC(S)_{d(S)-2}$ of $(d(S)-2)$-simplices. To the decomposition in irreducible components of the Fulton curves 
$\cF^\l(S)$, for all levels $\G^\l$, corresponds a decomposition of the procongruence pants graph $\kC_P(S)$ in closed subgraphs called 
\emph{profinite Farey subgraphs}. They are parameterized by the profinite set $\kC(S)_{d(S)-2}$. The profinite Farey subgraph associated 
to $\s\in\kC(S)_{d(S)-2}$ is denoted by $\hF_\s(S)$ and its set of vertices consists of the profinite set of facets of $\kC_P(S)$ containing $\s$.

\section{The $\I$-condition and the $\D$-condition on automorphisms}
\subsection{Decomposition and inertia groups}
Let $S$ be a hyperbolic surface. For an open subgroup $U$ of $\kG(S)$, we call the stabilizer $U_\s=U\cap\kG(S)_\s$ 
of a simplex $\s\in\kC(S)$ for the action of $U$ on $\kC(S)$ the \emph{decomposition group} of $U$ associated to $\s$. 

\begin{definition}\label{DecPres}The group of \emph{$\D$-automorphisms} $\Aut^\D(U)$ is the closed subgroup of $\Aut(U)$ consisting of 
automorphisms preserving the set of decomposition groups $\{U_\g\}_{\g\in\hL(S)_0}$.
\end{definition}

\begin{remark}\label{previous}In \cite{BF}, $\Aut^\D(U)$ was denoted by $\Aut^\ast(U)$. According to
 \cite[Proposition~7.2]{BF}, the group $\Aut^\D(U)$ preserve the set of all decomposition groups $\{U_\s\}_{\s\in\kC(S)}$ of $U$.
\end{remark} 

Let us introduce a slightly more restrictive condition, although in a more general context, which we will need later:

\begin{definition}\label{starcondition}\leavevmode\begin{enumerate}
\item For $\s\in\kC(S)$, we denote by $\hI_\s$ the closed abelian subgroup of $\kG(S)$ topologically generated by the Dehn twists $\tau_\g$, 
for $\g\in\s$. For $H$ a closed subgroup of $\kG(S)$, we then let $\hI_\s(H):=\hI_\s\cap H$ (the \emph{inertia group of $H$ associated to $\s$}). 
The \emph{topological type} of $\hI_\s$ is the topological type of $\s$.
\item The group of \emph{$\I$-automorphisms} $\Aut^\I(H)$ is the closed subgroup of $\Aut(H)$ consisting of those elements which preserve 
the set of inertia groups $\{\hI_\g(H)\}_{\g\in\hL(S)_0}$. 
\item We say that a subgroup $K$ of $H$ is $\I$-characteristic if it is preserved by all elements of $\Aut^\I(H)$. 
\end{enumerate}
\end{definition}

\begin{remarks}\label{algreal}Let us observe that:
\begin{enumerate}
\item By \cite[Proposition~6.8]{[B3]}, for every open subgroup $H$ of $\kG(S)$, the assignment $\s\mapsto\hI_\s(H)$ identifies 
the procongruence curve complex $\kC(S)$ with the simplicial profinite complex $\kC_{\cI(H)}(S)$ whose simplices are the inertia groups 
$\{I_\s(H)\}_{\s\in\kC(S)}$. In the sequel, we will identify all these complexes and denote them simply by $\kC(S)$. 

\item For $H$ a closed subgroup of $\kG(S)$, we have $\Inn H\subseteq\Aut^\I(H)$. In particular, an $\I$-characteristic 
subgroup of $H$ is normal. 
\end{enumerate}
\end{remarks}

\begin{proposition}\label{0simplices}We have:
\begin{enumerate}
\item An automorphism of $H$ is an $\I$-automorphism if and only if it preserves the set of all inertia groups 
$\{\hI_\s(H)\}_{\s\in\kC(S)}$ of $H$.
\item For $S\neq S_{1,2}$ and $H$ any open subgroup of $\kG(S)$, the action of $\Aut^\I(H)$ on the set $\{\hI_\s(H)\}_{\s\in\kC(S)}$ preserves topological types. 
\item For $H=\hG(S_{1,2})$ or $\hPG(S_{1,2})$, the action of $\Aut^\I(H)$ on the set $\{\hI_\s(H)\}_{\s\in\kC(S)}$ preserves topological types. 
\item $\Aut^\I(\kG(S))$ preserves the conjugacy classes of profinite braid twists in $\kG(S)$.
\end{enumerate}
\end{proposition}

\begin{proof}(i): One implication is obvious. For the other, let us observe that, 
for $\s=\{\g_0,\ldots,\g_k\}\in\kC(S)$, we have $\hI_\s(H)=\prod_{i=0}^k \hI_{\g_i}(H)$. Moreover, by \cite[Corollary~4.12]{BF}, 
a set of nonperipheral profinite simple closed curves $\{\g_0,\ldots,\g_k\}$ forms a $k$-simplex of $\kC(S)$ if and only if the associated
profinite Dehn twists or, equivalently, some nontrivial powers of them commute. The conclusion is then clear.
\smallskip

\noindent
(ii): Item (i) of Remarks~\ref{algreal} implies that, for every open subgroup $H$ of $\kG(S)$, 
there is a natural continuous homomorphism $\Aut^\I(H)\to\Aut(\kC(S))$. The conclusion then follows from \cite[Theorem~5.5]{BF} 
(this is also where the hypothesis $S\neq S_{1,2}$ is actually needed).
\smallskip

\noindent
(iii): Let us observe that an $\I$-automorphism of $\hG(S_{1,2})$ preserves the subgroup $\hPG(S_{1,2})$, since the latter is (topologically) generated
by Dehn twists. Both statements then follow from \cite[Lemma~2.19]{BF2}.
\smallskip

\noindent
(iv): For $n(S)$ or $k(S)\geq 2$ (otherwise the claim is trivial), let $b_\g\in\kG(S)$ be the profinite braid twist
associated to a simple closed curve $\g$ on $S$ bounding either a $2$-punctured disc or a disc containing two boundary components. 

By the previous item, for a given $f\in\Aut^\I(\kG(S))$, we can assume, after composing with an element of $\Inn\kG(S)$, that $f$ preserves the inertia group 
$\hI_\g=\tau_\g^\ZZ$. In particular, $f$ preserves the normalizer $N_{\kG(S)}(\hI_\g)$ of $\hI_\g$ in $\kG(S)$. 
From the description of this group (cf.\ \cite[Corollary~4.12, (i)]{BF}), it follows that,
for $(g(S),n(S)+k(S))\neq (0,4),(1,2)$, the center of $N_{\kG(S)}(\hI_\g)$ is generated by $b_\g$, which concludes the proof of item (ii) in this case.

For $(g(S),n(S)+k(S))= (0,4)$, a similar argument applies.
For $g(S)=1$ and either $n(S)=2$ or $k(S)=2$, the claim follows from (i) of Proposition~\ref{g=1n=2}.
\end{proof}

\begin{remarks}\label{lessrestrictive}Let $U$ be an open subgroup of $\kG(S)$:
\begin{enumerate}
\item By \cite[Corollary~4.12]{BF}, we have that $\Inn(U)\subseteq\Aut^\I(U)\subseteq\Aut^\D(U)$. 
We denote by $\Out^\I(U)$ (resp.\ $\Out^\D(U)$) the group of outer $\I$-automorphisms (resp.\ $\D$-automorphisms) of $U$.
\item There is a natural faithful Galois representation $\rho_\Q\co G_\Q\hookra\Out^\I(\kG(S))$, where $G_\Q$ is the absolute
Galois group of the rationals (cf.\ \cite[Corollary~7.6]{[B3]}).
\end{enumerate}
\end{remarks}

Except in the low genera cases, the $\D$-condition and the $\I$-condition are equivalent for full and pure procongruence mapping class groups. 
More precisely:

\begin{proposition}\label{comparison}We have:
\begin{enumerate}
\item For $d(S)>1$ and $S\neq S_{1,2}, S_1^2, S_2$, there holds: 
\[\Out^\I(\kG(S))=\Out^\D(\kG(S)).\]
\item For $S\neq S_{1,1}, S_1^1, S_2$, there holds:
\[\Out^\I(\kPG(S))=\Out^\D(\kPG(S)).\]
\item For $S=S_{1,1},S_1^1,S_{1,2},S_1^2$ or $S_2$, there is a natural isomorphism:
\[\Out^\D(\hG(S))\cong\Out^\I(\hG(S))\times\{\pm 1\}.\]
\end{enumerate}
\end{proposition}

\begin{proof}(i): for $d(S)>1$ and $S\neq S_{1,2}, S_1^2, S_2$, by \cite[Corollary~4.12]{BF}, we have $Z(\kG(S)_\g)=\hI_\g$ and 
$N_{\kG(S)}(\hI_\g)=\kG(S)_\g$, for all $\g\in\hL(S)_0$, and so the conclusion follows. 
\smallskip

\noindent
(ii): The same argument above applies to $\kPG(S)$ in the given cases.
\smallskip

\noindent
(iii): for $S=S_{1,1},S_1^1,S_{1,2},S_1^2$ or $S_2$, by \cite[Lemma~7.4]{BF}, there is an exact sequence:
\[1\to\Hom(\hG(S)/Z(\hG(S)),Z(\hG(S)))\to\Out(\hG(S))\to\Out(\hG(S)/Z(\hG(S))),\]
where the center $Z(\hG(S))$ is generated by the hyperelliptic involution $\iota$ and the abelianization of $\hG(S)/Z(\hG(S))$ is a cyclic group
of finite even order (cf.\ \cite[\S~5.1.3]{FM}). Therefore we have $\Hom(\hG(S)/Z(\hG(S)),Z(\hG(S)))\cong\{\pm 1\}$. Let us denote by $\phi_\iota$ its generator.

By (i) and (ii), we have that $\Out^\D(\hG(S)/Z(\hG(S)))=\Out^\I(\hG(S)/Z(\hG(S)))$. Therefore, if $f\in\Out^\D(\hG(S))$ is such that
$f\notin\Out^\I(\hG(S))$, we have $f\circ\phi_\iota\in\Out^\I(\hG(S))$. Since the intersection between the subgroups 
$\Out^\I(\hG(S))$ and $\Hom(\hG(S)/Z(\hG(S)),Z(\hG(S)))$ of $\Out(\hG(S))$ is trivial and $\phi_\iota$ normalizes $\Out^\I(\hG(S))$,
the conclusion follows.
\end{proof}

\subsection{The profinite Grothendieck-Teichm\"uller group}\label{GTdef}
The inertia condition of the previous sections plays a central role in Grothendieck-Teichm\"uller theory.
Let $\Aut^\flat(\hG_{0,n})$ be the subgroup of  of $\Aut(\hG_{0,n})$
consisting of those automorphisms which preserve the set of inertia groups $\{\hI_\g\}_{\g\in\hL^b(S_{0,n})}$ and let $\Out^\flat(\hG_{0,n})$
be the corresponding outer automorphism group.
The main theorem of Harbater and Schneps in \cite{HS} states that, for all $n\geq 5$, there is a natural isomorphism:
\begin{equation}\label{HS}
\Out^\flat(\hG_{0,n})\cong\Sigma_n\times\GT,
\end{equation} 
where $\GT$ is the profinite Grothendieck-Teichm\"uller group introduced 
by Drinfeld (cf.\ \cite{Drinfeld}) and the symmetric group $\Sigma_n$ identifies with $\Inn(\hG_{0,[n]})/\Inn(\hG_{0,n})$.
The key case of this isomorphism is for $n=5$ (cf.\ \cite[Theorem~4]{HS}). This also allows to give a more intrinsic definition of the
Grothendieck-Teichm\"uller group letting (note that, for $n=5$, the $\flat$ and $\I$-conditions are equivalent):
\[\GT:= Z_{\Out^\I(\hG_{0,5})}(\Sigma_5).\]

This characterization was further improved by Hoshi, Minamide and Mochizuki, who proved that 
$\Aut^\flat(\hG_{0,n})=\Aut(\hG_{0,n})$ (cf.\ \cite[Corollary~2.8]{HMM}).
Hence, for all $n\geq 5$, there is a natural isomorphism $\Out(\hG_{0,n})\cong\Sigma_n\times\GT$ and, in particular, we have,
even more intrinsically, that:
\[\GT= Z_{\Out(\hG_{0,5})}(\Sigma_5).\]

Since $\hG_{0,n}$ is a characteristic subgroup of $\hG_{0,[n]}$ (cf.\ \cite[Proposition~4.1, (ii)]{MN}), restriction of automorphisms 
induces a homomorphism $\Aut(\hG_{0,[n]})\to\Aut(\hG_{0,n})$. We can then reformulate the above isomorphisms in the following way:

\begin{proposition}\label{purevsfullgenus0}For $n\geq 5$, restriction of automorphisms induces an isomorphism $\Aut(\hG_{0,[n]})\cong\Aut(\hG_{0,n})$. 
In particular,  for $n\geq 5$, there is a natural isomorphism:
\begin{equation}\label{GT}
\Out(\hG_{0,[n]})\cong\GT.
\end{equation}
\end{proposition}

\begin{proof}By \cite[Theorem~4.14, (i)]{BF}, for $n\geq 5$, the centralizer of $\hG_{0,n}$ in $\hG_{0,[n]}$ is trivial, so that 
the natural homomorphism $\Inn(\hG_{0,[n]})\to\Aut(\hG_{0,n})$ is injective. By \cite[Lemma~3.3]{BF}, 
the homomorphism $\Aut(\hG_{0,[n]})\to\Aut(\hG_{0,n})$ is then also injective.

Let us consider the short exact sequence:
\[1\to\hG_{0,n}\to\hG_{0,[n]}\to\Sigma_n\to 1.\]
In order to prove the proposition, we have to show that, for $n\geq 5$, every automorphism of $\hG_{0,n}$ extends to $\hG_{0,[n]}$.
For this, we will need \cite[Lemma~2.18]{BF2}, which is based on Wells' exact sequence (cf.\ \cite[Theorem]{Wells}).

Since the outer representation $\rho\co\Sigma_n\to\Out(\hG_{0,n})$ associated to the above short exact sequence is faithful 
and the center of $\hG_{0,n}$ is trivial, the conditions of \cite[Lemma~2.18, (i) and (iii)]{BF2} are satisfied and we only need to check 
that condition (ii) is also satisfied. But this follows from the fact that $\rho(\Sigma_n)$ is a normal subgroup of $\Out(\hG_{0,n})$, which 
is an immediate consequence of the fact that, as explained above, for $n\geq 5$, there is a natural isomorphism
$\Out(\hG_{0,n})\cong\Sigma_n\times\GT$.
\end{proof}

The following result is essentially due to Minamide and Nakamura (cf.\ \cite{MN}):

\begin{proposition}\label{g=1n=2}
For $g(S)=1$ and, either $n(S)=2$ and $k(S)=0$, or $n(S)=0$ and $k(S)=2$, we have:
\begin{enumerate}
\item $\Out^\I(\hG(S))=\Out^\I(\hPG(S))=\Out(\hPG(S))\cong\GT$.
\item $\Out(\hG(S))\cong\{\pm 1\}\times\GT$.
\end{enumerate}
\end{proposition}

\begin{proof}Since $\G(S_1^2)=\G(S_{1,2})$ and $\PG(S_1^2)=\PG(S_{1,2})$, it is not restrictive to assume that $k(S)=0$ and $n(S)=2$.
By \cite[Corollary~C]{MN}, we know that $\Out(\hPG(S))\cong\GT$ which, in particular, implies that $\Out(\hPG(S))=\Out^\I(\hPG(S))$.
It is then enough to prove that:
\[\Out^\I(\hG(S))=\Out^\I(\hPG(S))\hspace{0.5cm}\mbox{ and }\hspace{0.5cm}\Out(\hG(S))\cong\{\pm 1\}\times\Out(\hPG(S)).\] 
The center $Z(\hG(S))$ is generated by the hyperelliptic involution $\iota$ and $\hG(S)=\langle\iota\rangle\times\hPG(S)$ 
and so $\hG(S)/Z(\hG(S))\cong\hPG(S)$. \cite[Lemma~7.4]{BF} then implies that there is a split short exact sequence:
\begin{equation}\label{lemmabfl}
1\to\Hom(\hG(S)/Z(\hG(S))),Z(\hG(S)))\to\Out(\hG(S))\to\Out(\hPG(S))\to 1.
\end{equation}
Since $\Hom(\hG(S)/Z(\hG(S)),Z(\hG(S)))\cong\{\pm 1\}$, it follows $\Out(\hG(S))\cong\{\pm 1\}\times\Out(\hPG(S))$ 
as claimed above. By the description of the image of the group $\Hom(\hG(S)/Z(\hG(S)),Z(\hG(S)))$ 
in $\Out(\hG(S))$, this has trivial intersection with $\Out^\I(\hG(S))$. The other claim above then also follows.
\end{proof}

\subsection{Nonseparating inertia groups}\label{nonseparatinginertia}
Let $C_0(S)$ be the \emph{nonseperating curve complex}. It is the subcomplex of the curve complex $C(S)$ consisting 
of the simplices $\s\in C(S)$ such that $S\ssm\s$ is connected. The \emph{procongruence nonseperating curve complex} $\kC_0(S)$ 
is the abstract simplicial profinite complex whose set of $k$-simplices is the closure of $C_0(S)_k$ in $\kC(S)_k$, for $k\geq 0$. 
In particular, a $0$-simplex of $\kC_0(S)$ is determined by a profinite simple closed curve of nonseparating topological type. 
We denote by $\hL^\mathrm{ns}(S)$ the closed subset of $\hL(S)$ consisting of such curves.
The topological type of a simplex $\s\in\kC_0(S)$ is determined by its dimension. 
Note that, for $g(S)=0$, we have $C_0(S)=\kC_0(S)=\emptyset$. 

\begin{definition}\label{Ins}For $H$ a closed subgroup of the procongruence mapping class group $\kG(S)$, 
the group of \emph{$\I_0$-automorphisms} $\Aut^{\I_0}(H)$ is the closed subgroup of $\Aut(H)$ consisting of those elements which preserve 
the set of inertia groups $\{\hI_\g(H)\}_{\g\in\hL^\mathrm{ns}(S)}$. We say that a subgroup $K$ of $H$ is $\I_0$-characteristic if it is preserved 
by all elements of $\Aut^{\I_0}(H)$. Note that, for $g(S)=0$, by definition, there holds $\Aut^{\I_0}(H)=\Aut(H)$.
\end{definition}

\begin{remark}\label{I_0rem}Let us assume that $H$ is an open subgroup of $\kG(S)$. Then,
the same argument of (i) of Proposition~\ref{0simplices} shows that an element $f\in\Aut(H)$ is an $\I_0$-automorphism 
if and only if $f$ preserves the set of inertia groups $\{I_\s(H)\}_{\s\in\kC_0(S)}$ of $H$. Moreover, from (ii) and (iii) of Proposition~\ref{0simplices},
it follows that, for either $S\neq S_{1,2}$ or $H=\kG(S_{1,2}),\kPG(S_{1,2})$, there holds $\Aut^{\I}(H)\leq \Aut^{\I_0}(H)$.
\end{remark} 

In fact, for pure procongruence mapping class groups, the $\I$-condition turns out to be no more restrictive than the $\I_0$-condition:

\begin{theorem}\label{IvsIns}For $d(S)>1$, we have $\Aut^\I(\kPG(S))=\Aut^{\I_0}(\kPG(S))$.
\end{theorem}

\begin{proof}We proceed by induction on the genus of $S$. The base for the induction will be provided by the cases $g(S)=0,1$, which we will
treat separately in the following lemmas. 

For $g(S)=0$, we have the following result which improves \cite[Corollary~2.8]{HMM}:  

\begin{lemma}\label{decpreservation}For $n\geq 5$, we have $\Aut^\I(\hG_{0,n})=\Aut(\hG_{0,n})\cong\Aut(\hG_{0,[n]})=\Aut^\I(\hG_{0,[n]})$.
\end{lemma}

\begin{proof}The first identity is \cite[Theorem~6.6]{BF2} while the isomorphism in the middle follows from Proposition~\ref{purevsfullgenus0}. 
The third identity then also follows.
\end{proof}

\begin{lemma}\label{decpreservation1}For $n\geq 2$, we have $\Aut^\I(\hG_{1,n})=\Aut^{\I_0}(\hG_{1,n})$.
\end{lemma}

\begin{proof}The case $n=2$ follows from (ii) of Proposition~\ref{g=1n=2} and the observation that, in the short exact sequence (cf.\ \cite[Lemma~7.4]{BF}):
\[1\to\Hom(\hG_{1,2}/Z(\hG_{1,2}),Z(\hG_{1,2}))\to\Aut(\hG_{1,2})\to\Aut(\hG_{1,2}/Z(\hG_{1,2})),\]
the image of $\Hom(\hG_{1,2}/Z(\hG_{1,2}),Z(\hG_{1,2}))$ has trivial intersection with $\Aut^{\I_0}(\hG_{1,2})$.

For $n\geq 3$, it is enough to prove that, for any $f\in\Aut^{\I_0}(\hG_{1,n})$ and $\g\in\cL(S)_0\subset\hL(S)_0$ separating, we have 
$f(\hI_\g)=\hI_{\g'}$, for some $\g'\in\hL(S)_0$. Let $\alpha$ be a nonseparating simple closed curve on $S$ such that $\tau_\g\in\G_\alpha$. 
After composing $f$ with a suitable inner automorphism, we can assume that $f$ preserves the stabilizer $\hG_\alpha$.
According to (ii) of Theorem~\ref{stabilizers}, this group is described by the short exact sequence:
\[1\to\hI_\alpha\to\hG_\alpha\to\hG_{0,n+2}\to 1.\]
Therefore, $f$ induces an automorphism $\bar f$ of the quotient group $\hG_\alpha/\hI_\alpha\cong\hG_{0,n+2}$. 
Let us denote by $\bar I_\g$ the image of $\hI_\g$ in this quotient. By Lemma~\ref{decpreservation}, we have that $\bar f(\bar I_\g)=\hI_{\bar\g'}$, 
for some profinite simple closed curve $\bar\g'$ on $S_{0,n+2}$. 

Let $\g'\in\hL(S_{1,n})$ be a lift of $\bar\g'$. Then, $\g'$ is also of separating type and there is an identity $f(\tau_\g)=\tau_{\g'}^h\tau_\alpha^k$, 
for some $h\in\ZZ^\ast$ and $k\in\ZZ$. The kernel of the natural epimorphism $\hG_{1,n}\to\hG_{1,1}$ is topologically generated by elements
of the form $\tau_{\beta_1}^{h}\tau_{\beta_2}^{-h}$, with $\beta_1,\beta_2\in\hL^\mathrm{ns}(S)$ and $h\in\ZZ^\ast$, and is therefore preserved
by $f$. In particular, $f(\tau_\g)$ belongs to this kernel. By projecting the identity $f(\tau_\g)=\tau_{\g'}^h\tau_\alpha^k$ to $\hG_{1,1}$, 
we then see that $k=0$, which concludes the proof of the lemma.
\end{proof}

We now proceed to the proof of Theorem~\ref{IvsIns} by induction on the genus. The base for the induction is provided by 
Lemma~\ref{decpreservation1}. So let us assume that the statement of Theorem~\ref{IvsIns} holds in genus $g-1\geq 1$ and let us prove 
it for genus $g=g(S)$. For simplicity, let us denote by $\{\hI_\s\}_{\s\in\kC(S)}$ the set of inertia groups of $\kPG(S)$. 
The key observation is that the hypothesis $g(S)\geq 2$ implies that a $k$-simplex $\s\in C(S)$, for $k=d(S)-1,d(S)-2$, contains at least a nonseparating curve. 

From the induction hypothesis, we will deduce that $\Aut^{\I_0}(\kPG(S))$ preserves the subsets of inertia groups $\{\hI_\s\}_{\s\in\kC(S)_k}$,
for $k=d(S)-1,d(S)-2$.  In fact, if $\alpha\in\s$ is nonseparating, after composing a given $f\in\Aut^{\I_0}(\kPG(S))$ with an inner automorphism, 
we can assume that $f(\alpha)=\alpha$, so that $f$ restricts to an automorphism of the stabilizer $\kPG(S)_\alpha$ which preserves the inertia group 
$\hI_\alpha$ and is contained in $\Aut^{\I_0}(\kPG(S)_\alpha)$. Since $\hI_\s\subset\kPG(S)_\alpha$ and $\kPG(S)_\alpha/\hI_\alpha\cong\kPG(S\ssm\alpha)$, 
the induction hypothesis implies that the automorphism $\bar f$ induced by $f$ on the quotient $\kPG(S)_\alpha/\hI_\alpha$ has the property that 
$\bar f(\hI_\s/\hI_\alpha)=\hI_{\s'}/\hI_\alpha$, for some $\alpha\in\s'\in\kC(S)_k$, for $k=d(S)-1,d(S)-2$, and then $f(\hI_\s)=\hI_{\s'}$. 

We can now proceed as in the proof of \cite[Theorem~6.6]{BF2}. 
Let us identify (cf.\ (i) of Remarks~\ref{algreal}) the procongruence curve complex $\kC(S)$ with the abstract simplicial 
profinite complex whose set of $k$-simplices is the set of closed subgroups $\{\hI_\s\}_{\s\in\kC_k(S)}$.
The dual graph $\kC^\ast(S)$ of $\kC(S)$ (cf.\ \cite[Definition~3.9]{BF}) has then for vertex set the set $\{\hI_\s\}_{\s\in\kC(S)_{d(S)-1}}$ 
and two such vertices $\hI_\s$ and $\hI_{\s'}$ are connected by an edge if and only if $|\s\cap\s'|=d(S)-1$. 
This happens if and only if $\hI_\s\cap\hI_{\s'}$ equals $\hI_{\s\cap\s'}$ and is a free $\ZZ$-module of rank $d(S)-1$.

The remarks above then imply that the natural action of $\Aut^{\I_0}(\kPG(S))$ on the set of closed subgroups of $\kPG(S)$ induces 
a representation $\Aut^{\I_0}(\kPG(S))\to\Aut(\kC^\ast(S))$. 

By \cite[Lemma~6.5]{BF}, the procongruence curve complex $\kC(S)$ can be reconstructed from its dual graph $\kC^\ast(S)$ 
and there is a natural isomorphism $\Aut(\kC^\ast(S))\cong\Aut(\kC(S))$. It then follows that the action of $\Aut^{\I_0}(\kPG(S))$ 
on the set of closed subgroups preserves the set of \emph{all} inertia groups of $\kPG(S)$. 
\end{proof}

For $g(S)\geq 1$, the pure procongruence mapping class group $\kPG(S)$ is topologically generated by nonseparating Dehn twists. 
It is then an $\I_0$-characteristic subgroup of the full procongruence mapping class group $\kG(S)$. We already observed
that, by \cite[Proposition~4.1, (ii)]{MN}, $\hG_{0,n}$ is a characteristic subgroup of $\hG_{0,[n]}$.
Therefore, restriction of $\I_0$-automorphisms from $\kG(S)$ to $\kPG(S)$
defines a natural homomorphism $\Aut^{\I_0}(\kG(S))\to\Aut^{\I_0}(\kPG(S))$. 
Since $\Aut^\I(\kG(S))\leq\Aut^{\I_0}(\kG(S))$ (cf.\ Remark~\ref{I_0rem}), from Theorem~\ref{IvsIns}, it follows:

\begin{corollary}\label{IvsIns2}For $d(S)>1$, we have $\Aut^\I(\kG(S))=\Aut^{\I_0}(\kG(S))$.
\end{corollary}

\section{Extending $\I$-automorphisms}
The main result of this section is the following:
 
\begin{theorem}\label{compurevsfull}Let $S=S_{g,n}^k$ be a hyperbolic surface with $d(S)>1$. 
Let us recall the notation $\ring{S}:=S\ssm\dS$. We then have:
\begin{enumerate}
\item Restriction of automorphisms induces an isomorphism $\Aut^\I(\kG(\ring{S}))\cong\Aut^\I(\kPG(S))$.
\item For $(g,n+k)\neq (1,2)$, there is a natural isomorphism:
\[\Out^\I(\kPG(S))\cong\Sigma_{n+k}\times\Out^\I(\kG(\ring{S})).\]
\item For $(g,n+k)= (1,2)$, there holds $\Out^\I(\hPG(S))=\Out^\I(\hG(S))$.
\end{enumerate}
\end{theorem}

\subsection{Preliminary lemmas}Thanks to the results of \cite{BF2}, 
there is a useful criterion to determine whether an $\I$-automorphism of $\kPG(S)$ is induced by an inner automorphism of $\kG(\ring{S})$.
After identifying the procongruence curve complex $\kC(S)$ with the simplicial profinite complex whose simplices are the inertia groups 
$\{I_\s\}_{\s\in\kC(S)}$ (cf.\ (i) of Remarks~\ref{algreal}), the vertices of the procongruence pants complex $\kC_P(S)$ are identified with the set 
$\{I_\s\}_{\s\in\kC(S)_{d(S)-1}}$ of inertia subgroups of maximal rank. We then have:

\begin{lemma}\label{innercriterion}Let $\phi\in\Aut^\I(\kPG(S))$ be an element which, via the natural representation $\Aut^\I(\kPG(S))\to\Aut(\kC(S))$,
acts trivially on the simplices corresponding to the vertices of a profinite Farey subgraph $\hF_\mu$ of $\kC_P(S)$ (cf.\ \cite[Definition~6.3]{BF}). 
Then, $\phi\in\Inn(\kG(\ring{S}))$.
\end{lemma}

\begin{proof}It is not restrictive to assume $\partial S=\emptyset$. For $d(S)=1$ the claim is trivial. For $d(S)\geq 2$, \cite[Theorem~2.17]{BF2}
implies that $\phi\in\Inn(\kG^\pm(\ring{S}))$, where $\kG^\pm(\ring{S})$ is the congruence completion of the extended mapping class group $\G^\pm(\ring{S})$.
Since, by hypothesis, in particular, $\phi$ preserves the orientation of the profinite Farey subgraph $\hF_\mu$, from \cite[Theorem~2.2]{BF2}, it follows that
$\phi\in\Inn(\kG(\ring{S}))$.
\end{proof}

Let us also observe that:

\begin{lemma}\label{innerinj}Let $S=S_{g,n}^k$ and assume that $(g,n+k)\neq (0,4)$. Then, the natural homomorphism 
$\Aut^\I(\kG(S))\to\Aut^\I(\kPG(S))$ is injective.
\end{lemma}

\begin{proof}The claim is trivial for $n(S)\leq 1$, while, for $g(S)=1$ and $n(S)=2$, $k(S)=0$, or $n(S)=0$, $k(S)=2$, it follows from (i) of
Proposition~\ref{g=1n=2}. Since $(g,n+k)\neq (0,4)$, by \cite[Theorem~4.14]{BF}, we can then assume that the centralizer of $\kPG(S)$ in 
$\kG(S)$ is trivial, so that the natural homomorphism $\Inn(\kG(S))\to\Aut^\I(\kPG(S))$ is injective. By \cite[Lemma~3.3]{BF}, the homomorphism 
$\Aut^\I(\kG(S))\to\Aut^\I(\kPG(S))$ is then also injective. 
\end{proof}

The following lemma will be essential for the proof of Theorem~\ref{compurevsfull}:

\begin{lemma}\label{purevsfull}For $d(S)>1$, the natural monomorphism $\Aut^\I(\kG(\ring{S}))\hookra\Aut^\I(\kPG(S))$ 
identifies $\Inn(\kG(\ring{S}))$ with a normal subgroup of $\Aut^\I(\kPG(S))$.
\end{lemma} 

\begin{proof}Since $\kPG(S)=\kPG(\ring{S})$, it is not restrictive to assume that $\partial S=\emptyset$ and so $\ring{S}=S$.
Let us also observe that $\kG(S)$ is generated by its normal subgroup $\kPG(S)$ and braid twists. 

Since $\Inn(\kPG(S))$ is a normal subgroup of $\Aut^\I(\kPG(S))$, it is then enough to show
that, for a given braid twist $b_\g$, where $\g\in\cL^b(S)$, for every $f\in\Aut^\I(\kPG(S))$, we have that $\phi:=f\circ\inn b_\g\circ f^{-1}\in\Inn(\kG(S))$.
The key observation is that, since $\inn b_\g$ restricts to the identity on the stabilizer $\kPG(S)_\g$, the automorphism $\phi$ restricts to 
the identity on the stabilizer $\kPG(S)_{f(\g)}$. Since, for every $\s\in\kC(S)_{d(S)-2}$ such that $\g\in\s$, the vertices of the profinite Farey 
subgraph $\hF_{f(\s)}\subset\kC_P(S)$ are subgroups of the stabilizer $\kPG(S)_{f(\g)}$,
the action of $\phi$ restricts to the identity on the vertices of all these profinite Farey subgraphs. 
By Lemma~\ref{innercriterion}, we then conclude that $\phi\in\Inn(\kG(S))$.
\end{proof}

Note that for  $S=S_{g,n}^k$ and $(g,n+k)\neq(1,2)$, the quotient $\left.\Inn(\kG(\ring{S}))\right/\Inn(\kPG(S))$ identifies with the quotient
$\kG(\ring{S})/\kPG(S)\cong\Sigma_{n+k}$. From Lemma~\ref{purevsfull}, it then follows that $\Sigma_{n+k}$ identifies with a normal subgroup
of $\Out^\I(\kPG(S))$. We have:

\begin{lemma}\label{autsym}For $S=S_{g,n}^k$ a hyperbolic surface such that $(g,n+k)\neq(1,2)$, restriction of inner automorphisms
induces an epimorphism $\Out^\I(\kPG(S))\to\Inn\Sigma_{n+k}$.
\end{lemma}

\begin{proof}We can assume that $n+k\geq 2$ and, since $(g,n+k)\neq(1,2)$, in particular, that the center of $\kG(S)$ is trivial. By the results of \cite{Segal},
it is enough to show that an automorphism of $\Sigma_{n+k}$ induced by restriction of an element of $\Inn(\Out^\I(\kPG(S)))$ sends a transposition to another 
transposition. This follows from the fact that, by (iv) of Proposition~\ref{0simplices}, the conjugacy action of $\Aut^\I(\kPG(S))$ on its normal 
subgroup $\Inn\kG(S)\cong\kG(S)$ sends a profinite braid twist to another profinite braid twists, since it preserves conjugacy classes of profinite Dehn twists. 
\end{proof}

In particular, for $n+k>2$, there is a natural epimorphism $\Out^\I(\kPG(S))\to\Sigma_{n+k}$ which is a left inverse of the inclusion
$\Sigma_{n+k}\hookra\Out^\I(\kPG(S))$. For $g>1$, this is also true for $n+k=2$, thanks to the following lemma:

\begin{lemma}\label{autcurvecomplex}For $S=S_{g,n}$ such that either $g\geq 1$ and $(g,n)\neq(1,2)$ or $g=0$ and $n\geq 5$, there is 
an epimorphism $\Aut^\I(\kPG(S))\to\Sigma_n$, such that its composition with the homomorphism $\inn\co\kG(S)\to\Aut^\I(\kPG(S))$ 
is the natural epimorphism $\kG(S)\to\Sigma_n$.
\end{lemma}

\begin{proof}The statement of the lemma is similar to \cite[Lemma~4.2]{BF2} and is proved essentially in the same way.
\end{proof}

\subsection{Proof of Theorem~\ref{compurevsfull}}We can assume $n+k\geq 2$, otherwise the theorem is trivial. 
For the proof of item (i) and (ii) of the theorem, it is not restrictive to assume $\dS=\emptyset$. 
\smallskip

\noindent
(i): For $S=S_{1,2}$, this is just (i) of Proposition~\ref{g=1n=2}. Therefore, we can also assume $S\neq S_{1,2}$ and then, 
by \cite[Theorem~4.14, (i)]{BF}, that the centralizer of $\kPG(S)$ in $\kG(S)$ is trivial. 

We have to show that the natural monomorphism $\Aut^\I(\kG(S))\hookra\Aut^\I(\kPG(S))$ (cf.\ Lemma~\ref{innerinj}) is surjective. 
For this, as in the proof of Proposition~\ref{purevsfullgenus0}, we consider the short exact sequence:
\[1\to\kPG(S)\to\kG(S)\to\Sigma_n\to 1.\]

Let $\rho\co\Sigma_n\to\Out(\kPG(S))$ be the outer representation associated to the above short exact sequence. The image of $\Inn(\kG(S))$ in $\Out(\kPG(S))$
coincides with $\rho(\Sigma_n)$. From Lemma~\ref{purevsfull}, it then follows that $\rho(\Sigma_n)$ is a normal subgroup of $\Out^\I(\kPG(S))$. 

Since $Z_{\kG(S)}(\kPG(S))=\{1\}$, we have that $\Sigma_n\cong\Inn(\kG(S))\left/\Inn(\kPG(S))\right.\subset\Out(\kPG(S))$, 
so that the representation $\rho$ is faithful. The same argument of the proof of
Proposition~\ref{purevsfullgenus0} then implies that every automorphism $f\in\Aut^\I(\kPG(S))$ extends to $\kG(S)$,
thus completing the proof of the first item.
\smallskip

\noindent
(ii): By the previous item of the theorem, there is a short exact sequence:
\[1\to\left.\Inn(\kG(S))\right/\Inn(\kPG(S))\to\left.\Aut^\I(\kPG(S))\right/\Inn(\kPG(S))\to\left.\Aut^\I(\kG(S))\right/\Inn(\kG(S))\to 1.\]

Since, by hypothesis, $n\neq 2$ for $g=1$, there holds $Z(\kG(S))=Z(\kPG(S))=\{1\}$ and so we have
$\left.\Inn(\kG(S))\right/\Inn(\kPG(S))\cong\left.\kG(S)\right/\kPG(S)\cong\Sigma_n$. By Lemma~\ref{autcurvecomplex}, 
the monomorphism $\Sigma_n\hookra\Out^\I(\kPG(S))$ has a left inverse $\Out^\I(\kPG(S))\to\Sigma_n$ which provides the splitting.

\noindent
(iii): For $S=S_{1,1}^1$, we have $\hPG(S)\cong\hG(S)$ and the statement is trivial.
For $S= S_{1,2}$ or $S_1^2$, this is just the first item of Proposition~\ref{g=1n=2}.

\subsection{Inertia conditions for procongruence relative mapping class groups}\label{starcondboundary}
For a surface $S=S_{g,n}^k$ with boundary $\dS=\cup_{i=1}^k\delta_i$, let $\tS\cong S_{g,n+2k}$ be the surface introduced in Definition~\ref{tildedS}. 
The natural embedding $S\hookra \tS$ induces a monomorphism of procongruence mapping class groups $\kG(S,\dS)\hookra\kG(\tS)$ 
(cf.\ Section~\ref{MCGboundary}).
Therefore, Definition~\ref{starcondition} also applies to procongruence relative mapping class groups.
Note that $\prod_{i=1}^k\tau_{\d_i}^\ZZ$ is the maximal normal inertia subgroup of $\kG(S,\dS)$. 
We call it the \emph{normal inertia group} of $\kG(S,\dS)$. From (iii) of Proposition~\ref{congcent}, it follows that this subgroup is 
also $\I$-characteristic. In particular, there is a natural homomorphism $\Aut^\I(\kG(S,\dS))\to\Aut^\I(\kG(S))$.

\begin{definition}\label{starcondboundaryns}As we observed above, by Definition~\ref{starcondition}, for an open subgroup $U$ of $\kG(S,\dS)$, 
the group $\Aut^\I(U)$ is defined to be the subgroup of $\Aut(U)$ consisting of those automorphisms which preserve 
the inertia groups $\{\hI_\g(U)\}_{\g\in\hL(\tS)_0}$. We then let $\Aut^{\I_0}(U)$ to be the subgroup of $\Aut(U)$ consisting 
of those automorphisms which preserve the set of inertia groups $\{\hI_\g(U)\}_{\g\in\hL^\mathrm{ns}(S)}\cup\{\hI_{\delta_i}(U)\}_{i=1,\dots,k}$ 
(note that $\hL^\mathrm{ns}(S)\equiv\hL^\mathrm{ns}(\tS)$).
\end{definition}

From (iii) of Proposition~\ref{congcent}, it follows
that the \emph{normal inertia group} $A:=U\cap\prod_{i=1}^k\tau_{\d_i}^\ZZ$ of $U$ is an $\I$-characteristic subgroup of $U$. 
From Definition~\ref{starcondboundaryns}, it then follows that it is also $\I_0$-characteristic in $U$. 
In particular, if we let $\ol{U}:=U/A$, there are natural homomorphisms 
$\Aut^\I(U)\to\Aut^\I(\ol{U})$ and $\Aut^{\I_0}(U)\to\Aut^{\I_0}(\ol{U})$.

We say that a group $G$ is \emph{virtually (topologically) generated} by a subset $T$, 
if $G$ contains a finite index subgroup (topologically) generated by this subset. 

\begin{lemma}\label{injectivehom}We have:
\begin{enumerate}
\item Let $U$ be an open subgroup of $\kG(S,\dS)$ virtually topologically generated by the inertia groups $\{\hI_\g(U)\}_{\g\in\hL(\tS)_0}$. 
Then, the natural homomorphisms: 
\[\Aut^\I(U)\to\Aut^\I(\ol{U})\times\Aut(A)\hspace{0.4cm}\mbox{ and }\hspace{0.4cm}\Out^\I(U)\to\Out^\I(\ol{U})\times\Aut(A)\] 
are injective.
\item Let $U$ be an open subgroup of $\kG(S,\dS)$ virtually topologically generated by the inertia groups 
$\{\hI_\g(U)\}_{\g\in\hL^\mathrm{ns}(S)}\cup\{\hI_{\delta_i}(U)\}_{i=1,\dots,k}$. Then, the natural homomorphisms:
\[\Aut^{\I_0}(U)\to\Aut^{\I_0}(\ol{U})\times\Aut(A)\hspace{0.4cm}\mbox{ and }\hspace{0.4cm}\Out^{\I_0}(U)\to\Out^{\I_0}(\ol{U})\times\Aut(A)\] 
are injective.
\end{enumerate}
\end{lemma}

\begin{proof}It is enough to prove the first item since the proof of the second is basically the same. By \cite[Lemma~7.4]{BF},
the kernel of the homomorphism $\Aut(U)_A\to\Aut(\ol{U})\times\Aut(A)$ (resp.\ $\Out(U)_A\to\Out(\ol{U})\times\Aut(A)$) 
identifies with the (free) abelian group $\Hom(\ol{U},A)$. 

The restriction of a homomorphism $\phi\in \Hom(\ol{U},A)\cap\Aut^\I(U)$ to the finite index subgroup of $\ol{U}$ topologically generated 
by inertia groups is trivial. Since $A$ is torsion free, this implies that $\phi$ itself is trivial and hence that $\Hom(\ol{U},A)\cap\Aut^\I(U)=\{1\}$.
\end{proof}

Let us observe that $\kPG(S,\dS)$, for $g(S)\geq 1$, is topologically generated by nonseparating Dehn twists, so that $\kG(S,\dS)$ is 
virtually topologically generated by them. An immediate consequence of Lemma~\ref{injectivehom} and Theorem~\ref{IvsIns} is then:

\begin{proposition}\label{IvsIns3}For $g(S)\geq 1$, there holds $\Aut^\I(\kPG(S,\dS))=\Aut^{\I_0}(\kPG(S,\dS))$ and 
$\Aut^\I(\kG(S,\dS))=\Aut^{\I_0}(\kG(S,\dS))$.
\end{proposition}

\subsection{Extending automorphisms of procongruence relative mapping class groups}\label{autrelprocong} 
In this section, we will prove a version of Theorem~\ref{compurevsfull} for procongruence relative mapping class groups.
For this, we will need a series of preliminary lemmas.

The subgroup $\kPG(S,\dS)$ of $\kG(S,\dS)$ is $\I$-characteristic. There is then a natural homomorphism $\Aut^\I(\kG(S,\dS))\to\Aut^\I(\kPG(S,\dS))$ and,
from Lemma~\ref{innerinj} and (i) of Lemma~\ref{injectivehom}, it follows:

\begin{lemma}\label{relativerestriction}Let $S=S_{g,n}^k$ and assume that $(g,n+k)\neq (0,4)$. Then, the natural homomorphism 
$\Aut^\I(\kG(S,\dS))\to\Aut^\I(\kPG(S,\dS))$ is injective.
\end{lemma}

For $(g,n+k)\neq (0,4)$, let us identify $\Inn(\kG(S,\dS))$ with a subgroup of $\Aut^\I(\kPG(S,\dS))$ by means of the above monomorphism. 
We then have:

\begin{lemma}\label{purevsfull2}For $d(S)>1$, there holds $\Inn(\kG(S,\dS))\lhd\Aut^\I(\kPG(S,\dS))$.
\end{lemma}

\begin{proof}Composing the natural monomorphism $\Aut^\I(\kG(S,\dS))\hookra\Aut^\I(\kPG(S,\dS))$ (cf.\ Lemma~\ref{relativerestriction})
with the monomorphism of (ii), Lemma~\ref{injectivehom}, we can identify $\Inn(\kG(S,\dS))$ with a subgroup of 
$\Aut^\I(\kPG(S))\times\Aut(\prod_{i=1}^k\tau_{\d_i}^\ZZ)$. Let us then observe that, for $S=S_{g,n}^k$, the image of $\Inn(\kG(S,\dS))$
in the quotient $\Out^\I(\kPG(S))\times\Aut(\prod_{i=1}^k\tau_{\d_i}^\ZZ)$ identifies with:
\[\Inn(\kG(S,\dS))\left/\Inn(\kPG(S,\dS))\right.\cong\Inn(\kG(S))\left/\Inn(\kPG(S))\right.\cong\Sigma_n\times\Sigma_k.\]

From Lemma~\ref{purevsfull} (and its proof), it then follows that the image of the group $\Inn(\kG(S,\dS))$ inside 
$\Aut^\I(\kPG(S))\times\Aut(\prod_{i=1}^k\tau_{\d_i}^\ZZ)$ is normalized by the image of $\Aut^\I(\kPG(S,\dS))$.
\end{proof}

We also have:

\begin{lemma}\label{autcurvecomplex2}For $S=S_{g,n}^k\neq S_{1,2}$ a hyperbolic surface, there is a natural
epimorphism $\Aut^\I(\kPG(S,\dS))\to\Sigma_n\times\Sigma_k$, such that its composition with the homomorphism 
$\inn\co\kG(S,\dS)\to\Aut^\I(\kPG(S,\dS))$ is the natural epimorphism $\kG(S,\dS)\to\Sigma_n\times\Sigma_k$.
\end{lemma}

\begin{proof}For $n=0$ and $k>0$, we consider the action of $\Aut^\I(\kPG(S,\dS))$ on the normal inertia subgroup $\prod_{i=1}^k\tau_{\d_i}^\ZZ$ 
induced by restriction of inner automorphisms. This action permutes the procyclic subgroups $\tau_{\d_i}^\ZZ$ and defines an epimorphism 
$\Aut^\I(\kPG(S,\dS))\to\Sigma_k$ with the desired properties.

For $n>0$, we consider the natural homomorphism $\Aut^\I(\kPG(S,\dS))\to\Aut^\I(\kPG(S))$.
We then get, for $g=0$, by the isomorphism~(\ref{HS}), and, for $g>0$, $(g,n+k)\neq(1,2)$, by Lemma~\ref{autcurvecomplex},
a homomorphism $\Aut^\I(\kPG(S,\dS))\to\Sigma_{n+k}$ compatible with the natural homomorphisms 
$\Inn(\kG(S,\dS))\hookra\Aut^\I(\kPG(S,\dS))$ and $\Inn(\kG(S,\dS))\twoheadrightarrow\Sigma_n\times\Sigma_k\subseteq\Sigma_{n+k}$.

Let $\{\g_0,\g_1\}\in C(S)\subset\kC(S)$ be a $1$-simplex bounding an annulus in $S$ containing a boundary component of the surface. 
From the description of the normalizer in $\kPG(S,\dS)$ of the closed subgroup topologically generated by $\tau_{\g_0}$ and $\tau_{\g_1}$ 
(cf.\ Theorem~\ref{normalizers multitwists}), it easily follows that an element of $\Aut^\I(\kPG(S,\dS))$ maps the simplex $\{\g_0,\g_1\}$ 
to another simplex of $\kC(S)$ which has the same topological type, that is to say which is in the $\kG(S)$-orbit of a $1$-simplex in $C(S)$ 
consisting of curves which also bound an annulus in $S$. 

Therefore, the action of $\Aut^\I(\kPG(S,\dS))$ on the set of all punctures and boundary
components of $S=S_{g,n}^k$ preserves the partition of this set into punctures and boundary components. 
This means that the image of the homomorphism $\Aut^\I(\kPG(S,\dS))\to\Sigma_{n+k}$ defined above is the subgroup $\Sigma_n\times\Sigma_k$,
which stabilizes this partition. The second claim of the lemma is also clear.
\end{proof}

The analogue of Theorem~\ref{compurevsfull} then holds:

\begin{theorem}\label{purevsfullboundary}Let $S=S_{g,n}^k$ be a hyperbolic surface such that $d(S)>1$.
\begin{enumerate}
\item Restriction of automorphisms induces an isomorphism: 
\[\Aut^\I(\kG(S,\dS))\cong\Aut^\I(\kPG(S,\dS)).\]
\item For $S\neq S_{1,2}$, there is a natural isomorphism:
\[\Out^\I(\kPG(S,\dS))\cong\Sigma_n\times\Sigma_k\times\Out^\I(\kG(S,\dS)).\]
\end{enumerate}
\end{theorem} 

\begin{proof}(i): We can assume $k>0$ and then, since $d(S)>1$, that the centers of $\kPG(S,\dS)$ and $\kG(S,\dS)$ 
do not contain a hyperelliptic involution. We have to show that the natural monomorphism $\Aut^\I(\kG(S,\dS))\hookra\Aut^\I(\kPG(S,\dS))$ 
(cf.\ Lemma~\ref{relativerestriction}) is surjective. For this, we consider the short exact sequence:
\[1\to\kPG(S,\dS)\to\kG(S,\dS)\to\Sigma_n\times\Sigma_k\to 1.\]

By Lemma~\ref{purevsfull2} and item (iii) of Proposition~\ref{congcent}, we can then argue as in the proof of the first item of Theorem~\ref{compurevsfull}.
\medskip

\noindent
(ii): By the first item of Theorem~\ref{purevsfullboundary}, proved above, and Lemma~\ref{autcurvecomplex2}, we can argue as in the proof of item
(ii) of Theorem~\ref{compurevsfull}.
\end{proof}

\section{Homomorphisms between outer automorphism groups of procongruence mapping class groups}
In this section, we will assume that $S$ is a hyperbolic surface with empty boundary.
\subsection{Functorial homomorphisms}\label{functorialhomomorphisms}
Let $\g$ be a nonseparating simple closed curve on $S$. By (ii) of Theorem~\ref{stabilizers}, there is a short exact sequence:
\[1\to\tau_\g^\ZZ\to\kPG(S)_{\vec\g}\to\kPG(S\ssm\g)\to 1.\]
Since the group $\kPG(S\ssm\g)$ is center free (cf.\ \cite[Corollary~6.2]{[B3]}), we have $Z(\kPG(S)_{\vec\g})=\tau_\g^\ZZ$.

We want to construct a natural homomorphism $\Aut^\I(\kPG(S))\to\Out^\I(\kPG(S\ssm\g))$.
The first step is to construct a homomorphism $\Aut^\I(\kPG(S))\to\Out^\I(\kPG(S)_{\vec\g})$. 

\begin{remark}\label{inertiapreservingdef}Note that, for $U$ an open subgroup of $\kG(S)_\g$ (e.g.\ $\kPG(S)_{\vec\g}$), the nontrivial inertia groups 
of $U$ are parameterized by the subcomplex $\mathrm{Star}(\g)\subset\kC(S)$ while the nontrivial inertia groups of the image of $U$ via the natural 
homomorphism $\kG(S)_\g\to\kG(S\ssm\g)$ are parameterized by the subcomplex $\Link(\g)\cong\kC(S\ssm\g)\subset\kC(S)$ (cf.\ \cite[Remark~4.7]{[B3]}). 
\end{remark}

We then proceed as in the proof of \cite[Theorem~7.4]{[B3]}.
Let us choose an orientation for every element of $\hL(S)$ (cf.\ \cite[Section~4]{[B3]}). By definition, an element
$f\in \kG(S)_\g$ preserves the orientation chosen for $\g$ if and only if $f\in\kG(S)_{\vec\g}$.

For $f\in\Aut^\I(\kPG(S))$, by items (ii) and (iii) of Proposition~\ref{0simplices}, there is an element $x\in\kPG(S)$ 
such that $\inn x\circ f$ preserves the inertia group $I_\g$ and, by \cite[Corollary~4.12, (i)]{BF}, also the decomposition group 
$\kPG(S)_\g$. Let us then observe that $\kPG(S)_{\vec\g}$ is a $\I$-characteristic subgroup of $\kPG(S)_\g$, because it is (topologically) generated
by the profinite Dehn twists contained in $\kPG(S)_\g$. Hence, by restriction, $\inn x\circ f$ induces an automorphism of $\kPG(S)_{\vec\g}$.

Since there is an element of $\kPG(S)_\g$ which reverses the orientation of $\g$, we can choose the element 
$x\in\kPG(S)$ in a way that it sends the orientation fixed on $\g$ to the orientation fixed on $x(\g)$. If $y\in\kPG(S)$ is another such element, 
then $xy^{-1}\in\kPG(S)_{\vec\g}$. In this way, the automorphism $\inn x\circ f$ is determined by $f$ modulo an inner automorphism of $\kPG(S)_{\vec\g}$,
so that we get a natural homomorphism:
\begin{equation}\label{stabrestriction}
\td{PR}_\g\co\Aut^\I(\kPG(S))\to\Out^\I(\kPG(S)_{\vec\g}).
\end{equation}

Since $Z(\kPG(S)_{\vec\g})=\tau_\g^\ZZ$ is a characteristic subgroup of $\kPG(S)_{\vec\g}$, there is also a natural homomorphism
$\Out^\I(\kPG(S)_{\vec\g})\to\Out^\I(\left.\kPG(S)_{\vec\g}\right/\tau_\g^\ZZ)\cong\Out^\I(\kPG(S\ssm\g))$.
Composing with the homomorphism~(\ref{stabrestriction}), we get the natural homomorphism: 
\[PR_\g\co\Aut^\I(\kPG(S))\to\Out^\I(\kPG(S\ssm\g)).\]

\begin{remark}\label{notOut}Note that $PR_\g$ does not descend to a homomorphism from $\Out^\I(\kPG(S))$ to $\Out^\I(\kPG(S\ssm\g))$.
In fact, an inner automorphism of $\kPG(S)$ which reverses the orientation of $\g$ (e.g.\ a hyperelliptic involution fixing $\g$) induces on
$\kPG(S\ssm\g)$ by the above procedure an automorphism which is not inner (it is in the image of $\Inn(\kG(S\ssm\g))$ in $\Aut^\I(\kPG(S\ssm\g))$).
\end{remark}

By Theorem~\ref{compurevsfull}, there is a natural homomorphism $\Out^\I(\kPG(S\ssm\g))\to\Out^\I(\kG(S\ssm\g))$.
Composing with $PR_\g$, we get the homomorphism:
\[\ol{PR}_\g\co\Aut^\I(\kPG(S))\to\Out^\I(\kG(S\ssm\g)).\]
It is easy to see that $\Inn(\kG(S))$ is contained in the kernel of this homomorphism. 
So, again by Theorem~\ref{compurevsfull}, we get a natural homomorphism:
\begin{equation}\label{stabinducedfull1}
R_\g\co\Out^\I(\kG(S))\to\Out^\I(\kG(S\ssm\g)).
\end{equation}

Let $\delta$ be a simple closed curve on $S$ bounding a $2$-punctured closed disc $D_\delta$. There is then a short exact sequence 
(cf.\ (ii) of Theorem~\ref{stabilizers}):
\begin{equation}\label{deltasequence}
1\to\tau_\delta^\ZZ\to\kPG(S)_\delta\to\kPG(S\ssm D_\delta)\to 1.
\end{equation}
Note that, from \cite[Corollary~6.2]{[B3]}, it follows that $Z(\kPG(S)_\delta)=\tau_\delta^\ZZ$.

The same construction used above for a nonseparating curve $\g$ (except that, in this case, since there holds $\kPG(S)_\delta=\kPG(S)_{\vec\delta}$, 
we do not need to consider orientations) shows that there is a natural homomorphism:
\begin{equation}\label{stabrestriction2}
\td{PR}_{\delta}\co\Aut^\I(\kPG(S))\to\Out^\I(\kPG(S)_\delta),
\end{equation}
whose kernel contains $\Inn(\kPG(S))$ and thus descends to a homomorphism:
\begin{equation}\label{stabrestriction3}
PR_{\delta}\co\Out^\I(\kPG(S))\to\Out^\I(\kPG(S)_\delta).
\end{equation}
Since $\Out^\I(\kPG(S)_\delta)$ preserves the center $\tau_\delta^\ZZ$ of $\kPG(S)_\delta$, from the short exact sequence~\eqref{deltasequence}, 
we also get a homomorphism:
\[PR_{D_\delta}\co\Out^\I(\kPG(S))\to\Out^\I(\kPG(S\ssm D_\delta)),\]
and then, by Theorem~\ref{compurevsfull}, a homomorphism:
\begin{equation}\label{stabinducedfull2}
R_{D_\delta}\co\Out^\I(\kG(S))\to\Out^\I(\kG(S\ssm D_\delta)).
\end{equation}

For $d(S)>1$ and $\g$ a simple closed curve on $S$ which is either nonseparating or bounds a $2$-punctured disc, $\Inn\kPG(S)_{\vec\g}$ 
acts trivially on the procyclic center $\hI_\g=\tau_\g^\ZZ$ of $\kPG(S)_{\vec\g}$. Therefore, there is a natural representation, induced by
restriction of automorphisms, $\Out^\I(\kPG(S)_\g)\to\Aut(\hI_\g)$. Composing with the homomorphism $\td{PR}_\g$
(cf.~(\ref{stabrestriction}) or~(\ref{stabrestriction2})), we get a representation $\td{\Chi}(S)_\g\co\Aut^\I(\kPG(S))\to\Aut(\hI_\g)$. 

It is easy to check that the image in $\Aut^\I(\kPG(S))$ of $\Inn\kG(S)$ is contained in the kernel of the representation $\td{\Chi}(S)_\g$. 
Therefore, again by Theorem~\ref{compurevsfull}, we get a natural character which only depends on the topological type of $\g$:
\begin{equation}\label{character}
\Chi(S)_\g\co\Out^\I(\kG(S))\to\Aut(\hI_\g)\cong\ZZ^\ast.
\end{equation}

The following theorem is a key result of this paper:

\begin{theorem}\label{functorialinj}Let us assume $\dS=\emptyset$ and $d(S)>1$. With the above definitions, we have:
\begin{enumerate}

\item For $g(S)\geq 1$, let $\g$ be a nonseparating simple closed curve on $S$. Then, there is a natural monomorphism:
\[R_\g\co\Out^\I(\kG(S))\hookra\Out^\I(\kG(S\ssm\g)).\]

\item For $n(S)\geq 2$, let $\delta$ be a simple closed curve bounding a 
$2$-punctured closed disc $D_\delta$ on $S$. Then, there is a natural monomorphism:
\[R_{D_\delta}\co\Out^\I(\hG(S))\hookra\Out^\I(\hG(S\ssm D_\delta)).\]
\end{enumerate}
\end{theorem}

For the proof of Theorem~\ref{functorialinj}, we will need a series of lemmas.

\begin{lemma}\label{stabilizervsquotient1}For $g(S)\geq 1$, let $\g$ be a nonseparating simple closed curve on $S$.
Then, the group $\Inn(\kG(S)_\g)$ identifies with a subgroup of $\Aut^\I(\kPG(S)_{\vec\g})$ and
there is a natural homomorphism:
\[\Aut^\I(\kPG(S)_{\vec\g})\to\Out^\I(\kG(S\ssm\g))\times\Aut(\hI_\g)\]
with kernel $\Inn(\kG(S)_\g)$.
\end{lemma}

\begin{proof} The first statement is clear since $\kPG(S)_{\vec\g}$ is a normal subgroup of $\kG(S)_\g$.

By \cite[Lemma~7.4]{BF}, the short exact sequence $1\to\hI_\g\to\kPG(S)_{\vec\g}\to\kPG(S\ssm\g)\to 1$ determines an exact sequence:
\[1\to\Hom(\kPG(S\ssm\g),\hI_\g)\to\Aut(\kPG(S)_{\vec\g})\to\Aut(\kPG(S\ssm\g))\times\Aut(\hI_\g).\]

By the explicit description of the image of $\Hom(\kPG(S\ssm\g),\hI_\g)$ in $\Aut(\kPG(S)_{\vec\g})$ given 
in the proof of \cite[Lemma~7.4]{BF}, we see that $\Aut^\I(\kPG(S)_{\vec\g})$ has trivial intersection with it.
Hence, the natural homomorphism $\Aut(\kPG(S)_{\vec\g})\to\Aut(\kPG(S\ssm\g))\times\Aut(\hI_\g)$ restricts to a monomorphism
$\Aut^\I(\kPG(S)_{\vec\g})\hookra\Aut^\I(\kPG(S\ssm\g))\times\Aut(\hI_\g)$. 

Since $Z(\kPG(S)_{\vec\g})=\hI_\g$, there is an isomorphism $\Inn(\kPG(S)_{\vec\g})\cong\Inn(\kPG(S\ssm\g))$ and then a natural monomorphism:
\begin{equation}\label{vecg}
\Out^\I(\kPG(S)_{\vec\g})\hookra\Out^\I(\kPG(S\ssm\g))\times\Aut(\hI_\g).
\end{equation}
By (i) of Theorem~\ref{compurevsfull}, there is a natural epimorphisms $\Out^\I(\kPG(S\ssm\g))\to\Out^\I(\kG(S\ssm\g))$ with kernel
$\Inn(\kG(S\ssm\g))\left/\Inn(\kPG(S\ssm\g))\right.$. 

For $n(S)=0$, the image of $\Inn(\kG(S)_\g)\left/\Inn(\kPG(S)_{\vec\g})\right.$ in $\Out^\I(\kPG(S\ssm\g))$ identifies with 
such kernel, which implies the lemma in this case.

For $n(S)\geq 1$, by (ii) of Theorem~\ref{compurevsfull}, there is a natural isomorphism:
\[\Out^\I(\kPG(S\ssm\g))\cong\Sigma_{n+2}\times\Out^\I(\kG(S\ssm\g)).\]
We claim that, via this isomorphism, the image of $\Out^\I(\kPG(S)_{\vec\g})$ in $\Out^\I(\kPG(S\ssm\g))$ intersects $\Sigma_{n+2}$ in 
the stabilizer $\Sigma_n\times\Sigma_2$ of the partition of the $n+2$ punctures of $S\ssm\g$ into those which have no boundary in $S$ and 
the two with boundary $\g$. This claim implies the lemma, since this stabilizer identifies with the image of $\Inn(\kG(S)_\g)$ in $\Out^\I(\kPG(S\ssm\g))$.

By Lemma~\ref{autsym}, the homomorphism $\Out^\I(\kPG(S)_{\vec\g})\to\Sigma_{n+2}$ is determined by the natural homomorphism
$\Aut^\I(\kPG(S)_{\vec\g})\to\Aut^\I(\kPG(S\ssm\g))$ and the inner action of $\Aut^\I(\kPG(S\ssm\g))$ on $\Inn\kG(S\ssm\g)$.
An element of $\Aut^\I(\kPG(S)_{\vec\g})$, modulo $\Inn\kPG(S)_{\vec\g}$, fixes any Dehn twist about a simple closed curve bounding 
an unpunctured genus $1$ subsurface of $S$ which contains $\g$. Therefore, an element of $\Aut^\I(\kPG(S\ssm\g))$ in the image
of $\Aut^\I(\kPG(S)_{\vec\g})$, modulo $\Inn\kPG(S\ssm\g)$, fixes any Dehn twist about a simple closed curve bounding a disc containing
only the two punctures of $S\ssm\g$ which are bounded by $\g$. In particular, by the proof of (iv) of Proposition~\ref{0simplices}, 
the inner action of such an element on $\Inn\kG(S\ssm\g)$ preserves the braid twist about the same curve, which proves the claim above.
\end{proof}

\begin{lemma}\label{stabilizervsquotient2}For $d(S)>1$, $g(S)\geq 0$, $n(S)\geq 2$ and $\delta$ a simple closed curve bounding a 
$2$-punctured closed disc $D_\delta$ on $S$, the group $\Inn(\kG(S)_\delta)$ identifies with a subgroup of $\Aut^\I(\kPG(S)_\delta)$ 
and there is a natural homomorphism:
\[\Aut^\I(\kPG(S)_\delta)\to\Out^\I(\kG(S\ssm D_\delta))\times\Aut(\hI_\delta)\] 
with kernel $\Inn(\kG(S)_\delta)$.
\end{lemma}

\begin{proof}The proof is essentially the same as that of Lemma~\ref{stabilizervsquotient1} but simpler.
\end{proof}

We will also need the lemma:

\begin{lemma}\label{item1}For $d(S)>1$, suppose that one of the following hypotheses is satisfied:
\begin{enumerate}
\item $g(S)\geq 1$ and $\g$ is a nonseparating simple closed curve on $S$;
\item $n(S)\geq 2$ and $\g$ is a simple closed curve on $S$ bounding a $2$-punctured disc.
\end{enumerate}
Then, the kernel of the natural homomorphism, induced by $\td{PR}_\g$ (cf.~(\ref{stabrestriction}) and~(\ref{stabrestriction2})): 
\[\Aut^\I(\kPG(S))\to\left.\Aut^\I(\kPG(S)_{\vec\g})\right/\Inn(\kG(S)_\g)\] 
is $\Inn(\kG(S))$.
\end{lemma}

\begin{proof}We already observed that $\Inn(\kG(S))$ is contained in the kernel of this homomorphism. Hence, we only need to show 
the reverse inclusion. 

Let $f\in\ker(\Aut^\I(\kPG(S))\to\left.\Aut^\I(\kPG(S)_{\vec\g})\right/\Inn(\kG(S)_\g))$. 
After composing $f$ with an element of $\Inn(\kG(S))$, we can assume
that $f\in\Aut^\I(\kPG(S))_\g$. Composing again by an element of $\Inn(\kG(S)_\g)\subset\Inn(\kG(S))$, we can then also assume
that $f$ restricts to the identity on the subgroup $\kPG(S)_{\vec\g}$. 

Since, for every $\s\in\kC(S)_{d(S)-2}$ such that $\g\in\s$, the vertices of the profinite Farey subgraph $\hF_{\s}\subset\kC_P(S)$ identify with 
subgroups of the stabilizer $\kPG(S)_{\vec\g}$, the action of $f$ on the vertices of the procongruence pants complex 
$\kC_P(S)$, induced by the natural representation $\Aut^\I(\kPG(S))\to\Aut(\kC(S))$, trivially extends to the identity map on the vertices 
of all profinite Farey subgraphs $\hF_\s\subset\kC_P(S)$ such that $\g\in\s$. 

Since $d(S)>1$, for $\g$ nonseparating, we have that $d(S\ssm\g)\geq 1$ (resp.\ $d(S\ssm D_\g)\geq 1$ for $\g$ separating). Therefore, 
the set of profinite Farey subgraphs $\hF_\s\subset\kC_P(S)$ such that $\g\in\s$ is non-empty and, from Lemma~\ref{innercriterion}, it follows 
that $f\in\Inn(\kG(S))$.
\end{proof}



\begin{proof}[Proof of Theorem~\ref{functorialinj}](i):  
By Lemma~\ref{stabilizervsquotient1} and (i) of Lemma~\ref{item1}, the kernel of the homomorphism:
$(\ol{PR}_\g,\td{\Chi}(S)_\g)\co\Aut^\I(\kPG(S))\to\Out^\I(\kG(S\ssm\g))\times\Aut(\hI_\g)$ is $\Inn(\kG(S))$,
so that, by Theorem~\ref{compurevsfull}, the natural homomorphism:
\[(R_\g,\Chi(S)_\g)\co\Out^\I(\kG(S))\to\Out^\I(\kG(S\ssm\g))\times\Aut(\hI_\g)\] 
is injective.

For $g(S)\geq 2$, the conclusion of the theorem then follows observing that the character~(\ref{character})
$\Chi(S)_\g\co\Out^\I(\kG(S))\to\Aut(\hI_\g)$ has the same kernel of the composition of the homomorphism
$R_\g\co\Out^\I(\kG(S))\to\Out^\I(\kG(S\ssm\g))$ with $\Chi(S\ssm\g)_{\g'}\co\Out^\I(\kG(S\ssm\g))\to\Aut(\hI_{\g'})$, 
where $\g'$ is any nonseparating simple closed curve on $S\ssm\g$.

For $g(S)=1$ and $n(S)=2$, through the natural isomorphism $\Out^\I(\kG(S))\cong\GT$ (cf.\ (i) of Proposition~\ref{g=1n=2}), 
the character~(\ref{character}) $\Chi(S)_\g\co\Out^\I(\kG(S))\to\Aut(\hI_\g)$ identifies with the natural character 
$\chi_\l\co\GT\to\ZZ^\ast$ on the Grothendieck-Teichm\"uller group defined by the assignment $(f,\l)\mapsto\l$. The latter
clearly factors through the homomorphism $R_\g\co\Out^\I(\kG(S))\to\Out^\I(\kG(S\ssm\g))$ and the natural isomorphism
$\Out^\I(\kG(S\ssm\g))\cong\GT$ (cf.\ Proposition~\ref{purevsfullgenus0}) and the conclusion follows as above.

For $g(S)=1$ and $n(S)> 2$, through a series of natural homomorphisms of type~(\ref{stabinducedfull2}), we can reduce to 
the case $g(S)=1$ and $n(S)=2$ treated above.
\smallskip

\noindent
(ii): 
From Lemma~\ref{stabilizervsquotient2} and (ii) of Lemma~\ref{item1}, it follows that the homomorphism:
\[(R_\delta,\Chi(S)_\delta)\co\Out^\I(\kG(S))\to\Out^\I(\kG(S\ssm D_\delta))\times\Aut(\hI_\delta)\] 
is injective. The conclusion of the theorem then follows from a similar argument as above.
\end{proof}

\subsection{Injective functorial homomorphisms}\label{injectivehoms}
There is a third natural homomorphism between outer $\I$-automorphism groups of procongruence mapping class groups 
which we need to consider. For $S$ a closed hyperbolic surface of genus $\geq 2$ and $P\in S$, let us consider the associated procongruence 
Birman short exact sequence (cf.\ \cite[Corollary~4.7]{congtop}):
\[1\to\hp_1(S,P)\to\kG(S\ssm P)\to\kG(S)\to 1,\]
where a simple generator $\g$ of $\pi_1(S,P)\subset\hp_1(S,P)$ is sent by the map $\hp_1(S,P)\to\kG(S\ssm P)$ to the bounding pair map 
$\tau_{\g_1}\tau_{\g_2}^{-1}$, where $\g_1$ and $\g_2$ are the boundary components of a tubular neighborhood of $\g$ in $S$.
In particular, the group $\hp_1(S,P)$ is (topologically) generated in $\kG(S\ssm P)$ by such bounding pair maps.

Since, by \cite[Theorem~5.5]{BF}, the action of $\Aut^\I(\kG(S\ssm P))$ on the procongruence curve complex $\kC(S\ssm P)$ preserves 
topological types of simplices, it follows, in particular, that the action of $\Aut^\I(\kG(S\ssm P))$ on $\kG(S\ssm P)$ preserves the topological 
type of bounding pair maps associated to annuli on $S$ containing the puncture and so it preserves the image of $\hp_1(S,P)$ in $\kG(S\ssm P)$. 
An element of $\Aut^\I(\kG(S\ssm P))$ then induces an automorphism on the quotient $\kG(S\ssm P)/\hp_1(S,P)\cong\kG(S)$.

Therefore, there is a natural homomorphism $\td{B}_P\co\Aut^\I(\kG(S\ssm P))\to\Aut^\I(\kG(S))$ and then a natural homomorphism:
\begin{equation}\label{Birmaninduced}
B_P\co\Out^\I(\kG(S\ssm P))\to\Out^\I(\kG(S)).
\end{equation}

In conclusion, we have constructed three types of natural homomorphisms between outer $\I$-automorphism groups of procongruence 
mapping class groups:
\begin{enumerate}
\item For $g(S)\geq 1$, a homomorphism $R_\g\co\Out^\I(\kG(S))\to\Out^\I(\kG(S\ssm\g))$.
\item For $n(S)\geq 2$, a homomorphism $R_{D_\delta}\co\Out^\I(\kG(S))\to\Out^\I(\kG(S\ssm D_\delta))$.
\item For $S$ a closed surface, a homomorphism $B_P\co\Out^\I(\kG(S\ssm P))\to\Out^\I(\kG(S))$.
\end{enumerate}

Let $S$ and $S'$ be surfaces such that $g(S)\geq g(S')$ and $\chi(S)\leq\chi(S')<0$. 
Cutting $S$ along $g(S)- g(S')$ nonseparating simple closed curves and then filling in $\chi(S')-\chi(S)$ punctures, 
we obtain a surface homeomorphic to $S'$. Composing $g(S)- g(S')$ maps of type~(i) with 
$\chi(S')-\chi(S)$ maps of type~(ii), if $n(S')>0$, or $\chi(S')-\chi(S)-1$ maps of type~(ii) and a
map of type~(iii), if $n(S')=0$, we obtain a homomorphism:
\[\mu_{S,S'}\co\Out^\I(\kG(S))\to\Out^\I(\kG(S')).\]

It is not difficult to check that changing the order of the compositions between the maps of the various types  
gives the same homomorphism and hence also that the collection of natural homomorphisms $\{\mu_{S,S'}\}$ thus obtained is functorial in the sense that, 
given surfaces $S,S'$ and $S''$ such that
$g(S)\geq g(S')\geq g(S'')$ and $\chi(S)\leq\chi(S')\leq\chi(S'')<0$, there holds $\mu_{S,S''}=\mu_{S',S''}\circ\mu_{S,S'}$. 
By Theorem~\ref{functorialinj}, Corollary~\ref{IvsIns2} and the above discussion, we have:

\begin{theorem}\label{functorialhom}Let $S$ and $S'$ be surfaces with empty boundary such that $g(S)\geq g(S')$, $\chi(S)\leq\chi(S')$ 
and $d(S')\geq 1$. Then, there is a natural and functorial monomorphism: 
\begin{equation}\label{functor}
\mu_{S,S'}\co\Out^{\I_0}(\kG(S))\hookra\Out^{\I_0}(\kG(S')).
\end{equation}
\end{theorem}

\begin{proof}We only have to show that the homomorphism $B_P\co\Out^\I(\kG(S\ssm P))\to\Out^\I(\kG(S))$ is injective.
This follows from the fact that, for $\g$ a nonseparating simple closed curve on $S\ssm P$, the composition of $B_P$ with the map
$R_\g\co\Out^\I(\kG(S))\to\Out^\I(\kG(S\ssm\g))$ equals the composition of the map $R_\g\co\Out^\I(\kG(S\ssm P))\to\Out^\I(\kG((S\ssm P)\ssm\g))$
with the map $R_{D_\delta}\co\Out^\I(\kG((S\ssm P)\ssm\g))\to\Out^\I(\kG(S\ssm\g))$, where $\delta$ is a simple closed curve on $(S\ssm P)\ssm\g$
bounding a disc containing two punctures, and we already know (cf.\ Theorem~\ref{functorialinj}) that the latter two homomorphisms are injective.
\end{proof}

\section{Automorphisms of procongruence mapping class groups}\label{GTsection}
In this section, we will prove the procongruence analogue of Ivanov theorem on automorphism groups of mapping class groups 
(cf.\ Theorem~A). We will actually prove a more precise and general formulation of this result. 

Let us recall that, by definition, the profinite Grothendieck-Teichm\"uller group comes with a natural character $\chi_\l\co\GT\to\ZZ^\ast$ 
defined by the assignment $(f,\l)\mapsto\l$ and such that  its composition with the embedding $G_\Q\subseteq\GT$ 
is the cyclotomic character. We have:

\begin{theorem}\label{GT=Out}Let $S=S_{g,n}^k$ be a hyperbolic surface with $d(S)>1$:
\begin{enumerate}
\item There is a natural isomorphism: 
\[\Out^{\I_0}(\kG(S))\cong\GT.\]
\item For $(g,n+k)\neq (1,2)$, there is a natural isomorphism:
\[\Out^{\I_0}(\kPG(S))\cong\Sigma_{n+k}\times\GT.\]
\item There is a natural isomorphism:
\[\Out^{\I}(\kG(S,\dS))\cong\GT.\]
\item The outer action of $\GT$ on $\kG(S,\dS)$ of item (iii) restricts to a genuine action on the normal inertia group 
$\prod_{i=1}^k\tau_{\d_i}^\ZZ$ with the property that each procyclic subgroup $\tau_{\d_i}^\ZZ$ is preserved and acted upon through 
the character $\chi_\l\co\GT\to\ZZ^\ast$.
\item For $S\neq S_{1,2}$, there is a natural isomorphism:
\[\Out^{\I}(\kPG(S,\dS))\cong\Sigma_n\times\Sigma_k\times\GT.\]
\end{enumerate}
\end{theorem}

\begin{remark}
In the above as in all following statements, "natural" means that the isomorphisms are compatible with the homomorphisms $\mu_{S,S'}$ introduced in
Section~\ref{injectivehoms}.
\end{remark}

\subsection{The proof of Theorem~\ref{GT=Out}}By (ii) of Theorem~\ref{compurevsfull} (cf.\ also Theorem~\ref{IvsIns} and  Corollary~\ref{IvsIns2}) 
and (ii) of Theorem~\ref{purevsfullboundary}, we have that (i)$\Ra$(ii) and (iii)$\Ra$(v). Therefore, it is enough 
to prove items (i), (iii) and (iv) of Theorem~\ref{GT=Out}. For the proof, we will need the following fundamental lemma, 
whose proof will be postponed (cf.\ Section~\ref{postponed}):

\begin{lemma}\label{GTreprcong}For a hyperbolic surface $S$, there is a natural representation:
\[\Psi_{(S,\dS)}\co\GT\to\Out^\I(\kG(S,\dS)).\]
\end{lemma}

In the next two subsections, we will show how Theorem~\ref{GT=Out} follows from Lemma~\ref{GTreprcong}.

\subsection{Lemma~\ref{GTreprcong} implies (i) of Theorem~\ref{GT=Out}}\label{LemmaImplies1}
Let us assume first that $S$ has empty boundary.
By Theorem~\ref{functorialhom}, for $d(S)>1$, there is then a natural monomorphism:
\[\mu_{S,S_{0,5}}\co\Out^{\I_0}(\kG(S))\hookra\Out^{\I_0}(\kG(S_{0,5}))=\GT.\] 
By Lemma~\ref{GTreprcong} and Corollary~\ref{IvsIns2}, there is a homomorphism $\Psi_{(S,\emptyset)}\co\GT\to\Out^{\I_0}(\kG(S))$
so that we get a series of natural homomorphisms:
\[\GT\to\Out^{\I_0}(\kG(S))\hookra\GT,\]
which proves that $\mu_{S,S_{0,5}}$ is surjective and then (i) of Theorem~\ref{GT=Out}, for $S=\ring{S}$, follows.
\smallskip

Let now $S=S_{g,n}^k$, with $k\geq 1$. By Proposition~\ref{g=1n=2} and Corollary~\ref{IvsIns2}, we already know that 
$\Out^{\I_0}(\kG(S_1^2))\cong\Out^{\I_0}(\kG(S_{1,1}^1))\cong\GT$, so that we can assume $(g,n+k)\neq(1,2)$.  
Let us also observe that, since $\kPG(S)\cong\kPG(\ring{S})$, by the case $S=\ring{S}$, just proved, of (i) of Theorem~\ref{GT=Out} 
and (ii) of Theorem~\ref{compurevsfull}, we have that $\Out^{\I_0}(\kPG(S))\cong\Sigma_{n+k}\times\GT$.
Let us then consider the short exact sequence:
\[1\to\kPG(S)\to\kG(S)\to\Sigma_n\times\Sigma_k\to 1.\]

By the above assumptions and \cite[Theorem~4.14]{BF}, we have $Z(\kPG(S))=Z(\kG(S))=\{1\}$ and, for $n,k\neq 2$, 
there also holds $Z(\Sigma_n\times\Sigma_k)=\{1\}$. Therefore, from \cite[Lemma~4.2]{MN} and Lemma~\ref{autsym}, 
it follows that there is a natural isomorphism:
\[\Out^{\I_0}(\kG(S))\cong Z_{\Out^{\I_0}(\kPG(S))}(\Inn\kG(S)\left/\Inn\kPG(S)\right.),\]
where $\Inn\kG(S)\left/\Inn\kPG(S)\right.$ is identified, via the isomorphism $\Out^{\I_0}(\kPG(S))\cong\Sigma_{n+k}\times\GT$, 
with the subgroup $\Sigma_n\times\Sigma_k$ of $\Sigma_{n+k}$. It then follows that:
\[Z_{\Out^{\I_0}(\kPG(S))}(\Inn\kG(S)\left/\Inn\kPG(S)\right.)\cong\GT,\]
proving that $\Out^{\I_0}(\kG(S))\cong\GT$, for $n,k\neq 2$.

For $n=2$ and $k\neq 2$, Lemma~\ref{autsym} and an argument similar to the proof of \cite[Lemma~4.2]{MN} imply that there is a
short exact sequence:
\[1\to\Sigma_2\to Z_{\Out^{\I_0}(\kPG(S))}(\Inn\kG(S)\left/\Inn\kPG(S)\right.)\to\Out^{\I_0}(\kG(S))\to 1.\]
In this case, there is also an isomorphism $Z_{\Out^{\I_0}(\kPG(S))}(\Inn\kG(S)\left/\Inn\kPG(S)\right.)\cong\Sigma_2\times\GT$ which
identifies the copy of $\Sigma_2$ on the right hand side with the image of $\Sigma_2$ in the above short exact sequence.  
This implies the isomorphism $\Out^{\I_0}(\kG(S))\cong\GT$ also in this case.
For $n\neq 2$ and $k=2$, we can argue exactly in the same way. 

For $n=k=2$, there is a short exact sequence:
\[1\to\Sigma_2\times\Sigma_2\to Z_{\Out^{\I_0}(\kPG(S))}(\Inn\kG(S)\left/\Inn\kPG(S)\right.)\to\Out^{\I_0}(\kG(S))\to 1\]
and an isomorphism $Z_{\Out^{\I_0}(\kPG(S))}(\Inn\kG(S)\left/\Inn\kPG(S)\right.)\cong\Sigma_2\times\Sigma_2\times\GT$,
so that, as above, we conclude that there is an isomorphism $\Out^{\I_0}(\kG(S))\cong\GT$.

\subsection{Lemma~\ref{GTreprcong} implies (iii) and (iv) of Theorem~\ref{GT=Out}}\label{LemmaImplies2}
Let us first observe that \cite[Theorem~A]{MN} and its proof imply the following special case of items (iii) and (iv) of Theorem~\ref{GT=Out}:

\begin{lemma}\label{MinaNaka}For $S=S^1_{0,n}$ and $n\geq 4$, there is a natural isomorphism:
\[\Out^{\I}(\kG(S,\dS))\cong\GT,\]
such that $\GT$ acts on the center $\tau_{\d_1}^\ZZ$ of $\kG(S,\dS)$ through the character $\chi_\l\co\GT\to\ZZ^\ast$.
\end{lemma}

\begin{proof}Note that, in the notation of \cite{MN}, we have $\kG(S,\dS)=\wh{B}_n$ and $\kG(S)=\wh{\mathcal B}_n$.
Since $\tau_{\d_1}^\ZZ$ is a characteristic subgroup of $\kG(S,\dS)$, the natural epimorphism $\kG(S,\dS)\to\kG(S)$ induces a 
homomorphism $\Aut(\kG(S,\dS))\to\Aut(\kG(S))$ which, by \cite[Theorem~4.6]{MN}, is surjective and whose kernel, as described 
in \cite[Lemma~4.5]{MN}, has trivial intersection with $\Aut^\I(\kG(S,\dS))$. Therefore, there is a natural isomorphism 
$\Aut^\I(\kG(S,\dS))\cong\Aut^\I(\kG(S))$. 

Since, by \cite[Proposition~2.2 and Proposition~3.1]{MN}, $\kPG(S)$ is a 
characteristic subgroup of $\kG(S)$, there is also a natural homomorphism $\Aut(\kG(S))\to\Aut(\kPG(S))$ which, since the
centralizer of $\kG(S)$ in $\kPG(S)$ is trivial (cf.\ for instance, \cite[Theorem~4.14]{BF}), is injective. 

By Lemma~\ref{decpreservation}, we have that $\Aut^\I(\kPG(S))=\Aut(\kPG(S))$, which then implies that 
we have $\Aut^\I(\kG(S))=\Aut(\kG(S))$ as well. By \cite[Theorem~4.3, (i)]{MN}, there
is a natural isomorphism $\Out(\kG(S))\cong\GT$. Combined with the previous isomorphisms, this yields the isomorphism, claimed in the lemma:
$\Out^{\I}(\kG(S,\dS))\cong\GT$. 

The last claim of the lemma then follows from the explicit description of this action given, for instance, in the appendix of \cite{HS}.
\end{proof}

Let us now observe that, given a nonseparating simple closed curve $\g$ on $S$, by the same construction of 
the homomorphism~(\ref{stabinducedfull1}), where we use Theorem~\ref{relativestabilizers} and Theorem~\ref{normalizers multitwists}, 
instead of Theorem~\ref{stabilizers} and \cite[Corollary~6.2]{[B3]}, and Theorem~\ref{purevsfullboundary} instead of 
Theorem~\ref{compurevsfull}, we get a natural homomorphism:
\[R_\g\co\Out^\I(\kG(S,\dS))\to\Out^\I(\kG(S\ssm\g,\dS)).\]
For $S=S_{g,n}^k$, by composing $g$ of such homomorphisms, we get a natural homomorphism:
\begin{equation}\label{stabinducedfullrel1}
R_0\co\Out^\I(\kG(S,\dS))\to\Out^\I(\kG(S_{0,n+2g}^k,\dS)).
\end{equation}

\begin{lemma}\label{indrelinj}For $d(S)>1$, the homomorphism $R_0$~(\ref{stabinducedfullrel1}) is injective.
\end{lemma}

\begin{proof}There is a natural commutative diagram:
\[\begin{array}{ccc}
\Out^\I(\kG(S,\dS)) &\to &\Out^\I(\kG(S))\times\Aut(A)\\
\;\downarrow\!\!{\scriptscriptstyle R_0}  & &\downarrow  \\
\Out^\I(\kG(S_{0,n+2g}^k,\dS))&\to & \Out^\I(\kG(S_{0,n+2g}^k))\times\Aut(A),
\end{array} \]
where $A:=\prod_{i=1}^k\tau_{\delta_i}^\ZZ$ identifies with the normal inertia subgroups of both groups
$\kG(S,\dS)$ and $\kG(S_{0,n+2g}^k,\dS)$. By (i) of Lemma~\ref{injectivehom}, we have that both horizontal homomorphisms
are injective. From the isomorphisms $\Out^\I(\kG(S))\cong\GT$ and $\Out^\I(\kG(S_{0,n+2g}^k))\cong\GT$
(cf.\ Section~\ref{LemmaImplies1}), it then follows that also the righthand vertical homomorphism is injective and so
the lefthand vertical homomorphism $R_0$ is injective too.
\end{proof}

Let $R_\emptyset\co\Out^{\I}(\kG(S,\dS))\to\Out^{\I}(\kG(S))$ be the natural homomorphism  defined in Section~\ref{autrelprocong}. We have:

\begin{lemma}\label{relinjgenus0}For $S=S_{0,n}^k$ with $n+k\geq 5$ and $k\geq 1$, the natural homomorphism 
$R_\emptyset\co\Out^{\I}(\kG(S,\dS))\to\Out^{\I}(\kG(S))$ is injective.
\end{lemma}

\begin{proof}By (ii) of Theorem~\ref{purevsfullboundary}, the group $\Out^{\I}(\kG(S,\dS))$ identifies with the 
centralizer of $\Inn(\hG(S,\dS))\left/\Inn(\hPG(S,\dS))\right.\cong\Sigma_n\times\Sigma_k$ in $\Out^{\I}(\kPG(S,\dS))$. 
This implies that the action of $\Out^{\I}(\kG(S,\dS))$ on $\prod_{i=1}^k\tau_{\d_i}^\ZZ$ preserves the procyclic subgroup 
$\tau_{\d_i}^\ZZ$, for $i=1,\ldots,k$. Moreover, since these subgroups are all conjugated by the action of $\Inn(\hG(S,\dS))$,
one of this action, say the one on $\tau_{\d_1}^\ZZ$, determines them all. From (i) of Lemma~\ref{injectivehom}, 
it then follows that there is a natural monomorphism:
\begin{equation}\label{monoproduct}
\Out^{\I}(\kG(S,\dS))\hookra\Aut(\tau_{\d_1}^\ZZ)\times\Out^{\I}(\kG(S)).
\end{equation}

By the remarks above, the outer action of $\Out^{\I}(\kG(S,\dS))$ on $\kPG(S,\dS)$ 
(cf.\ item (ii) of Theorem~\ref{purevsfullboundary}) preserves the kernel of the epimorphism $\kPG(S,\dS)\to\kPG(S',\dS')$, where
we let $S':=S_{0,n+k-1}^1$. Therefore, the latter epimorphism induces a homomorphism $\Out^{\I}(\kG(S,\dS))\to\Out^{\I}(\kPG(S',\dS'))$ and then,
by (ii) of Theorem~\ref{purevsfullboundary}, a homomorphism:
\begin{equation}\label{onecomponent}
\Out^{\I}(\kG(S,\dS))\to\Out^{\I}(\kG(S',\dS')).
\end{equation}

Let us observe that the component $\Out^{\I}(\kG(S,\dS))\to\Aut(\tau_{\d_1}^\ZZ)$ of the monomorphism~(\ref{monoproduct})  
factors through the homomorphism~(\ref{onecomponent}). By Lemma~\ref{MinaNaka}, this is then determined by the character 
$\chi_\l\co\GT\to\ZZ^\ast$. Hence, since, by the proof of item (i) of Theorem~\ref{GT=Out} in Section~\ref{LemmaImplies1}, 
we already know that $\Out^{\I}(\kG(S))\cong\GT$, it follows that the component $\Out^{\I}(\kG(S,\dS))\to\Out^{\I}(\kG(S))$ of the 
monomorphism~(\ref{monoproduct}) is injective. 
\end{proof}

For $S=S_{g,n}^k$ and $d(S)>1$, by Lemma~\ref{indrelinj}, Lemma~\ref{relinjgenus0} and the results of Section~\ref{LemmaImplies1}, 
there is a natural monomorphism:
\[R_\emptyset\circ R_0\co\Out^\I(\kG(S,\dS))\hookra\Out^\I(\kG(S_{0,n+2g}^k))\cong\GT.\]
Thus, by Lemma~\ref{GTreprcong}, we get a series of natural monomorphisms:
\[\GT\hookra\Out^{\I}(\kG(S,\dS))\hookra\GT,\]
from which we conclude that $\Out^{\I}(\kG(S,\dS))\cong\GT$. Item (iv) of Theorem~\ref{GT=Out} follows from the proof above as well.

\subsection{The proof of Lemma~\ref{GTreprcong}}\label{postponed}
In the following sections, we will first prove a version of Lemma~\ref{GTreprcong} for profinite mapping class groups in genus $\leq 2$.
By the congruence subgroup property in genus $\leq 2$, this will complete the proof of Lemma~\ref{GTreprcong} and thus of 
Theorem~\ref{GT=Out} for the genus $\leq 2$ case. 

For higher genus, we will need first to answer an old open question in Grothendieck-Teichm\"uller theory, namely whether there exists 
a natural $\GT$-action on the profinite mapping class group associated to every hyperbolic surface.

Let $\Aut^{\I_0}(\hPG(S))$ be the closed subgroup of $\Aut(\hPG(S))$ consisting of those elements which preserve the conjugacy 
class of a procyclic subgroup (topologically) generated by a Dehn twists about a nonseparating simple closed curve
and by $\Out^{\I_0}(\hPG(S))$ its quotient by the subgroup of inner automorphisms. We will prove that:

\begin{theorem}\label{GTrepr}For $S$ a hyperbolic surface such that $d(S)>1$, there is a natural faithful representation: 
\[\hat{\rho}_{\GT}\co\GT\hookra\Out^{\I_0}(\hPG(S,\dS)).\]
This outer representation restricts to an action by automorphisms on the central inertia group $\prod_{i=1}^k\tau_{\d_i}^\ZZ$ of $\hPG(S,\dS)$ 
with the property that each procyclic subgroup $\tau_{\d_i}^\ZZ$ is preserved and acted upon through the character $\chi_\l\co\GT\to\ZZ^\ast$, 
for $i=1,\ldots,k$.
\end{theorem}

The proof of Lemma~\ref{GTreprcong} will then be obtained as a corollary of Theorem~\ref{GTrepr} (or better a corollary of its proof).

\subsection{Profinite hyperelliptic mapping class groups}\label{hypsection}
A \emph{hyperelliptic hyperbolic surface} $(S,\u)$ is the datum of a hyperbolic surface $S$ and a hyperelliptic involution $\u$ on $S$,
that is to say, an involution such that the quotient surface $S_\u:=S/\langle\u\rangle$ has genus $0$.
The \emph{hyperelliptic mapping class group $\U(S)$} is then defined to be the centralizer of $\u$ in the mapping class group $\G(S)$. 
For $\cP$ a finite subset of $S$ and $\vec\cP$ the same subset with a fixed order of its elements, we then let $\G(S,\cP)\cong\G(S\ssm\cP)$ 
be the mapping class group of the marked surface $(S,\cP)$ and $\G(S,\vec\cP)$ be the kernel of the natural representation $\G(S,\cP)\to\Sigma_\cP$.

The \emph{hyperelliptic mapping class groups $\U(S,\cP)$ and $\U(S,\vec\cP)$} are, respectively, the inverse images of 
$\U(S)$ by the epimorphisms $\G(S,\cP)\to\G(S)$ and $\G(S,\vec\cP)\to\G(S)$. 

The \emph{profinite hyperelliptic mapping class groups} $\hU(S,\cP)$ and $\hU(S,\vec\cP)$ are then, respectively, the profinite completions 
of $\U(S,\cP)$ and $\U(S,\vec\cP)$. 
The congruence topology on $\G(S,\cP)\cong\G(S\ssm\cP)$ induces a profinite topology on $\U(S,\cP)$ and $\U(S,\vec\cP)$ via the embeddings
$\U(S,\vec\cP)\subseteq\U(S,\cP)\subseteq\G(S,\cP)$ also called the congruence topology. By \cite[Theorem~7.2]{HyperNotes}, this
topology coincides with the full profinite topology. Hence, in the sequel, we will identify profinite and procongruence hyperelliptic mapping class groups.

\subsection{The complex of symmetric profinite curves}\label{symcomplex}
A \emph{symmetric multicurve} on a marked hyperelliptic surface $(S,\u,\cP)$ is a multicurve on $S\ssm\cP$ whose isotopy class in $S$
is preserved by the hyperelliptic involution $\u$. The \emph{complex of symmetric curves} $C(S,\u,\cP)$ is the (abstract) simplicial
complex whose simplices are the symmetric multicurves on $(S,\u,\cP)$ (cf.\ \cite[Definition~2.7]{HyperNotes}). 
This is clearly a full subcomplex of the curve complex $C(S\ssm\cP)$. We then define the \emph{complex of symmetric profinite curves} 
$\hC(S,\u,\cP)$ on $(S,\u,\cP)$ to be the (abstract) simplicial profinite complex whose (profinite) set of $k$-simplices is the closure of
$C(S,\u,\cP)_k$ inside $\kC(S\ssm\cP)_k$ (cf.\ \cite[Definition~7.11]{HyperNotes}, where this complex is denoted by $L(\hP(S\ssm\cP),\u)$).

Simplices of $C(S,\u,\cP)$ parameterize the multitwists of the hyperelliptic mapping class group $\U(S,\cP)$ as well as the abelian subgroups 
generated by Dehn twists. In \cite[Section~7.7]{HyperNotes}, we showed that the simplices of $\hC(S,\u,\cP)$ parameterize profinite 
multitwists in the profinite hyperelliptic mapping class group $\hU(S,\cP)$.
In particular, for $U$ an open subgroup of $\hU(S,\cP)$ (which is then a closed subgroup of $\kG(S\ssm\cP)$), the nontrivial decomposition
and inertia groups $U_\s$ and $\hI_\s(U)$ of $U$ (cf.\ Definition~\ref{DecPres} and Definition~\ref{starcondition}) are parameterized by $\hC(S,\u,\cP)$.
The same arguments of the proofs of Proposition~7.2 and \cite[Theorem~7.3, (ii)]{BF} show that there is a natural continuous 
homomorphism $\Aut^\D(U)\to\Aut(\hC(S,\u,\cP))$ whose kernel identifies with the abelian group $\Hom(U/Z(U),Z(U))$. The usual argument then shows
that this has trivial intersection with the subgroup $\Aut^\I(U)$ of $\Aut^\D(U)$, so that there is a monomorphism $\Aut^\I(U)\hookra\Aut(\hC(S,\u,\cP))$.

It is likely that a version of \cite[Theorem~5.5]{BF} holds in the hyperelliptic case. This would show,
in particular, that the $\I$-automorphisms of $U$ preserve the topological types of the inertia groups $\hI_\s(U)$, for $\s\in\hC(S,\u,\cP)$.
Since we will not need this result in the sequel, we instead give the following definition:

\begin{definition}\label{starconditionhyp}For $H$ a closed subgroup of $\hU(S,B)$, we let $\Aut^{\I_\mathrm{top}}(H)$ 
be the closed subgroup of $\Aut^\I(H)$ consisting of those elements which preserve the topological types of the inertia groups of $H$. 
We then say that a subgroup $K$ of $H$ is \emph{$\I_\mathrm{top}$-characteristic} if it is preserved by $\Aut^{\I_\mathrm{top}}(H)$.
Note that $\Inn(U)\leq\Aut^{\I_\mathrm{top}}(U)$, for $U$ an open subgroup of $\hU(S,\cP)$, so that, in particular, an 
$\I_\mathrm{top}$-characteristic subgroup of $U$ is normal.
\end{definition}

For the closed surface case, we have the following result: 

\begin{proposition}\label{outhyp}For $S$ a closed surface of genus $g\geq 2$, endowed with a hyperelliptic involution $\u$, 
there is a natural isomorphism $\Out(\hU(S))\cong\{\pm 1\}\times\Out^\I(\hU(S))$ and a series of natural isomorphisms:
\begin{equation}\label{GThypiso}
\Out^\I(\hU(S))=\Out^{\I_0}(\hU(S))\cong\Out(\hG_{0,[2g+2]})\cong\GT.
\end{equation}
In particular, we have $\Aut^\I(\hU(S))=\Aut^{\I_\mathrm{top}}(\hU(S))$.
\end{proposition}

\begin{proof}By \cite[Lemma~7.4]{BF}, there is an exact sequence:
\[1\to\Hom(\hU(S)/Z(\hU(S)),Z(\hU(S)))\to\Out(\hU(S))\to\Out(\hU(S)/Z(\hU(S))),\]
where $\hU(S)/Z(\hU(S))\cong\hG_{0,[2g+2]}$ and the center $Z(\hU(S))$ is generated by the hyperelliptic involution $\u$. 
Since the abelianization of $\hG_{0,[2g+2]}$ is isomorphic to $\Z/2(2g+1)$ (cf.\ \cite[\S~5.1.3]{FM}), 
we have that $\Hom(\hU(S)/Z(\hU(S)),Z(\hU(S)))\cong\{\pm 1\}$. 

By the usual argument, the image of $\Hom(\hU(S)/Z(\hU(S)),Z(\hU(S)))$ in $\Out(\hU(S))$ has trivial intersection with the subgroups
$\Out^\I(\hU(S))$ and $\Out^{\I_0}(\hU(S))$ of the latter group. 

Since $\hU(S)/Z(\hU(S))=\hG_{0,[2g+2]}$, from Lemma~\ref{decpreservation} and \cite[Lemma~6.7]{BF2}, it then follows that
the image of $\Out(\hU(S))$ in $\Out(\hU(S)/Z(\hU(S)))$ identifies with the images there of both its subgroups $\Out^\I(\hU(S))$ and $\Out^{\I_0}(\hU(S))$. 
This implies that there is an isomorphism $\Out(\hU(S))\cong\{\pm 1\}\times\Out^\I(\hU(S))$ and also the identity $\Out^\I(\hU(S))=\Out^{\I_0}(\hU(S))$.

Let $(S^\circ,\u)$ be the hyperelliptic surface with boundary obtained from the hyperelliptic surface $(S,\u)$ by replacing one of the Weierstrass points 
with a circle. The Birman-Hilden theorem states that there is a natural isomorphism $Z_{\G(S^\circ,\dd S^\circ)}(\u)\cong B_{2g+1}$, where $B_{2g+1}$ 
is Artin braid group. Moreover, if $z$ is the standard generator of the (cyclic) center of $B_{2g+1}$, this isomorphism induces an isomorphism
$\U(S)\cong  B_{2g+1}/\langle z^2\rangle$ (cf.\ \cite[Chapter~9]{FM}).

In \cite[Appendix~4]{Ihara}, Ihara showed that there is an action of the profinite Grothendieck-Teichm\"uller group $\GT$ 
on the profinite completion $\wh{B}_n$ of the Artin braid group $B_n$, for all $n\geq 3$, which preserves the conjugacy class of the procyclic 
subgroups generated by the standard braids. The Birman-Hilden theorem then implies that there is a representation
$\GT\to\Out^{\I_0}(\hU(S))$ compatible with the natural monomorphism $\Out^{\I_0}(\hU(S))\hookra\Out(\hU(S)/Z(\hU(S)))$ and the isomorphism
$\Out(\hU(S)/Z(\hU(S)))\cong\GT$. This yields the series of isomorphisms~\eqref{GThypiso}.
The last claim then follows from the isomorphism $\Out^\I(\hU(S))\cong\Out(\hG_{0,[2g+2]})$ and Lemma~\ref{decpreservation}.
\end{proof}

\subsection{Modular subgroups of the hyperelliptic mapping class group}
The definition of hyperelliptic mapping class groups makes sense also when $(S,\u,\cP)$ is a marked disconnected hyperelliptic surface 
(cf.\ \cite[Section~2.8 and 2.9]{HyperNotes}). For instance, for $\g$ a symmetric simple closed curve on $(S,\u,\cP)$, the marked 
hyperelliptic surface $(S\ssm\g,\u,\cP)$ is a possibly disconnected marked hyperelliptic surface.
Let us label by $Q_+,Q_-$ the pair of punctures on $S\ssm\g$ bounded by $\g$ and denote by 
$\U(S\ssm\g,\cP)_{\{Q_+,Q_-\}}$ the stabilizer of this pair of punctures for the action of $\U(S\ssm\g,\cP)$ on the set of all punctures of $S\ssm\g$.
The stabilizer $\U(S,\cP)_\g$ for the action of the hyperelliptic mapping class group $\U(S,\cP)$ on the set $C(S,\u,\cP)_0$ 
is then described by the short exact sequence (cf.\ the short exact sequence~(14) in \cite{HyperNotes}):
\[1\to\tau_\g^\Z\to\U(S,\cP)_\g\to\U(S\ssm\g,\cP)_{\{Q_+,Q_-\}}\to 1.\]
A similar short exact sequence also describes the stabilizer $\U(S,\vec\cP)_\g$.

In \cite[Section~7.6]{HyperNotes}, the above description of stabilizers is extended to profinite hyperelliptic mapping class group.
More precisely, for $\g$ a symmetric simple closed curve on $(S,\u,\cP)$, which we now regard as an element of the profinite set $\hC(S,\u,\cP)_0$,
the stabilizer $\hU(S,\cP)_\g$ for the action of $\hU(S,\cP)$ on $C(S,\u,\cP)_0$  is described by the short exact sequence 
(cf.\ \cite[Corollary~7.8]{HyperNotes}):
\[1\to\tau_\g^\ZZ\to\hU(S,\cP)_\g\to\hU(S\ssm\g,\cP)_{\{Q_+,Q_-\}}\to 1,\]
where $\hU(S,\cP)_\g$ and $\hU(S\ssm\g,\cP)_{\{Q_+,Q_-\}}$ are naturally isomorphic to the profinite completions 
of the groups $\U(S,\cP)_\g$ and $\U(S\ssm\g,\cP)_{\{Q_+,Q_-\}}$, respectively. Again, a similar description holds 
for the stabilizer $\hU(S,\vec\cP)_\g$.

For the applications in this paper, we are more interested in the case when $\g$ is a nonseparating simple closed curve and the set of marked
points $\cP$ is ordered. Therefore, we will give a more precise description of the stabilizer in this case.

For $S$ a closed surface and $\g$ a nonseparating symmetric simple closed curve on $(S,\u,\cP)$, let $S_\g:=S\ssm\g$ 
and let $\U(S_\g,\vec\cP)_\circ$ be the index $2$ subgroup of $\U(S_\g,\vec\cP)$ consisting of those elements 
which do not swap the punctures of $S_\g$ labeled by $Q_+$ and $Q_-$ (this is also described as the kernel 
of the natural representation $\U(S_\g,\vec\cP)\to\Sigma_{\{Q_+,Q_-\}}$). The stabilizer
$\U(S,\vec\cP)_{\vec\g}$ of the oriented simple closed curve $\vec\g$ is described by the short exact sequence:
\[1\to\tau_\g^\Z\to\U(S,\vec\cP)_{\vec\g}\to\U(S_\g,\vec\cP)_\circ\to 1.\]

Let $N_\g$ be the normal subgroup of $\U(S_\g,\vec\cP)_\circ$ generated by the Dehn twists about symmetric separating simple closed curves on 
$(S_\g,\u,\cP)$ bounding a disc which contains no marked points and only the two punctures labeled by $Q_+$ and $Q_-$ and let $\ol{S}_\g$ be 
the surface obtained from $S_\g$ by filling in the the punctures on $S_\g$ labeled by $Q_+$ and $Q_-$ with two points (which we also label by 
$Q_+$ and $Q_-$). The group $\U(S_\g,\vec\cP)_\circ$ then fits in the short exact sequence (cf.\ \cite[Theorem~3.1]{HyperNotes}):
\[1\to N_\g\to\U(S_\g,\vec\cP)_\circ\to\U(\ol{S}_\g,\ovr{\cP\cup Q_+})\to 1.\]
From the above results, it follows:

\begin{proposition}\label{hyperellipticstab}Let $(S,\u,\cP)$ be a marked hyperelliptic hyperbolic surface and $\g$ 
a symmetric nonseparating simple closed curve on $(S,\u,\cP)$. With the above notations, we have:
\begin{enumerate}
\item The stabilizer $\hU(S,\vec\cP)_\g$ for the action of $\hU(S,\vec\cP)$ on the profinite set $\hC(S,\u,\cP)_0$ 
is described by the two short exact sequences:
\[\begin{array}{ll}
&1\to\hU(S,\vec\cP)_{\vec\g}\to\hU(S,\vec\cP)_\g\sr{\epsilon}{\to}\{\pm 1\}\to 1\\
\mbox{and }&\\
&1\to\tau_\g^\ZZ\to\hU(S,\vec\cP)_{\vec\g}\to\hU(S_\g,\vec\cP)_\circ\to 1,
\end{array}\]
where $\epsilon$ is the orientation character associated to the curve $\g$ and $\hU(S_\g,\vec\cP)_\circ$ is the kernel of the natural 
representation $\hU(S_\g,\vec\cP)\to\Sigma_{\{Q_+,Q_-\}}$.  
\item The profinite group $\hU(S_\g,\vec\cP)_\circ$ fits in the short exact sequence:
\[1\to \ol{N}_\g\to\hU(S_\g,\vec\cP)_\circ\to\hU(\ol{S}_\g,\ovr{\cP\cup Q_+})\to 1,\]
where $\ol{N}_\g$ is the closure of the group $N_\g$ in the profinite group $\hU(S_\g,\cP)$.
\end{enumerate}
\end{proposition}

\subsection{Homomorphisms between outer automorphism groups of profinite hyperelliptic mapping class groups}
Let $(S,\u,\cP)$ be a marked hyperelliptic closed surface of genus $g\geq 2$ and $\g$ a symmetric nonseparating simple closed curve on $(S,\u,\cP)$.
The subgroup $\hU(S,\vec\cP)_{\vec\g}$ of $\hU(S,\vec\cP)_\g$ is $\I_\mathrm{top}$-characteristic.
By the same construction of the homomorphism~(\ref{stabrestriction}), there is then a natural homomorphism:
\[\Aut^{\I_\mathrm{top}}(\hU(S,\vec\cP))\to\Out^{\I_\mathrm{top}}(\hU(S,\vec\cP)_{\vec\g}).\]

From \cite[Theorem~7.13]{HyperNotes}, for $g\geq 2$, it follows that $Z(\hU(S_\g,\vec\cP)_\circ)=\{1\}$. Hence,
by (i) of Proposition~\ref{hyperellipticstab}, we have $Z(\hU(S,\vec\cP)_{\vec\g}))=\tau_\g^\ZZ$ and there is a natural homomorphism:
\[\Out^{\I_\mathrm{top}}(\hU(S,\vec\cP)_{\vec\g})\to\Out^{\I_\mathrm{top}}(\hU(S_\g,\vec\cP)_{\circ})\] 
By composing the two above homomorphisms, we then get the homomorphism:
\begin{equation}\label{hyprestriction2}
\Aut^{\I_\mathrm{top}}(\hU(S,\vec\cP))\to\Out^{\I_\mathrm{top}}(\hU(S_\g,\vec\cP)_{\circ}).
\end{equation}

The subgroup $\ol{N}_\g$ of $\hU(S_\g,\vec\cP)_{\circ}$ (cf.\ (ii) of Proposition~\ref{hyperellipticstab}) is (topologically) generated by 
the profinite Dehn twists about symmetric profinite curves on $(S_\g,\u,\cP)$ of topological type a simple closed curve bounding a disc which
only contains the two punctures labeled by $Q_+$ and $Q_-$. Therefore, $\ol{N}_\g$ is an ${\I_\mathrm{top}}$-characteristic 
subgroup of $\hU(S_\g,\vec\cP)_{\circ}$ and there is
a natural homomorphism $\Out^{\I_\mathrm{top}}(\hU(S_\g,\vec\cP)_\circ)\to\Out^{\I_\mathrm{top}}(\hU(\ol{S}_\g,\ovr{\cP\cup Q_+}))$. 
Composing the latter with the homomorphism~(\ref{hyprestriction2}), we get a natural homomorphism:
\[\Aut^{\I_\mathrm{top}}(\hU(S,\vec\cP))\to\Out^{\I_\mathrm{top}}(\hU(\ol{S}_\g,\ovr{\cP\cup Q_+})).\]
It is not difficult to see that the kernel of this homomorphism contains $\Inn(\hU(S,\vec\cP))$. Therefore, there is a natural homomorphism:
\begin{equation}\label{hyprestriction3}
\Out^{\I_\mathrm{top}}(\hU(S,\vec\cP))\to\Out^{\I_\mathrm{top}}(\hU(\ol{S}_\g,\ovr{\cP\cup Q_+})).
\end{equation}

\begin{proposition}\label{GThypaction}For $(S,\u,\cP)$ a marked hyperelliptic closed surface of genus $g\geq 2$,
there is a natural representation: 
\[\Phi_{(S,\vec\cP)}\co\GT\to\Out^{\I_\mathrm{top}}(\hU(S,\vec\cP)).\]
\end{proposition}

\begin{proof}Let $n:=\sharp \cP$ and let $\s:=\{\g_1,\ldots,\g_n\}$ be a set of disjoint symmetric closed curves on a hyperelliptic closed surface $(S',\u)$ 
of genus $g+n$ such that $S'_\s:=S'\ssm\s$ is connected. Then, $(S'_\s,\u)$ is a hyperelliptic surface endowed with $n$ pairs of symmetric punctures 
(with respect to the hyperelliptic involution $\u$). Let $(S,\u)$ be the hyperelliptic closed surface obtained filling the $2n$ punctures on $S'_\s$ with the
set of points $\cP\cup\u(\cP)$. Then, composing the $n$ homomorphism of type~(\ref{hyprestriction3}) associated, recursively, to the set of disjoint 
symmetric closed curves $\s$ (note that each step adds a marked point out of the set $\cP$), by Proposition~\ref{outhyp}, we get the representation:
\[\Phi_{(S,\vec\cP)}\co\GT\cong\Out^{\I_\mathrm{top}}(\hU(S'))\to\Out^{\I_\mathrm{top}}(\hU(S,\vec\cP)).\]
\end{proof}

\begin{corollary}\label{GTgenus2action}For $S=S_{2,n}$ and all $n\geq 0$, there is a natural representation: 
\[\Psi_{S}\co\GT\to\Out^\I(\hG(S)).\]
\end{corollary}

\begin{proof}Note that, for $S'$ a closed surface of genus $2$ and $\cP$ a set of distinct $n$ points on $S'$, we have 
$\hPG(S'\ssm\cP)\cong\hU(S',\vec\cP)$. Let then $S:=S'\ssm\cP$ and compose the homomorphism $\Phi_{(S',\vec\cP)}$ of Proposition~\ref{GThypaction}
with the homomorphism $\Out^{\I_\mathrm{top}}(\hPG(S))=\Out^\I(\hPG(S))\to\Out^\I(\hG(S))$ 
(cf.\ (ii) of Proposition~\ref{0simplices} and (ii) of Theorem~\ref{compurevsfull}).
\end{proof}

\subsection{Proof of Lemma~\ref{GTreprcong} for $g(S)\leq 2$ and $\dS=\emptyset$}\label{proofGTreprcong}
The genus $2$ case was already treated in Corollary~\ref{GTgenus2action}.
The genus $0$ case, for $n(S)\geq 5$, follows from Proposition~\ref{purevsfullgenus0} and Lemma~\ref{decpreservation}. 
The case $g(S)=0$ and $n(S)=4$ then follows from (ii) of Theorem~\ref{functorialinj}. 

For the genus $1$ case, let $S:=S_{2,n}$ and $S':=S_{1,n+1}$, with $n\geq 0$. By Theorem~\ref{functorialhom}, there is then a natural 
injective homomorphism $\mu_{S,S'}\co\Out^{\I}(\hG(S))\hookra\Out^{\I}(\hG(S'))$,
which, precomposed with $\Psi_S$ (cf.\ Corollary~\ref{GTgenus2action}) yields the homomorphism
$\Psi_{(S',\emptyset)}\co\GT\to\Out^{\I}(\hG(S'))$.

\subsection{Proof of Lemma~\ref{GTreprcong} for $g(S)\leq 2$ and $\dS\neq\emptyset$}\label{withboundary}
Let $S\subset \tS\cong S_{g,n+2k}$ and $\s=\{\delta_1,\ldots,\delta_k\}$ be as in Definition~\ref{tildedS}. 
The same construction as that of the homomorphism~(\ref{stabrestriction3}) provides a natural homomorphism:
\[PR_{\vec\s}\co\Out^\I(\hPG(\tS))\to\Out^\I(\hPG(\tS)_{\vec\s}).\]
After identifying $\PG(S,\dS)$ with $\PG(\tS)_{\vec\s}$, we then get a natural homomorphism:
\[PR_{(S,\dS)}\co\Out^\I(\hPG(\tS))\to\Out^\I(\hPG(S,\dS)).\]
Composing $PR_{(S,\dS)}$ with the natural homomorphism $\Out^\I(\hPG(S,\dS))\to\Out^\I(\hG(S,\dS))$ (cf.\ Theorem~\ref{purevsfullboundary}),
we get a homomorphism $\Out^\I(\hPG(\tS))\to\Out^\I(\hG(S,\dS))$, whose kernel contains 
$\Inn(\G(\tS))\left/\Inn(\PG(\tS))\right.=\Inn(\hG(\tS))\left/\Inn(\hPG(\tS))\right.$ and thus, 
by Theorem~\ref{compurevsfull}, induces a natural homomorphism:
\begin{equation}\label{restrwithboundary}
R_{(S,\dS)}\co\Out^{\I}(\hG(\tS))\to\Out^{\I}(\hG(S,\dS)).
\end{equation}
For $g(\tS)\leq 2$, precomposing $R_{(S,\dS)}$ with the representation $\Psi_{(\tS,\emptyset)}$ obtained in Section~\ref{proofGTreprcong}, 
we finally get the representation:
\[\Psi_{(S,\dS)}\co\GT\to\Out^{\I}(\hG(S,\dS)).\]

\subsection{Proof of Theorem~\ref{GTrepr} for $g(S)\leq 2$}\label{proofcomplete}
This, in particular, follows from the case $g(S)\leq 2$ of Theorem~\ref{GT=Out}, which, as we showed in Section~\ref{LemmaImplies2}, in its turn,
follows from the case $g(S)\leq 2$ of Lemma~\ref{GTreprcong} proved in Section~\ref{proofGTreprcong} and Section~\ref{withboundary} above.

\subsection{A presentation for the profinite relative mapping class group of a surface}
In \cite{Gervais}, Gervais gave a finite presentation of the pure mapping class group of a surface with boundary. In this section, we will show
that this result implies that the profinite pure mapping class group of an open surface admits a presentation as a quotient of a profinite 
amalgamated product of stabilizers of oriented nonseparating simple closed curves by a few simple commutator relations. 

For a finite set of groups $G_1,\ldots, G_k$ and monomorphisms $H\hookra G_i$, for $i=1,\ldots,k$, we denote by 
$\bigoasterisk_H^{i=1,\ldots k}G_i$ their amalgamated free product with amalgamated subgroup $H$.
If the groups $G_1,\ldots, G_k,H$ are profinite and the monomorphisms $H\hookra G_i$ continuous, we denote by 
$\coprod_H^{i=1,\ldots k}G_i$ their amalgamated free profinite product with amalgamated subgroup $H$
(cf.\ \cite[Section~9.2]{RZ}). With the above notations, we have:

\begin{proposition}\label{amalgamatedproduct}Let $S=S_{g}^k$ with $k\geq1$ and $g\geq 3$, let $\s=\{\g_1,\ldots,\g_{g-1}\}$ be a set 
of disjoint simple closed curves on $S$ such that $S\ssm\s$ is connected and $\s'=\{\beta_1,\ldots,\beta_{g-1}\}$ a set of disjoint simple closed 
curves such that $S\ssm\s'$ is connected and the intersection $\beta_i\cap\g_j$ is a single point for $i=j$ and empty otherwise. 
Let $\s_i:=\s\ssm\{\g_i\}$, for $i=1,\ldots,g-1$. 

There is then a natural epimorphism:
\[\Theta\co\bigoasterisk_{\PG(S,\dS)_{\vec\s}}^{i=1,\ldots,g-1}\PG(S,\dS)_{\vec{\s}_i}\twoheadrightarrow\PG(S,\dS),\]
whose kernel is normally generated by the commutators $[\tau_{\beta_i},\tau_{\beta_j}]$, for $i,j=1,\ldots,g-1$.
\end{proposition}

\begin{proof}The proposition is a consequence of \cite[Theorem~1]{Gervais}. Let the curves $\{\beta_1,\ldots,\beta_{g-1}\}$
in the statement of the proposition correspond to the curves denoted with the same letters in \cite[Fig.\ 1]{Gervais}. Let instead $\g_i$
be a simple closed curve contained in the handle of the surface in \cite[Fig.\ 1]{Gervais} through which $\beta_i$ passes and such 
that $\g_i$ intersects $\beta_i$ in a single point, for $i=1,\ldots,g-1$. In particular, $\g_i$ does not intersect any other curve in the
set ${\mathscr G}_{g,n}$ defined in \cite{Gervais}, for $i=1,\ldots,g-1$.

Let us denote by $G_\Theta$ the quotient of the group $\bigoasterisk_{\PG(S,\dS)_{\vec\s}}^{i=1,\ldots,g-1}\PG(S,\dS)_{\vec{\s}_i}$
by the subgroup normally generated by the commutators $[\tau_{\beta_i},\tau_{\beta_j}]$, for $i,j=1,\ldots,g-1$.
Since the group $\PG(S,\dS)$ is generated by the union of the subgroups $\PG(S,\dS)_{\vec{\s}_i}$, 
for $i=1,\ldots,g-1$, we have that $\Theta$ is surjective and descends to an epimorphism $\ol{\Theta}\co G_\Theta\to\PG(S,\dS)$, 
such that its restriction to the subgroup $\PG(S,\dS)_{\vec{\s}_i}$ is injective for $i=1,\ldots,g-1$.

With the notations of \cite[Theorem~1]{Gervais} (except that our "$k$" corresponds to the "$n$" in Gervais' notation), 
let $F$ be the free group in the letters $b,b_1,\dots,b_{g-1}$, $a_1,\dots,a_{2g+k-2}$, $\{c_{i,j}\}$.
Sending these elements to the images of the corresponding Dehn twists in $G_\Theta$ then defines a homomorphism $\Eta\co F\to G_\Theta$. 
Since, with the exception of the "disjointness relations" $\tau_{\beta_i}\tau_{\beta_j}=\tau_{\beta_j}\tau_{\beta_i}$,  
all relations in the presentation given in \cite[Theorem~1]{Gervais} are supported on some subsurface $S\ssm\s_i$, 
for $i=1,\ldots,g-1$, of $S$, the homomorphism $\Eta$ induces a homomorphism $\ol{\Eta}\co\PG(S,\dS)\to G_\Theta$, such
that its restriction to the subgroup $\PG(S,\dS)_{\vec{\s}_i}$ is injective for $i=1,\ldots,g-1$.
In particular, $\ol{\Eta}$ is surjective and so the composition $\ol{\Theta}\circ\ol{\Eta}$ is also surjective. 
Since the mapping class group $\PG(S,\dS)$ is finitely generated and residually finite,
it is Hopfian, so that $\ol{\Theta}\circ\ol{\Eta}$ is an isomorphism. This implies that $\ol{\Theta}$ is an isomorphism as well.
\end{proof} 

\begin{remark}\label{gap}After a careful reading of the paper \cite{Gervais}, I noted a gap in the proof of \cite[Lemma~8]{Gervais} where
the author claims: "Pasting a pair of pants to $\g_{2g+n-3,1}$ allows us to view $\Sigma_{g,n-1}$ as a subsurface of $\Sigma_{g,n}$".
But this makes sense only for $n>1$, which means that the proof of \cite[Theorem~1]{Gervais} is complete 
only for $n\geq 1$. This is the reason why, in Proposition~\ref{amalgamatedproduct}, we assume $k\geq 1$.
\end{remark}

By the congruence subgroup property in genus $\leq 2$ and (ii) of Theorem~\ref{stabilizers}, the closures of the subgroups $\PG(S)_{\vec\s}$
and $\PG(S)_{\vec{\s}_i}$, for $i=1,\ldots,g-1$, in the profinite mapping class group $\hPG(S)$ identify with their respective profinite completions.
Therefore, by the universal property of amalgamated free profinite products (cf.\ \cite[Section~9.2]{RZ}), we have:

\begin{corollary}\label{proamalgamatedproduct}With the notations of Proposition~\ref{amalgamatedproduct}, there is a natural continuous epimorphism:
\[\wh{\Theta}\co\coprod_{\hPG(S,\dS)_{\vec\s}}^{i=1,\ldots,g-1}\hPG(S,\dS)_{\vec{\s}_i}\twoheadrightarrow\hPG(S,\dS),\]
whose kernel is (topologically) normally generated by the commutators $[\tau_{\beta_i},\tau_{\beta_j}]$, for $i,j=1,\ldots,g-1$.
\end{corollary}

\subsection{Proof of Theorem~\ref{GTrepr} for $S=S_{g}^k$, with $k\geq1$ and $g\geq 3$}\label{compactcase}
Let $\s=\{\g_1,\ldots,\g_{g-1}\}$ and $\s'=\{\beta_1,\ldots,\beta_{g-1}\}$ be as in the statement
of Proposition~\ref{amalgamatedproduct} and let $S_\s$ be the compact surface of genus $1$ obtained replacing with boundary circles 
the punctures of  $S\ssm\s$. Let us denote by $\g_i^+$ and $\g_i^-$ the boundary circles which replace the two punctures obtained removing 
$\g_i$ from $S$, for $i=1,\ldots,g-1$. By the genus $1$ case of (v) of Theorem~\ref{GT=Out}, there is a natural isomorphism:
\begin{equation}\label{iso1}
\Out^\I(\hPG(S_\s,\dS_\s))\cong\Sigma_{k+2g-2}\times\GT.
\end{equation}

Let $\s''=\{\alpha_1,\ldots,\alpha_{g-1}\}$ be a set of disjoint separating simple closed curves on $S$, with empty intersection with the curves in $\s$ 
and $\s'$, such that $\alpha_i$ is the only boundary component of a compact genus $1$ subsurface $S_{\alpha_i}$ of $S$ which contains $\g_i$ and 
$\beta_i$, for $i=1,\ldots,g-1$. Let then $S_{\s''}:=S\ssm\cup_{i=1}^{g-1}\ring{S}_{\alpha_i}$ be the compact genus $1$ subsurface of $S_\s$ with boundary 
the union of the curves $\alpha_1,\ldots,\alpha_{g-1}$, $\delta_1,\dots,\delta_k$. By the genus $1$ case of (v) of Theorem~\ref{GT=Out}, 
there is a natural isomorphism:
\begin{equation}\label{iso2}
\Out^\I(\hPG(S_{\s''},\dS_{\s''}))\cong\Sigma_{k+g-1}\times\GT.
\end{equation}


A given $f\in\GT$ then defines an element of $\Out^\I(\hPG(S_{\s''},\dS_{\s''}))$ which extends to an element of $\Out^\I(\hPG(S_\s,\dS_\s))$
and which we also denote by $f$. This element centralizes $\Sigma_{k+g-1}$ in $\Out^\I(\hPG(S_{\s''},\dS_{\s''}))$ and $\Sigma_{k+2g-2}$ in
$\Out^\I(\hPG(S_\s,\dS_\s))$. Therefore, $f$ preserves the procyclic subgroups generated by the Dehn twists  $\tau_{\alpha_i}^\ZZ$, for $i=1,\ldots,g-1$,
and $\tau_{\g_i^+}$, $\tau_{\g_i^-}$, for $i=1,\ldots,g-1$, and, by the genus $\leq 2$ case of (iv) of Theorem~\ref{GT=Out}, acts on these through 
the character $\chi_\l\co\GT\to\ZZ^\ast$. There is then a lift $\breve f\in\Aut^\I(\hPG(S_\s,\dS_\s))$ of $f$ with the same properties.

There is a natural epimorphism $\hPG(S_\s,\dS_\s)\to\hPG(S,\dS)_{\vec\s}$, whose kernel is topologically generated by the set of multitwists 
$\tau_{\g_i^+}\tau_{\g_i^-}^{-1}$, for $i=1,\ldots,g-1$. Since $\breve f$ preserves all the procyclic subgroups
generated by the Dehn twists $\tau_{\g_i^+}$, $\tau_{\g_i^-}$ and act on these through the character $\chi_\l$,
it follows that $\breve f$ descends to to an automorphism $\td f\in\Aut^\I(\hPG(S,\dS)_{\vec\s})$ and so $f$ descends 
to an outer automorphism $\bar f\in\Out(\hPG(S,\dS)_{\vec\s})$ such that $\td f$ lifts $\bar f$ and both preserve
the procyclic subgroups $\tau_{\alpha_i}^\ZZ$, for $i=1,\ldots,g-1$.

As in Proposition~\ref{amalgamatedproduct}, let $\s_i:=\s\ssm\{\g_i\}$, for $i=1,\ldots,g-1$. We have:

\begin{lemma}\label{lift}Let $\td f\in\Aut^\I(\hPG(S,\dS)_{\vec\s})$ be the element defined above.
Then, $\td f$ admits a unique extension $\td f_i\in\Aut^\I(\hPG(S,\dS)_{\vec{\s}_i})$, for all $i=1,\ldots,g-1$.
\end{lemma}

\begin{proof}Let $S_{\s_i}$ be the compact surface of genus $2$ obtained from $S_{\s}$ identifying the boundary circles $\g_i^+$ and $\g_i^-$,
so that $\ring{S}_{\s_i}=S\ssm\s_i$. By the genus $2$ case of item (v) of Theorem~\ref{GT=Out}, there is a natural isomorphism, for $i=1,\ldots,g-1$:
\begin{equation}\label{iso3}
\Out^\I(\hPG(S_{\s_i},\dS_{\s_i}))\cong\Sigma_{k+2g-4}\times\GT.
\end{equation}

There is a natural epimorphism $\hPG(S_\s,\dS_\s)\to\hPG(S_{\s_i},\dS_{\s_i})_{\vec\g_i}$, for $i=1,\ldots,g-1$, whose kernel is
topologically generated by the multitwist $\tau_{\g_i^+}\tau_{\g_i^-}^{-1}$. By the same argument above, this kernel is preserved by the element 
$f\in\GT\subset\Out^\I(\hPG(S_\s,\dS_\s))$. Therefore, $f$ descends to an outer automorphism of $\hPG(S_{\s_i},\dS_{\s_i})_{\vec\g_i}$ which, 
by the isomorphism~(\ref{iso3}), then admits a unique extension $f_i\in\GT\subset\Out^\I(\hPG(S_{\s_i},\dS_{\s_i}))$. 

There is a natural epimorphism $\hPG(S_{\s_i},\dS_{\s_i})\to\hPG(S,\dS)_{\vec{\s}_i}$, for $i=1,\ldots,g-1$,
whose kernel, topologically generated by the multitwists $\tau_{\g_j^+}\tau_{\g_j^-}^{-1}$ for $j\neq i$, by the usual argument, is preserved by $f_i$. 
Hence, $f_i$ descends to $\bar f_i\in\Out(\hPG(S,\dS)_{\vec{\s}_i})$, for $i=1,\ldots,g-1$.

Let $\breve f_i\in\Aut(\hPG(S,\dS)_{\vec{\s}_i})$ be a lift of the element $\bar f_i$ defined above.
The same construction of the homomorphism~(\ref{stabrestriction}) then provides a natural homomorphism:
\[\td{PR}_{{\vec\g}_i}\co\Aut^\I(\hPG(S,\dS)_{\vec{\s}_i})\to\Out^\I(\hPG(S,\dS)_{\vec\s}),\]
such that $\td{PR}_{{\vec\g}_i}(\breve f_i)=\bar f$, for $i=1,\ldots,g-1$. Let $x_i\in\hPG(S,\dS)_{\vec{\s}_i}$ be an element such that 
$\inn x_i\circ\breve f_i$ preserves the inertia group $\hI_{{\g}_i}$ and $x_i$ fixes a chosen orientation on $\g_i$.
As it is shown in the definition of the homomorphism~(\ref{stabrestriction}), with these choices, $\inn x_i\circ\breve f_i$ restricts to an automorphism of 
$\hPG(S,\dS)_{\vec\s}$ which differs from $\td f$ by an inner automorphism $\inn y_i\in\Inn(\hPG(S,\dS)_{\vec\s})$. It follows that
the automorphism $\td f_i:=\inn (y_ix_i)\circ\breve f_i$ of the group $\hPG(S,\dS)_{\vec{\s}_i}$ is an extension of $\td f$, for all $i=1,\ldots,g-1$. 

In order to conclude the proof of Lemma~\ref{lift}, we have to prove that this extension is unique, that is to say that, 
given an automorphism $\td f_i'\in\Aut^\I(\hPG(S,\dS)_{\vec{\s}_i})$ which preserves the subgroup $\hPG(S,\dS)_{\vec\s}$ 
and restricts there to $\td f$, we have $\td f_i'=\td f_i$, for $i=1,\ldots,g-1$.

By Theorem~\ref{relativestabilizers} and Theorem~\ref{normalizers multitwists}, the center $C$ of $\hPG(S,\dS)_{\vec{\s}_i}$ is the direct sum:
\[C=\prod_{j=1}^k\tau_{\delta_j}^\ZZ\;\;\oplus\prod_{l=1,\dots\hat{i}\dots,g-1}\tau_{\g_l}^\ZZ.\]
Therefore, the quotient of $\hPG(S,\dS)_{\vec{\s}_i}$ by its center $C$ is isomorphic to $\hPG(S\ssm\s_i)$ and 
there is a natural homomorphism, for $i=1,\ldots,g-1$: 
\begin{equation}\label{mono}
\Aut^\I(\hPG(S,\dS)_{\vec{\s}_i})\to\Aut^\I(\hPG(S\ssm\s_i)).
\end{equation}

\begin{lemma}\label{monoproof}The homomorphism~(\ref{mono}) is injective.
\end{lemma}

\begin{proof}Let us consider the wreath products $\ZZ^\ast\wr\Sigma_k$ and $\ZZ^\ast\wr\Sigma_{g-2}$ and their respective canonical
primitive actions on $\prod_{j=1}^k\tau_{\delta_j}^\ZZ$ and $\prod_{l=1,\dots\hat{i}\dots,g-1}\tau_{\g_l}^\ZZ$.

The natural action of $\Aut^\I(\hPG(S,\dS)_{\vec{\s}_i})$ on $C$ factors through the above canonical actions and a natural homomorphism:
\[\Aut^\I(\hPG(S,\dS)_{\vec{\s}_i})\to(\ZZ^\ast\wr\Sigma_k)\times(\ZZ^\ast\wr\Sigma_{g-2}).\]
Let us then observe that this homomorphism can be recovered from: 
\begin{itemize}
\item a natural character $\Aut^\I(\hPG(S,\dS)_{\vec{\s}_i})\to\Aut(\tau_{\delta_1}^\ZZ)$,
which, by the above construction and the case $g\leq 2$ of Theorem~\ref{GT=Out}, is determined by the character $\chi_\l\co\GT\to\ZZ^\ast$;
\item a natural character $\Aut^\I(\hPG(S,\dS)_{\vec{\s}_i})\to\Aut(\tau_{\g_1}^\ZZ)$, which is determined by the homomorphism~(\ref{mono}) 
and a natural character $\Aut^\I(\hPG(S\ssm\s_i))\to\Aut(\tau_{\g}^\ZZ)$, where $\g$ is a nonseparating simple closed curve on $S\ssm\s_i$;
\item a natural epimorphism $\Aut^\I(\hPG(S,\dS)_{\vec{\s}_i})\to\Sigma_k\times\Sigma_{g-2}$ which can be recovered from the 
homomorphism~(\ref{mono}) and the epimorphism $\Aut^\I(\hPG(S\ssm\s_i))\to\Sigma_{k+2g-4}$ (cf.\ (ii) of Theorem~\ref{compurevsfull}).
\end{itemize}

The above remarks imply that the action of $\Aut^\I(\hPG(S,\dS)_{\vec{\s}_i})$ on $C$ is determined by the homomorphism~(\ref{mono}).
Since, by \cite[Lemma~7.4]{BF} and the same argument of the proof of Lemma~\ref{injectivehom}, the natural homomorphism:
\[\Aut^\I(\hPG(S,\dS)_{\vec{\s}_i})\to\Aut(C)\times\Aut^\I(\hPG(S\ssm\s_i))\]
is injective, it follows that the homomorphism~(\ref{mono}) is already injective.
\end{proof}

Let us now denote, respectively, by $\ring f_i$ and $\ring f_i'$ the images in $\Aut^\I(\hPG(S\ssm\s_i))$ of the elements $\td f_i$ and $\td f_i'$ 
defined above by the monomorphism~(\ref{mono}). 

By hypothesis, $\td f_i'\cdot\td f_i^{-1}$ restricts to the identity on the subgroup $\hPG(S,\dS)_{\vec\s}$ and so induces the identity automorphism
on the quotient $\hPG(S\ssm\s_i)_{\vec\g_i}$ of this subgroup. From (i) of Lemma~\ref{item1}, it then follows that 
$\ring f_i'\cdot\ring f_i^{-1}\in\Inn(\hG(S\ssm\s_i))$.

By construction, $\ring f_i'\cdot\ring f_i^{-1}$ preserves the inertia group $\hI_{\g_i}$. Hence,
$\ring f_i'\cdot\ring f_i^{-1}\in\Inn(\hG(S\ssm\s_i)_{\g_i})$, where the latter group is identified with a subgroup of 
$\Inn(\hG(S\ssm\s_i))\subset\Aut^\I(\hPG(S\ssm\s_i))$. 

Above, we observed that $\td f_i'\cdot\td f_i^{-1}$ induces the identity on 
$\hPG(S\ssm\s_i)_{\vec\g_i}$ and so $\ring f_i'\cdot\ring f_i^{-1}$ identifies with an inner automorphism of $\hG(S\ssm\s_i)_{\g_i}$ which restricts to 
the identity on $\hPG(S\ssm\s_i)_{\vec\g_i}$. By \cite[Theorem~4.14]{BF}, this implies that $\ring f_i'\cdot\ring f_i^{-1}=1$. Since the 
homomorphism~(\ref{mono}) is injective, it follows that $\td f_i'\cdot\td f_i^{-1}=1$ and then $\td f_i'=\td f_i$.
\end{proof}

The automorphisms $\td f_i$, for $i=1,\ldots,g-1$, then glue to an automorphism:
\[\td F\in\Aut(\coprod_{\hPG(S,\dS)_{\vec\s}}^{i=1,\ldots,g-1}\hPG(S,\dS)_{\vec{\s}_i}).\]

\begin{lemma}\label{kerF}With the notations of Corollary~\ref{proamalgamatedproduct}, there holds $\td F(\ker\wh{\Theta})\subseteq\ker\wh{\Theta}$.
\end{lemma}

\begin{proof}By Corollary~\ref{proamalgamatedproduct}, it is enough to prove that there holds $\td F([\tau_{\beta_i},\tau_{\beta_j}])\in\ker\wh{\Theta}$, 
for $i,j=1,\ldots,g-1$. Letting $\underbar F:=\wh{\Theta}\circ\td F$, this is equivalent to the identity
$\underbar F([\tau_{\beta_i},\tau_{\beta_j}])=[\underbar F(\tau_{\beta_i}),\underbar F(\tau_{\beta_j})]=1$, for $i,j=1,\ldots,g-1$, 
and so to the fact that the element $\underbar F(\tau_{\beta_i})$ commutes with $\underbar F(\tau_{\beta_j})$ in the profinite mapping class group 
$\hPG(S,\dS)$, for $i,j=1,\ldots,g-1$.

Let us recall that $S_{\alpha_i}$ is the compact genus $1$ subsurface of $S$ with boundary $\alpha_i$ which contains the simple closed curve
$\beta_i$, for $i=1,\ldots,g-1$. 

By (i) of Theorem~\ref{stabilizers} and the subgoup congruence property for mapping class groups in genus $\leq 2$,
there are natural monomorphisms $\hPG(S_{\alpha_i},\alpha_i)\hookra\hPG(S,\dS)_{\vec{\s}_i}$. Let us identify $\hPG(S_{\alpha_i},\alpha_i)$ 
with its image in $\hPG(S,\dS)_{\vec{\s}_i}$, for $i=1,\ldots,g-1$. 

We defined the automorphism $\td f_i$ of $\hPG(S,\dS)_{\vec{\s}_i}$ in such a way that it preserves the procyclic subgroup $\tau_{\alpha_i}^\ZZ$.
Thus, $\td f_i$ also preserves the normalizer of $\tau_{\alpha_i}^\ZZ$ in $\hPG(S,\dS)_{\vec{\s}_i}$. By Theorem~\ref{normalizers multitwists}, 
$\td f_i$ then preserves the subgroup $\hPG(S_{\alpha_i},\alpha_i)$ (which contains the element $\tau_{\beta_i}$), for $i=1,\ldots,g-1$.

Since, inside $\hPG(S,\dS)$, for $i\neq j$, all elements of the subgroup $\hPG(S_{\alpha_i},\alpha_i)$ commute with all elements of the subgroup 
$\hPG(S_{\alpha_j},\alpha_j)$, it follows that $\underbar F(\tau_{\beta_i})\in\hPG(S_{\alpha_i},\alpha_i)$ commutes with 
$\underbar F(\tau_{\beta_j})\in\hPG(S_{\alpha_j},\alpha_j)$ for $i,j=1,\ldots,g-1$, which concludes the proof of the lemma.
\end{proof}

By the Hopfian property of (topologically) finitely generated profinite groups, Corollary~\ref{proamalgamatedproduct} and Lemma~\ref{kerF} imply 
that $\td F$ descends to an automorphism $\bar F\in\Aut(\hPG(S,\dS))$.
Moreover, by definition, $\bar F$ preserve the inertia groups $\hI_{\g_i}$, for $i=1,\ldots,g-1$. In particular, we have that
$\bar F\in\Aut^{\I_0}(\hPG(S,\dS))$.

The only choice involved in the definition of $\bar F$ is that of the lift $\breve f\in\Aut^\I(\hPG(S_\s,\dS_\s))$ of the given $f\in\GT$, which
we had identified with an element of $\Out^\I(\hPG(S_\s,\dS_\s))$. Therefore, 
if we let $F$ be the image of $\bar F$ in $\Out^{\I_0}(\hPG(S,\dS))$, the assignment $f\mapsto F$ is well defined. 

Moreover, following through the definition of $\bar F$, we see that, if $\breve f_1$ and $\breve f_2$ are
lifts of elements $f_1,f_2\in\GT$, their product $\breve f_1\cdot\breve f_2$ is a lift of $f_1\cdot f_2\in\GT$ with the property 
that the associated element in $\Out^{\I_0}(\hPG(S,\dS))$ is the product $\bar F_1\cdot\bar F_2$.

Thus, the assignment $f\mapsto F$ determines a homomorphism $\GT\to\Out^{\I_0}(\hPG(S,\dS))$, as claimed in Theorem~\ref{GTrepr}. 
That this representation is faithful, follows by considering its restriction to the subgroup $\hPG(S,\dS)_{\vec\s}$.
The last part of the statement of Theorem~\ref{GTrepr} about the action on the central inertia group follows from the construction above.

\subsection{Proof of Theorem~\ref{GTrepr} for $S=S_{g,n}^k$ and $n+k\neq 0$}\label{generalcase}
This case of Theorem~\ref{GTrepr} follows from the one considered in the previous section.
In fact, the kernel of the natural epimorphism $\hPG(S_g^{k+n},\partial S_g^{k+n})\to\hPG(S_{g,n}^k,\partial S_{g,n}^k)$ is preserved
by the outer action of $\GT$ on $\hPG(S_g^{k+n},\partial S_g^{k+n})$ as defined in the previous Section~\ref{compactcase}. Therefore, we get a
representation $\hat\rho_{\GT}\co\GT\to\Out^{\I_0}(\hPG(S,\partial S))$, for $S=S_{g,n}^k$, with the properties claimed in Theorem~\ref{GTrepr}.

\subsection{Proof of Theorem~\ref{GTrepr} for $S=S_g$}\label{closedcase}
The closed surface case of Theorem~\ref{GTrepr} is instead a consequence of the following considerations, which will be fundamental also for the proof 
of the case $g(S)\geq 3$ of Lemma~\ref{GTreprcong} given in Section~\ref{lemma>2}. 

With the notations of Section~\ref{compactcase}, the automorphism $\bar F$ restricts on the subgroup $\hPG(S,\dS)_{\vec\s}$ of $\hPG(S,\dS)$ 
to the automorphism $\td f\in\Aut^\I(\hPG(S,\dS)_{\vec\s})$. We then have:

\begin{lemma}\label{boundingpair}Let $\g,\g'$ be simple nonseparating curves on $S\ssm\s$ which, 
together with the boundary curve $\d_1$, bound a $3$-holed genus $0$ subsurface of $S$. Then, the automorphism $\td f$ preserves the
conjugacy class in $\hPG(S,\dS)_{\vec\s}$ of the procyclic subgroup generated by $\tau_\g^{-1}\tau_{\g'}$.
\end{lemma}

\begin{proof}Let us observe again that, by the subgroup congruence property in genus $\leq 2$, we have $\hPG(S\ssm\s)=\kPG(S\ssm\s)$ and 
$\hPG(S,\dS)_{\vec\s}=\kPG(S,\dS)_{\vec\s}$.

The group $\Aut^\I(\hPG(S,\dS)_{\vec\s})$ acts on the complex of profinite curves $\kC(S\ssm\s)$ through the natural homomorphism 
$\Aut^\I(\hPG(S,\dS)_{\vec\s})\to\Aut^\I(\hPG(S\ssm\s))$. From \cite[Theorem~5.5 and Proposition~7.2]{BF}, it then follows that
the topological type of the bounding pair $\{\g,\g'\}\in\kC(S\ssm\s)$ is preserved by the action of $\td f$, which means that $\td f$ maps the pair
$\{\g,\g',\}$ to some other bounding pair $\{\td f(\g),\td f(\g')\}$ of the same topological type, that is to say, in the $\Inn(\hG(S,\dS)_{\vec\s})$-orbit 
of $\{\g,\g',\}$. 

Since $\td f$ preserves the procyclic subgroup generated by $\tau_{\d_1}$, from Theorem~\ref{relativestabilizers} and 
Theorem~\ref{normalizers multitwists}, it follows that it also preserves the marked topological type of the bounding pair $\{\g,\g'\}$, 
that is to say, the bounding pair $\{\td f(\g),\td f(\g')\}$ is in the $\Inn(\hPG(S,\dS)_{\vec\s})$-orbit of $\{\g,\g'\}$. 
By the construction of the automorphism $\td f$, we also know 
that it acts on a nonseparating profinite Dehn twist in $\hPG(S,\dS)_{\vec\s}$ by conjugation twisted by the character $\chi_\l\co\GT\to\ZZ^\ast$.
These facts then imply that $\td f$ preserves the conjugacy class in $\hPG(S,\dS)_{\vec\s}$ of the procyclic subgroup generated by 
$\tau_\g^{-1}\tau_{\g'}$. 
\end{proof}

Let $\ol{S}$ be the surface obtained glueing a disc $D$ to the boundary component $\d_1$ of $S$
and let us fix a base point $P\in D$. There is then a profinite Birman exact sequence:
\begin{equation}\label{Birmancompact}
1\to\hp_1(\ol{S},P)\times\tau_{\d_1}^\ZZ\to\hPG(S,\dS)\to\hPG(\ol{S},\partial\ol{S})\to 1,
\end{equation}
where the conjugacy class of $\tau_\g^{-1}\tau_{\g'}$ in $\hPG(\ol{S},\partial\ol{S})$ (topologically) generates the normal subgroup $\hp_1(\ol{S},P)$.
By Lemma~\ref{boundingpair}, the automorphism $\bar F$ preserves the conjugacy class in $\hPG(S,\dS)$ of the procyclic subgroup 
generated by $\tau_\g^{-1}\tau_{\g'}$. Therefore, the representation $\hat{\rho}_{\GT}\co\GT\hookra\Out^{\I_0}(\hPG(S,\dS))$ 
preserves the normal subgroup $\hp_1(\ol{S},P)\times\tau_{\d_1}^\ZZ$ of $\hPG(S,\dS)$ and so induces a representation on the quotient:
\[\hat\rho_{\GT}\co\GT\to\Out^{\I_0}(\hPG(\ol{S},\partial\ol{S})).\]
That this representation is also faithful follows from its compatibility with the restriction of $\hat\rho_{\GT}$ to modular subgroups of 
$\hPG(\ol{S},\partial\ol{S})$. Note that, for $k=1$, the surface $\ol{S}$ is closed and $\partial\ol{S}=\emptyset$, so that, in particular, 
we get the closed surface case of Theorem~\ref{GTrepr}.

\subsection{Proof of Lemma~\ref{GTreprcong} for $g(S)\geq 3$}\label{lemma>2}
The image of the natural representation:
\[\hat{\rho}_{S}\co\hPG(S,\dS)\to\Aut(\hp_1(\ol{S},P)),\]
induced by restriction of inner automorphisms to the normal subgroup $\hp_1(\ol{S},P))$ appearing in the profinite Birman 
exact sequence~(\ref{Birmancompact}), is precisely the procongruence mapping class group $\kPG(S)$ 
(cf.\ \cite[Corollary~4.7]{congtop}). 

Let us then show that $\hat{\rho}_{\GT}\co\GT\hookra\Out^{\I_0}(\hPG(S,\dS))$ induces a representation: 
\begin{equation}\label{checkrepr}
\check{\rho}_{\GT}\co\GT\to\Out^{\I_0}(\kPG(S)).
\end{equation}
This follows from the lemma:

\begin{lemma}\label{congkernel}The outer automorphisms of $\hPG(S,\dS)$ in the image of $\hat{\rho}_{\GT}$ preserve the normal subgroup 
$\ker\hat{\rho}_{S}$ of $\hPG(S,\dS)$.
\end{lemma}

\begin{proof}Let $\td\phi\in\Aut^{\I_0}(\hPG(S,\dS))$ be a lift of an element $\phi\in\Im\hat{\rho}_{\GT}\subset\Out^{\I_0}(\hPG(S,\dS))$.
By Lemma~\ref{boundingpair}, we have that $\td\phi(\hp_1(\ol{S},P))=\hp_1(\ol{S},P)$. 
Let $x\in\ker\hat{\rho}_{S}$, that is to say an element $x\in\hPG(S,\dS)$ such that the restriction of $\inn x$ to $\hp_1(\ol{S},P)$ is trivial. 
Then, also $\td\phi\circ\inn x\circ\td\phi^{-1}$ restricts to the trivial automorphism on $\hp_1(\ol{S},P)$ and,
from the identity $\td\phi\circ\inn x\circ\td\phi^{-1}=\inn(\td\phi(x))$, it follows that $\td\phi(x)\in\ker\hat{\rho}_{S}$.
\end{proof}

Note that the representation $\check{\rho}_{\GT}$~(\ref{checkrepr}) is defined for every hyperbolic nonclosed surface.  
Composing with the natural homomorphism $\Out^{\I}(\kPG(S))\to\Out^{\I}(\kG(\ring{S}))$ 
(cf.\ Theorem~\ref{IvsIns} and Theorem~\ref{compurevsfull}), we then get the natural representation:
\[\Psi_{(\ring{S},\emptyset)}\co\GT\to\Out^{\I}(\kG(\ring{S})).\]

If $\ring{S}=S'\ssm P$, where $S'$ is a closed surface, we can further compose $\Psi_{(\ring{S},\emptyset)}$ with the 
homomorphism~(\ref{Birmaninduced}) $B_P\co\Out^{\I}(\kG(\ring{S}))\to\Out^{\I}(\kG(S'))$, thus obtaining the representation:
\[\Psi_{(S',\emptyset)}\co\GT\to\Out^{\I}(\kG(S')).\]

By construction, these representations are functorial with respect to the restriction to subsurfaces of $S$. 
This proves Lemma~\ref{GTreprcong} for all the cases when $g(S)\geq 3$ and $\dS=\emptyset$.

The representation $\Psi_{(S,\dS)}$, for $g(S)\geq 3$ and $\dS\neq\emptyset$, is then obtained composing
$\Psi_{(\tS,\emptyset)}$, for $\tS=S_{g,n+2k}$ as in Remark~\ref{identify}, with the homomorphism $R_{(S,\dS)}$~(\ref{restrwithboundary}).

\subsection{The automorphism group of the procongruence Grothendieck-Teichm\"uller tower}\label{GGT}
In the landmark paper \cite{Drinfeld}, Drinfeld, following Grothendieck's \cite[\emph{Esquisse d'un Programme}]{Esquisse},
speculated that the profinite Grothendieck-Teichm\"uller group, which is defined taking into account only
the first two levels of the genus $0$ stage of the profinite Grothendieck-Teichm\"uller tower, naturally acts on the full tower and that, moreover, 
this action induces an isomorphism:

\begin{conjecture}[Drinfeld-Grothendieck]\label{Drinfeld}$\GT\cong\Aut(\hat\fT^\mathrm{out})$.
\end{conjecture}

Combining Conjecture~\ref{Drinfeld} with Conjecture~\ref{Grothendieck} yields the Grothendieck-Teichm\"uller conjecture:

\begin{conjecture}\label{GTconj}$\GT\cong G_\Q$.
\end{conjecture}

As an application of Theorem~\ref{GT=Out}, we will now prove the following procongruence version of Conjecture~\ref{Drinfeld}:

\begin{theorem}\label{GT=AutT}There is a natural isomorphism $\GT\cong\Aut(\check\fT^\mathrm{out})$.
\end{theorem}

Of course, assuming the congruence subgroup property, the profinite and procongruence versions of the Drinfeld-Grothendieck conjecture 
are just equivalent. Even if this is still an open problem, the following immediate consequence of Theorem~\ref{GTrepr} shows that 
there are no additional restrictions coming from the profinite Grothendieck-Teichm\"uller tower:

\begin{corollary}\label{GTaction}There is a natural faithful representation $\hat\rho_{\GT}\co\GT\hookra\Aut(\hat\fT^\mathrm{out})$.
In particular, Conjecture~\ref{Grothendieck} implies Conjecture~\ref{GTconj}.
\end{corollary}

\subsection{Proof of Theorem~\ref{GT=AutT}}
Let us observe that there are natural homomorphisms:
\[\hat{\tr}_{S}\co\Aut(\hat\fT^\mathrm{out})\to\Out(\hPG(S,\dS))\hspace{0.5cm}\mbox{and}\hspace{0.5cm}
\check{\tr}_{S}\co\Aut(\check\fT^\mathrm{out})\to\Out(\kPG(S,\dS)),\]
defined restricting an automorphism of the full tower to one of its components.

\begin{lemma}\label{trace}There holds $\Im\hat{\tr}_{S}\subseteq\Out^\I(\kPG(S,\dS))$ and $\Im\check{\tr}_{S}\subseteq\Out^\I(\kPG(S,\dS))$.
\end{lemma}

\begin{proof}The proof is exactly the same for the profinite and the procongruence case. Let us then prove the lemma 
for the procongruence Grothendieck-Teichm\"uller tower which is the case we need for the proof of Theorem~\ref{GT=AutT}.

Let us show first that, for $\dS=\cup_{i=1}^k\d_i\neq\emptyset$, the image of $\Aut(\check\fT^\mathrm{out})$ in the group $\Out^\I(\kPG(S,\dS))$
preserve each procyclic subgroup $\tau_{\d_i}^\ZZ$, for $i=1,\ldots,k$. Let $\td{S}_i$ be the surface obtained attaching a $1$-punctured disc to the boundary
component $\d_i$ of $S$. The map $S\to\td{S}_i$ then induces, in the procongruence Grothendieck-Teichm\"uller tower, the homomorphism of procongruence 
pure relative mapping class groups $\kPG(S,\dS)\to\kPG(\td{S}_i,\dd\td{S}_i)$. By definition,
an element of $\Aut(\check\fT^\mathrm{out})$ is compatible with this homomorphism and so preserves its kernel which is precisely the procyclic subgroup
$\tau_{\d_i}^\ZZ$, for $i=1,\ldots,k$.

In order to show that the image of $\Aut(\check\fT^\mathrm{out})$ in $\Out^\I(\kPG(S,\dS))$ preserves the conjugacy class of the procyclic subgroup
generated by the Dehn twist $\tau_\g$, where $\g$ is any nonperipheral simple closed curve on $S$, it is enough to observe that there is an
embedding $S'\hookra S$ in $\cS$ such that $\g$ is in the image of a boundary component of $S'$.
From the previous case, it then follows that an automorphism of $\Aut(\check\fT^\mathrm{out})$ is compatible with the associated homomorphism 
$\kPG(S',\dS')\to\kPG(S,\dS)$ only if its image in $\Out(\kPG(S,\dS))$ preserves the conjugacy class of $\tau_\g^\ZZ$.
\end{proof}

\begin{remark}\label{misunderstanding}Lemma~\ref{trace} shows, in particular, that the inertia condition of Grothendieck-Teichm\"uller
theory is intrinsic to the whole Grothendieck-Teichm\"uller tower, even though it may not be intrinsic to the single 
components (the conjecture, however, is that, for $d(S)>1$ and $\kPG(S,\dS)$ center free, it is).   
\end{remark}

From Theorem~\ref{GT=Out}, it follows that there is a natural monomorphism $\GT\hookra\Aut(\check\fT^\mathrm{out})$. 
In order to prove that $\Aut(\check\fT^\mathrm{out})\cong\GT$, it is then enough to determine the image of $\check{\tr}_{S}$, 
for $S=S_g^k$ and $d(S)>2$. 

By (v) of Theorem~\ref{GT=Out}, for $S=S_g^k$ and $d(S)>2$, we have $\Out^\I(\kPG(S,\dS))\cong\Sigma_k\times\GT$ so that,
by  Lemma~\ref{trace}, to prove Theorem~\ref{GT=AutT}, it is enough to show that $\Im\check{\tr}_{S}$ 
has trivial projection to $\Sigma_k$, for $S\neq S_{1,2}$.
This claim follows considering the natural maps $S\to \ol{S}_i$, where $\ol{S}_i$ is the surface obtained glueing a disc $D$ 
to the boundary component $\d_i$ of $S=S_g^k$, for $i=1,\ldots,k$. 
In fact, an element of $\Out^\I(\kPG(S,\dS))$, which projects to a nontrivial element of the symmetric group 
$\Sigma_k\cong\Inn(\kG(S,\dS))/\Inn(\kPG(S,\dS))$, does not preserve the kernel $\hp_1(\ol{S}_i,P)\times\tau_{\delta_i}^\ZZ$ 
of the homomorphism $\kPG(S,\dS)\to \kPG(\ol{S}_i,\partial\ol{S}_i)$, for $i=1,\ldots,k$ (cf.\ (\ref{Birmancompact})).

\section{The automorphism group of the arithmetic procongruence mapping class group}\label{arithmeticsec}
In this section, we assume that $S$ is a hyperbolic surface without boundary.

\subsection{The arithmetic profinite mapping class group}
By Grothendieck theory of the \'etale fundamental group, for a geometric base point $\ol\xi$
of the moduli stack $\cM(S)_\Q$, the structural morphism $\cM(S)_\Q\to\Spec(\Q)$ induces the short exact sequence of \'etale fundamental groups:
\begin{equation}\label{fundamentalseq}
1\to\pi_1^\mathrm{et}(\cM(S)_{\ol{\Q}},\ol{\xi})\to\pi_1^\mathrm{et}(\cM(S)_\Q,\ol{\xi})\to G_\Q\to 1.
\end{equation}
The left hand term $\pi_1^\mathrm{et}(\cM(S)_{\ol{\Q}},\ol{\xi})$ of this short exact sequence is the \emph{geometric \'etale fundamental group of} 
$\cM(S)_\Q$. It is naturally isomorphic to the profinite completion of the topological fundamental group, with base point the image of $\ol\xi_\C$, 
of the complex analytic stack associated to the complex DM stack $\cM(S)_{\C}$. 

Thus, $\pi_1^\mathrm{et}(\cM(S)_{\ol{\Q}},\ol{\xi})$ can be identified with the profinite completion $\hG(S)$ of the 
mapping class group $\G(S)$ associated to the surface $S$. To be more precise, this identification is determined
by the isotopy class of a diffeomorphism $S\to(\cC(S)_{\ol{\xi}}\times\C)^\mathrm{an}$ compatible with the orientations.

A (tangential) rational point on $\cM(S)_\Q$ determines a splitting of the short exact sequence~(\ref{fundamentalseq}). There is
then a (faithful) representation $\hat{\rho}_\Q\co G_\Q\to\Aut(\hG(S))$ which induces an isomorphism of profinite groups:
\[\pi_1^\mathrm{et}(\cM(S)_\Q,\ol{\xi})\cong\hG(S)\rtimes_{\hat{\rho}_\Q} G_\Q.\]

It is natural to denote by $\hG(S)_\Q$ the profinite group $\pi_1^\mathrm{et}(\cM(S)_\Q,\ol{\xi})$ which we will call the 
\emph{arithmetic profinite mapping class group}. Let $\mathrm{P}\hG(S)$ be the profinite completion of the pure mapping class group $\mathrm{P}\G(S)$ 
associated to the surface $S$. After the identification of $\pi_1^\mathrm{et}(\mathrm{P}\cM(S)_{\ol{\Q}},\ol{\xi})$ with $\mathrm{P}\hG(S)$, 
we will also put $\mathrm{P}\hG(S)_\Q:=\pi_1^\mathrm{et}(\mathrm{P}\cM(S)_\Q,\ol{\xi})$.

\subsection{The arithmetic procongruence mapping class group}\label{arithmetic}
Let $\cC(S)\to\cM(S)$ be the universal punctured curve over the moduli stack $\cM(S)$ and $\mathrm{P}\cC(S)\to \mathrm{P}\cM(S)$ be the universal
curve over the moduli stack $\mathrm{P}\cM(S)$. 

Let $\cC(S)_{\ol\xi}$ be the fiber over the geometric point $\ol\xi\in\cM(S)_\Q$ of the universal curve $\cC(S)\to\cM(S)$ and 
$\tilde{\xi}\in\cC(S)_{\ol\xi}$ a geometric point. There is then a short exact sequence of algebraic fundamental groups:
\[1\to\pi_1^\mathrm{et}(\cC(S)_{\ol\xi},\tilde{\xi})\to\pi_1^\mathrm{et}(\cC(S)_\Q,\tilde{\xi})\to\pi_1^\mathrm{et}(\cM(S)_\Q,\ol{\xi})\to 1.\]
The associated outer representation
\[{\rho}^\mathrm{et}_S\co\pi_1^\mathrm{et}(\cM(S)_\Q,\ol{\xi})\to\Out(\pi_1^\mathrm{et}(\cC(S)_{\ol\xi},\tilde{\xi}))\]
is called the {\it universal \'etale monodromy representation}. 

After identifying $\pi_1^\mathrm{et}(\cM(S)_{\ol{\Q}},\ol{\xi})$ with $\hG(S)$ and $\pi_1^\mathrm{et}(\cM(S)_\Q,\ol{\xi})$ with $\hG(S)_\Q$, we then let 
\[\kG(S):={\rho}^\mathrm{et}_S(\hG(S))\hspace{1cm}\mbox{ and }\hspace{1cm}\kG(S)_\Q:={\rho}^\mathrm{et}_S(\hG(S)_\Q)\] 
and call them, respectively, the \emph{procongruence} and the \emph{arithmetic procongruence} mapping class groups.
The \emph{pure} procongruence and arithmetic procongruence mapping class groups are then defined by:
\[\kPG(S):={\rho}^\mathrm{et}_S(\mathrm{P}\hG(S))\hspace{1cm}\mbox{ and }\hspace{1cm}\kPG(S)_\Q:={\rho}^\mathrm{et}_S(\mathrm{P}\hG(S)_\Q).\] 

Hoshi and Mochizuki in \cite{HM} (see also \cite[Corollary~7.10]{[B3]}) showed that the kernel of ${\rho}^\mathrm{et}_S$ identifies with the 
congruence kernel, i.e.\ the kernel of the natural epimorphism $\hG(S)\to\kG(S)$. Therefore, the short exact sequence~(\ref{fundamentalseq}) 
gives rise to a short exact sequence:
\begin{equation}\label{fundamentalseqcong}
1\to\kG(S)\to\kG(S)_\Q\to G_\Q\to 1,
\end{equation}
such that the associated outer representation $\rho_\Q\co G_\Q\to\Out(\kG(S))$ is faithful.


The splitting of the short exact sequence~(\ref{fundamentalseq}) determined by a tangential base point also determines a splitting of
the short exact sequence~(\ref{fundamentalseqcong}). There is then a (faithful) representation $\check{\rho}_\Q\co G_\Q\to\Aut(\kG(S))$ 
which induces an isomorphism of profinite groups:
\begin{equation}\label{splittingprocong}
\kG(S)_\Q\cong\kG(S)\rtimes_{\check{\rho}_\Q} G_\Q.
\end{equation}

\subsection{The automorphism group of the arithmetic procongruence mapping class group}
Let us define inertia and decomposition groups for the arithmetic procongruence mapping class group $\kG(S)_\Q$:

\begin{definition}\label{decgroups}Let $U$ be an open subgroup of $\kG(S)_\Q$. 
For $\s=\{\g_0,\ldots,\g_k\}\in\kC(S)$, the \emph{inertia group} $\hI_\s(U)$ is the closed abelian subgroup of $U$ topologically generated 
by the powers of the profinite Dehn twists parameterized by $\s$ and contained in $U$. The \emph{decomposition group} $D_\s(U)$ 
is then the normalizer of $\hI_\s(U)$ in $U$. 
\end{definition}

\begin{remark}Since all closed strata of $\ccM(S)$ are defined over $\Q$ and contain $\Q$-rational points, for every $\s\in\kC(S)$, there is a splitting of the 
short exact sequence~(\ref{fundamentalseqcong}) such that the action $\tilde{\rho}_\Q\co G_\Q\to\Aut(\kG(S))$ preserves the stabilizer $\kG(S)_\s$ 
and so the associated isomorphism~(\ref{splittingprocong}) induces an isomorphism $D_\s(\kG(S)_\Q)\cong\kG(S)_\s\rtimes_{\tilde{\rho}_\Q} G_\Q$.
\end{remark}

\begin{definition}\label{decpreserving}For $U$ an open subgroup of $\kG(S)_\Q$, we define:
\begin{enumerate}
\item $\Aut^\D(U)$ is the closed subgroup of $\Aut(U)$ consisting of those elements which preserve the set of decomposition groups 
$\{D_\g(U)\}_{\g\in\hL(S)_0}$.
\item $\Aut^\I(U)$ is the closed subgroup of $\Aut(U)$ consisting of those elements which preserve the set of inertia groups $\{\hI_\g(U)\}_{\g\in\hL(S)_0}$.
\item $\Aut^{\I_0}(U)$ is the closed subgroup of $\Aut^\I(U)$ consisting of those elements which preserve the set of inertia groups 
$\{\hI_\g(U)\}_{\g\in\hL^\mathrm{ns}(S)}$.
\end{enumerate}
\end{definition}

\begin{remark}A result similar to (i) of Proposition~\ref{0simplices} holds for all the groups of automorphisms defined above.
\end{remark}

With the above notation, \cite[Theorem~9.16]{BF} states, in particular, that, for $U$ an open normal subgroup of $\kG(S)_\Q$ 
and $S\neq S_{1,2},S_{2}$, we have:
\begin{equation}\label{absoluteanabelian}
\Aut^\D(U)=\Inn(\kG(S)_\Q)\cong\kG(S)_\Q.
\end{equation}

Thanks to the results of Section~\ref{nonseparatinginertia}, we then get the following generalization of \cite[Corollary~C]{IN}:

\begin{corollary}\label{absoluteanabelian2}For $d(S)>1$, we have:
\[\Aut^{\I_0}(\kPG(S)_\Q)\cong\Aut^{\I_0}(\kG(S)_\Q)=\Inn(\kG(S)_\Q).\]
\end{corollary}

\begin{proof}For $S\neq S_{1,2}$ and $S_2$, the centralizer of $\kPG(S)_\Q$ in $\kG(S)_\Q$ is trivial. Therefore, since $\kPG(S)_\Q$ is an 
$\I_0$-characteristic subgroup of $\kG(S)_\Q$, by \cite[Lemma~3.3]{BF}, Theorem~\ref{IvsIns}, item (i) of Proposition~\ref{comparison} 
and the identity~(\ref{absoluteanabelian}), for $S\neq S_{1,2}$ and $S_2$, there is a series of natural monomorphisms 
\[\Inn(\kG(S)_\Q)\hookra\Aut^{\I_0}(\kG(S)_\Q)\hookra\Aut^{\I_0}(\kPG(S)_\Q)\hookra\Aut^{\D}(\kPG(S)_\Q)=\Inn(\kG(S)_\Q)\]
and the conclusion follows.

The cases $S= S_{1,2}$ and $S_2$ can be reduced to the above observing that $\Aut^{\I_0}(\kG(S)_\Q)=\Aut^{\I_0}(\kG(S)_\Q/Z(\kG(S)_\Q))$
(cf.\ the proofs of (ii) of Proposition~\ref{g=1n=2} and of Proposition~\ref{outhyp}).
\end{proof}

\subsection{The extended mapping class group}
The \emph{extended mapping class group} $\G^\pm(S)$ is the group of isotopy classes of all diffeomorphisms of $S$.
If $\chi_\R\co\G^\pm(S)\to\{\pm 1\}$ denotes the orientation character, there is a short exact sequence:
\begin{equation}\label{extendedmcg}
1\to\G(S)\to\G^\pm(S)\stackrel{\chi_\R}{\to}\{\pm 1\}\to 1.
\end{equation}
Note that the choice of any antiholomorphic involution on $S$ splits this sequence.

The moduli stack $\cM(S)_\R$ is a real model of the complex moduli stack $\cM(S)_\C$. Hence, the latter carries an antiholomorphic involution 
$\iota\co\cM(S)_\C\to\cM(S)_\C$ associated to such a real model. There is a notion of equivariant fundamental group associated 
to $(\cM(S)_\C,\iota)$ (cf.\ \cite[Section~3]{Huisman}) and, for a base point $[C]\in\cM(S)_\R$, we let 
\[\pi_1^\R(\cM(S)_\R,[C]):=\pi_1^\mathrm{equiv}((\cM(S)_\C/\langle\iota\rangle),[C]).\] 

The equivariant fundamental group $\pi_1^\R(\cM(S)_\R,[C])$ then identifies with the extended mapping class group $\G^\pm(S)$ and the short
exact sequence~(\ref{extendedmcg}) identifies with the associated short exact sequence (cf.\ \cite[Proposition~3.2]{Huisman}).
Let us also observe that, in this setting, the orientation group $\{\pm 1\}$ identifies with the absolute Galois group of the reals $G_\R$.

\subsection{Automorphisms of extended mapping class group}
Let us denote respectively by $\Aut^{\I_0}(\G(S))$ and $\Aut^{\I_0}(\G^\pm(S))$) the group of automorphisms of $\G(S)$ and $\G^\pm(S)$
which preserve the conjugacy class of a cyclic subgroup generated by a Dehn twist about a nonseparating simple closed curve.

A classical result by Ivanov and McCarthy (cf.\ \cite[Theorem~2]{[I2]} and \cite[Theorem~1]{[McC]}) 
asserts that, for $d(S)>1$ and $S\neq S_{1,2},S_2$ (or, equivalently, if $Z(\G(S))=\{1\}$), there holds:
\[\Aut(\G(S))=\Aut^{\I_0}(\G(S))\cong\Aut^{\I_0}(\G^\pm(S))=\Aut(\G^\pm(S))=\Inn(\G^\pm(S)).\] 
In particular, in this case, there is also a natural isomorphism $\Out(\G(S))\cong G_\R$.

\subsection{Comparing automorphism groups}\label{comparing}
As observed in the proof of \cite[Theorem~9.16]{BF}, $\kG(S)$ is a characteristic subgroup of $\kG(S)_\Q$. From
\cite[Proposition~3.9]{[B3]} and \cite[Lemma~3.3]{BF}, it then follows that, for $d(S)>1$, the natural homomorphism 
$\Aut^{\I_0}(\kG(S)_\Q)\to\Aut^{\I_0}(\kG(S))$ is injective. 
A natural question is whether, in analogy with the extended mapping class group case, this is actually an isomorphism:

\begin{question}\label{mainconjecture1}For $d(S)>1$, is there a natural isomorphism:
\[\Aut^{\I_0}(\kG(S)_\Q)\cong\Aut^{\I_0}(\kG(S))?\]
Otherwise stated, does every $\I_0$-automorphism of $\kG(S)$ extend to $\kG(S)_\Q$?
\end{question}

By Corollary~\ref{absoluteanabelian2}, we then have:

\begin{corollary}\label{outerautogroup}A positive answer to Question~\ref{mainconjecture1} implies that, 
for $d(S)>1$, there is a natural isomorphism
\[\Out^{\I_0}(\kG(S))\cong\GT\cong G_\Q.\]
\end{corollary}

\begin{proof}The conclusion follows from the series of isomorphisms: 
\[\Aut^{\I_0}(\kG(S))\cong\Aut^{\I_0}(\kG(S)_\Q)\cong\Inn(\kG(S)_\Q),\]
dividing out by the normal subgroup $\Inn(\kG(S))$ and applying the first item of Theorem~\ref{GT=Out}.
\end{proof}

\subsection{Automorphisms of the procongruence curve complex}
The automorphism group $\Aut^{\I_0}(\kG(S)_\Q)$ acts on the procongruence curve complex $\kC(S)$ through the natural homomorphism 
$\Aut^{\I_0}(\kG(S)_\Q)\to\Aut^{\I_0}(\kG(S))=\Aut^{\I}(\kG(S))$ considered in Section~\ref{comparing}.

\begin{proposition}\label{faithfulrepr}For $d(S)>1$, there is a natural faithful representation: 
\[\psi\co\Aut^{\I_0}(\kG(S)_\Q)\hookra\Aut(\kC(S)).\]
\end{proposition}

\begin{proof}As observed in Section~\ref{comparing}, the homomorphism $\Aut^{\I_0}(\kG(S)_\Q)\to\Aut^{\I_0}(\kG(S))$ is injective.
The conclusion then follows from \cite[Theorem~7.3, (ii)]{BF}.
\end{proof}

Ivanov's fundamental theorem on automorphisms of curve complexes can be formulated, in our terminology, saying that, for $d(S)>1$ 
and $S\neq S_{1,2},S_2$, there is a natural isomorphism (cf.\ \cite[Theorem~1]{[I2]}): 
\[\Aut(\G^\pm(S))=\Aut^{\I_0}(\G^\pm(S))=\Aut^\I(\G^\pm(S))\cong\Aut(C(S)).\] 
A similar question can then be asked for the arithmetic procongruence mapping class group:

\begin{question}\label{mainconjecture2}For $d(S)>1$, is there a natural isomorphism:
\[\Aut^{\I_0}(\kG(S)_\Q)\cong\Aut(\kC(S))?\]
\end{question}

By \cite[Theorem~7.3, (ii)]{BF} and Proposition~\ref{faithfulrepr}, with the above hypotheses, there is a series of monomorphisms:
\[\Aut^{\I_0}(\kG(S)_\Q)\hookra\Aut^{\I_0}(\kG(S))\hookra\Aut(\kC(S)).\]
Therefore, a positive answer to Question~\ref{mainconjecture2} implies a positive answer to Question~\ref{mainconjecture1}.


\begin{thebibliography}{99}





\bibitem{Asada}M. Asada. \textsl{The faithfulness of the monodromy representations associated
with certain families of algebraic curves.} J. Pure Appl. Algebra {\bf 159}, (2001), 123--147. 





\bibitem{hyp} M. Boggi. \textsl{The congruence subgroup property for the hyperelliptic 
Teichm\"uller modular group: the open surface case}. Hiroshima Math. J. {\bf 39} (2009), 351--362.  

\bibitem{[B3]} M. Boggi. \textsl{On the procongruence completion of the Teichm\"uller modular groups.}
Trans. Amer. Math. Soc. {\bf 366} (2014), 5185--5221. 



\bibitem{[BZ2]} M. Boggi, P. Zalesskii. \textsl{A restricted Magnus property for profinite surface groups.}
Trans. Amer. Math. Soc. {\bf 371} (2019), 729--753.

\bibitem{congtop}M. Boggi. \textsl{Congruence topologies on the mapping class groups}. J. Algebra {\bf 546} (2020), 518--552. 


\bibitem{HyperNotes}M. Boggi. \textsl{Notes on hyperelliptic mapping class groups.} Glasgow Math. J. {\bf 66} (2024), 126--161.

\bibitem{BF}M.~Boggi. L.~Funar.  \textsl{Automorphisms of procongruence curve and pants complexes}. J. Topol. {\bf 16} (2023), 936--989.

\bibitem{BF2} M.~Boggi. L.~Funar.  \textsl{Automorphisms of procongruence pants complexes}. \url{http://arxiv.org/abs/2311.17025v3} (2026).


\bibitem{Drinfeld} V.G. Drinfeld. \textsl{On quasitriangular quasi-Hopf algebras and a group closely connected with $\mathrm{Gal}(\ol{\Q}/\Q)$.}
Leningrad Math. J. Vol. {\bf 2}, No. 4 (1991), 829--860.


\bibitem{FM}B. Farb and D. Margalit. \textsl{A primer on mapping class groups}. Princeton Math. Series, Vol. {\bf 49} (2012). 


\bibitem{Gervais} S. Gervais. \textsl{A finite presentation of the mapping class group of a punctured surface.}
Topology {\bf 40} (2001), 703--725.

\bibitem{Esquisse}A. Grothendieck. \textsl{Esquisse d'un Programme.} In {\it Geometric Galois Actions I}, 
edited by L. Schneps and P. Lochak, London Math. Soc. LNS {\bf 242} (1997), 5--48.





\bibitem{HS} D. Harbater, L. Schneps. \textsl{Fundamental groups of moduli and the Grothendieck-Teichm\"uller group.} 
Trans. Amer. Math. Soc. {\bf 352} (2000), 3117--3148.


\bibitem{HLS} A. Hatcher, P. Lochak, L. Schneps. \textsl{On the Teichm{\"u}ller tower 
of mapping class groups.} J. Reine Angew. Math. {\bf 521} (2000), 1--24.

\bibitem{HM}Y.\ Hoshi, S.\ Mochizuki. \textsl{On the combinatorial anabelian geometry
of nodally nondegenerate outer representations}. Hiroshima Math. J.
Volume {\bf 41}, Number 3 (2011), 275--342.

\bibitem{HMM}Y.\ Hoshi, A. Minamide, S.\ Mochizuki. \textsl{Group-theoreticity of numerical invariants
and distinguished subgroups of configuration space groups}. Kodai Math. J. {\bf 45} (2022), 295--348.


\bibitem{Huisman} J. Huisman. \textsl{The equivariant fundamental group, uniformization of real algebraic curves, and global complex analytic
coordinates on Teichm\"uller spaces.} Ann. Fac. Sci. Toulouse Math. $6^e$ s\'erie, tome {\bf 10}, $n^o 4$ (2001), 659--682.

\bibitem{IN} Y. Ihara, H. Nakamura. \textsl{Some illustrative examples for anabelian geometry in high dimensions.} 
In {\it Geometric Galois Actions I}, edited by L. Schneps and P. Lochak, Lon. Math. Soc. LNS {\bf 242} (1997), 127-138.

\bibitem{Ihara}Y. Ihara. \textsl{Some details on the $\GT$-action on $\widehat{B}_n$}. Appendix to: Y. Ihara and
M. Matsumoto. \textsl{On Galois Actions on Profinite Completions of Braid Groups}. In
{\it Recent Developments in the Inverse Galois Problem}. Contemp. Math. {\bf 186}, AMS (1995), 173--200.


\bibitem{[I2]} N.V. Ivanov. \textsl{Automorphisms of complexes of curves and of Teichm{\"u}ller spaces.}
Int. Math. Res. Not. IMRN {\bf 14} (1997), 651--666.





\bibitem{LNS} P. Lochak, H. Nakamura, L. Schneps. \textsl{On a new version of the Grothendieck-Teichm{\"u}ller group.} 
C. R. Acad. Sci. Paris, t. {\bf 325} (1997), 11--16.

\bibitem{LNS2} P. Lochak, H. Nakamura, L. Schneps. \textsl{The Grothendieck-Teichm\"uller group $\GT$ acts on the genus 
$g$ mapping class group with $0$ or $1$ marked point.} \url{http://arxiv.org/abs/2602.12462v1} (2026).








\bibitem{[McC]} J.D. McCarthy. \textsl{Automorphisms of surface mapping class groups:
A recent theorem of N.Ivanov.} Invent. Math. {\bf 84} (1986), 49--71.


\bibitem{MN}A. Minamide, H. Nakamura. \textsl{The automorphism groups of the profinite braid groups.}
Amer. J. Math. {\bf 144} (5) (2022), 1159--1176.





\bibitem{Pop2} F. Pop. \textsl{Finite tripod variants of I/OM}. Invent. Math. {\bf 216} (2019), 745--797.



\bibitem{RZ}L. Ribes, P. Zalesskii. \textsl{Profinite groups. Second edition}. Erg. der Math.
und ihrer Grenz. 3. Folge {\bf 40}, Springer-Verlag (2010).








\bibitem{Segal} I.E. Segal. \textsl{The automorphisms of the symmetric group}. Bull. Amer. Math. Soc. {\bf 46} (6) (1940), 565--565.


\bibitem{Vaintrob}D. Vaintrob.  \textsl{Moduli of framed formal curves.} 
\url{https://arxiv.org/abs/1910.11550} (2019).

\bibitem{Wells}C. Wells.  \textsl{Automorphisms of group extensions.} Trans. Amer. Math. Soc. {\bf 155}, no. 1 (1971), 189--194. 
\end{thebibliography}
\end{document}